\documentclass[12pt,a4paper]{article}
\usepackage[UTF8]{ctex} % 中文字体支持
\usepackage{mathrsfs}
\usepackage{amssymb}
\usepackage{graphicx}
\usepackage{amsmath}
\let\oldsum\sum   % 备份默认的 \sum
\let\oldint\int   % 备份默认的 \int
\renewcommand{\sum}{\oldsum\limits}  % 强制所有 \sum 下标在正下方
\renewcommand{\int}{\oldint\limits}  % 强制所有 \int 下标在正下方
\usepackage[usenames]{color}
\usepackage{extarrows}%箭头宏包
\usepackage{titlesec}
\titleformat{\section}{\centering\Large\bfseries}{\thesection}{1em}{}
\numberwithin{equation}{section}%???????
\usepackage{tikz}
\usetikzlibrary{decorations.pathreplacing} % 大括号库
\usepackage{caption} % 提供更好的标题控制
\usepackage{float} 
\usepackage{enumitem}
\usepackage{cases} % 提供 numcases 和 subnumcases 环境
\usepackage{bm} % 在导言区加载

\newtheorem{theorem}{Theorem}[section]
\newtheorem{definition}[theorem]{Definition}
\newtheorem{prop}[theorem]{Proposition}
\newtheorem{coro}[theorem]{Corollary}
\newtheorem{lemma}[theorem]{Lemma}
\newtheorem{remark}{Remark}
\newtheorem{example}[theorem]{Example}
\newenvironment{pf}{\vspace{3mm}\indent{\bf  Proof. }}{\hfill $\Box$ \vspace{3mm}}
\usepackage{indentfirst}%章节后第一段不缩进，使用这个宏包
\newcommand{\bigast}{%大写的*
	\mathop{\scalebox{1.3}{\raisebox{-0.2ex}{$\ast$}}}%
}

\begin{document}	
\allowdisplaybreaks[4]
\title{Spectrality of Moran-Type Measures Generated by Two-Element and Three-Element Digit Sets
	\thanks{The research is supported by NNSF of China (no. 11971109)}}

\author{Yong-Shen Cao, Qi-Rong Deng, Ming-Tian Li}

\maketitle

\begin{abstract}
The necessary and sufficient conditions for the Moran-type measures generated by two-element and three-element digit sets to be spectral have been given. The main result indicates that the spectrality of such a Moran-type measure is completely determined by the number of 2-factors and 3-factors in the zero set of the Fourier transform of each term in the convolution. The proof shows that, without the existence of Hadamard triples, the problem becomes much more complicated and difficult than the case with the existence of Hadamard triples.

\medskip

\end{abstract}	
	
\tableofcontents

\newpage

\section{Introduction}

A Borel probability measure $\mu$ on $\mathbb{R}^{d}$ is called a spectral measure if there exists a countable set $\Lambda\subset \mathbb{R}^{d}$ such that the family of exponential functions
\begin{equation*}
	E(\Lambda)=\{e^{2\pi i\langle\lambda,x\rangle}:\lambda\in\Lambda\}
\end{equation*}
forms an orthonormal basis for $L^{2}(\mu)$. Such a set $\Lambda$ is called a spectrum of $\mu$. Spectral measures arise naturally in harmonic analysis, fractal geometry, wavelet theory and signal analysis, and their characterization has become one of the central topics in the study of Fourier analysis on fractals.

The study of spectral measures goes back to Fuglede \cite{Fuglede1974}, who conjectured that a measurable set of positive Lebesgue measure is spectral if and only if it tiles the Euclidean space by translations. In the same paper, he established the conjecture in the lattice setting by proving that a measurable set is spectral with respect to a lattice spectrum precisely when it tiles $\mathbb{R}^{d}$ by the dual lattice. Tao \cite{Tao2004} and, Kolountzakis and Matolcsi \cite{KolountzakisMatolcsi2006-1,KolountzakisMatolcsi2006-2} showed that Fuglede's conjecture fails in sufficiently high dimensions. Although Fuglede's conjecture is false in general, it has stimulated extensive research on the spectrality of fractal measures, particularly self-affine measures.

Let $p\geqslant 2$ be an integer and let $C,\ D \subset \mathbb{Z}$ be two finite sets with the same cardinality.
If
\begin{equation*}
	\frac{1}{\sqrt{\#D}}
	\left[e^{2\pi i dc/p}\right]_{d\in D,\,c\in C}
\end{equation*}
is a unitary matrix, then the triple $(p,\,D,\,C)$ is called a Hadamard triple.

The spectrality of self-similar measures has been extensively investigated during the past three decades. Jorgensen and Pedersen \cite{Jorgensen1998} proved that the middle-fourth Cantor measure is a spectral measure, providing the first example of a singular spectral measure. Łaba and Wang \cite{Laba2002} proved that, if  the triple
$(p,\,D,\,C)$ is a Hadamard triple, then the self-similar measure $\mu_{p,\,D}$ is a spectral measure. This result was later generalized to arbitrary dimensions by Dutkay, Haussermann and Lai \cite{Dutkay2018}, who proved that every self-affine measure generated by a Hadamard triple is spectral. The spectrality of Bernoulli convolutions and their generalizations has also been extensively investigated. Dai, He and Lai \cite{Dai2013}, Dai, He and Lau \cite{Dai2014} established a complete characterization of self-similar measures generated by consecutive digit sets, showing that such measures are spectral if and only if the reciprocal of the contraction ratio is an integer multiple of the cardinality of the digit set. There have been extensive studies on the spectrality of self-affine measures, see, for example, \cite{An2015,Cao2022,Dai2016,Dutkay2009,Fu2017,Fu2020}.

\medskip

Compared with self-affine measures, the study of Moran-type measures is relatively recent. Since the contraction ratios and digit sets may vary from level to level, Moran-type measures exhibit richer geometric and arithmetic structures, making the characterization of their spectrality considerably more challenging.

Let $\mathbb{P}:=\{p_n\}_{n=1}^{\infty}$ be a sequence of integers and
$\mathbb{D}:=\{D_n\}_{n=1}^{\infty}$ be a sequence of finite subsets of
$\mathbb{Z}$ satisfying $|p_n|\geqslant2$ for all $n\geqslant1$ and
\begin{equation}\label{eq1.1-1}
	\sup\Big\{\sum_{n=1}^{\infty}\Big|p_1^{-1}p_2^{-1}\cdots p_n^{-1}a_n\Big|:\,a_n\in D_n,\,n\geqslant1\Big\}<\infty.
\end{equation}
By \cite{Strichartz2006}, the sequence of probability measures associated with
\begin{equation*}
	\mu_n:=\delta_{p_1^{-1}D_1}\bigast\delta_{p_1^{-1}p_2^{-1}D_2}\bigast\cdots\bigast\delta_{p_1^{-1}\cdots p_{n}^{-1}D_n}
\end{equation*}
converges weakly to a Borel probability measure whose support is the compact set
\begin{equation*}
	T(\mathbb{P},\,\mathbb{D}):=\sum_{n=1}^{\infty}p_1^{-1}p_2^{-1}\cdots p_n^{-1}D_n,
\end{equation*}
where $\delta_{A}:=\frac{1}{\#A}\sum_{a\in A}\delta_a$ and $\delta_a$ denotes the Dirac measure at the point $a$. The limiting measure is usually denoted by
\begin{equation}\label{eq1.1}
	\mu_{\mathbb{P},\,\mathbb{D}}:=\delta_{p_1^{-1}D_1}\bigast\delta_{p_1^{-1}p_2^{-1}D_2}\bigast\cdots\bigast\delta_{p_1^{-1}\cdots p_{n}^{-1}D_n}\bigast\cdots,
\end{equation}
and is called a Moran-type measure (for the general definition of Moran-type measures, see \cite{Cao2026}).

The spectrality of Moran-type measures has attracted increasing attention in recent years. An and He \cite{AH2014} established a sufficient condition for the spectrality of Moran-type measures generated by consecutive digit sets. Later, Deng and Li \cite{Deng2022} proved that this condition is also necessary, thereby obtaining a complete characterization for Moran-type measures generated by consecutive digit sets.

\begin{theorem}\textup{\cite{AH2014,Deng2022}}
	Let $\mu_{\mathbb{P},\,\mathbb{D}}$ be the Moran-type self-similar measure associated to integer sequence $\mathbb{P}=\{p_n\}_{n=1}^{\infty}$ with $|p_n|\geqslant2$ and a sequence of consecutive digit sets $\mathbb{D}=\{D_n\}_{n=1}^{\infty}$, where $D_n=\{0,~1,~\cdots,~\allowbreak b_n-1\}$ with $b_n\geqslant2$ for $n\geqslant1$. Assume that the sequence $\{b_n\}_{n=1}^{\infty}$ is bounded. Then $\mu_{\mathbb{P},\,\mathbb{D}}$ is a spectral measure if and only if $b_n\mid p_n$ for all $n\geqslant2$.
\end{theorem}

To go beyond the consecutive-digit setting, Shi \cite{Shi2019} considered Moran-type measures generated by sequences of two-element and three-element digit sets satisfying the following three additional assumptions.
\begin{enumerate}	
	\item[\bfseries (S1)] For every $n\geqslant1$, $\#D_n\in\{2,~3\}$.
	\item[\bfseries (S2)] If $\#D_n=2$, the set $D_n$ has the form $\{0,~d_n\}$ such that $d_n=2^{l_n}d_n'$, where $l_n\in\mathbb{N}$ and $d_n'$ is a positive odd number. Moreover, $2^{l_n+1}$ divides $p_n$.
	\item[\bfseries (S3)] If $\#D_n=3$, the set $D_n$ has the form $\{0,~a_n,~b_n\}$ such that $0<a_n<b_n<p_n$, $\{a_n,~b_n\}\equiv\{1,~-1\}(\hspace{-0.5em}\mod{3})$, $\gcd(a_n,~b_n)=1$ and $3$ divides $p_n$.	
\end{enumerate}

\begin{theorem}\textup{\cite{Shi2019}}
	Let $\mathbb{P}:=\{p_n\}_{n=1}^{\infty}$ be a sequence of integers with $p_n>1$. Let $\mathbb{D}:=\{D_n\}_{n=1}^{\infty}$ be a sequence of sets with $D_n\subset\mathbb{N}$. Suppose that the pair $(\mathbb{P},~\mathbb{D})$ satisfies the conditions {\bfseries (S1)}, {\bfseries (S2)} and {\bfseries (S3)}.  Then, $\mu_{\mathbb{P},\,\mathbb{D}}$ is a spectral measure.	
\end{theorem}

Then An, Fu and Lai \cite{AFL2019} extended this result to more general cases.

However, all the above results require that there exist $C_n\subset\mathbb{Z}$ such that $(p_n,\,D_n,\ C_n)$ are Hadamard triples for all $n\ge1$. A natural question arise:

{\bf What will happen if we don't have the existence of Hadamard triples? }

Motivated by this, Deng and Li \cite{Deng2023}, Cao, Deng, Li and Wu\cite{Cao2024} established the necessary and sufficient conditions for $\mu_{\mathbb{P},\,\mathbb{D}}$ to be a spectral measure when $\#D_n=2$ for all $n\ge1$ and $\bigcup\limits_{n=1}^\infty D_n$ is bounded. To formulate their characterization, we first recall the multiplicity of a prime factor in an integer.

For any $n\in\mathbb{Z}$ and prime $m\in\mathbb{N}$, we denote by $\nu_m(n)$ the greatest integer $k\geqslant0$ such that $m^k$ is a factor of $n$, i.e., $\nu_m(n)=\max\{k\in\mathbb{Z}:\,m^k\mid n\}$ and write $\nu_m(0)=\infty$. Furthermore, for $\frac{a}{b}\in\mathbb{Q}$ with $a,b\in\mathbb{Z}$ and $b\ne0$, we define
\begin{equation*}
	\nu_m\Big(\frac{a}{b}\Big)=\nu_m(a)-\nu_m(b).
\end{equation*}

\begin{remark}
	If $m\in\mathbb{N}$ is not a prime number, then $\nu_m(a)-\nu_m(b)$ can not be determined uniquely by $\frac{a}{b}$. For example, $m=6$, $a=2$ and $b=6$. In this case, $\frac{2}{6}=\frac{1}{3}$, but $\nu_6(2)-\nu_6(6)=-1$ and $\nu_6(1)-\nu_6(3)=0$.
\end{remark}

\begin{theorem}\textup{\cite{Cao2024,Deng2023}}
	Let $\mathbb{P}:=\{p_n\}_{n=1}^{\infty}$ be a sequence of integers and $\mathbb{D}:=\{D_n\}_{n=1}^{\infty}$ be a sequence of sets of the form $\{0,~d_n\}\subset\mathbb{Z}$ satisfying $0<|d_n|<|p_n|$ for all $n\geqslant1$. Assume further that the sequence $\{d_n\}_{n=1}^{\infty}$ is bounded. Then, $\mu_{\mathbb{P},\,\mathbb{D}}$ is a spectral measure if and only if $k_i\neq k_j$ for all $j>i>1$, where
	\begin{equation*}
		k_{n}:=\nu_2\big(\frac{p_1p_2\cdots p_n}{2d_n}\big)=\nu_{2}(p_1p_2\cdots p_n)-\nu_2(2d_n),\qquad n=1,~2,~3,~\cdots.
	\end{equation*}		
\end{theorem}

This characterization was subsequently extended by Xiong \cite{Xiong2023} to the case: $\#D_n=3$ for all $n\ge1$.

\begin{theorem}\textup{\cite{Xiong2023}}
	Let $\mathbb{P}:=\{p_n\}_{n=1}^{\infty}$ be a sequence of integers and $\mathbb{D}:=\{D_n\}_{n=1}^{\infty}$ with $D_n=\{0,~a_n,~b_n\}\subset\mathbb{Z}$ such that $0<|a_n|<|b_n|<|p_n|$ for all $n\geqslant1$. If $\{b_n\}_{n=1}^{\infty}$ is bounded. Then, $\mu_{\mathbb{P},\,\mathbb{D}}$ is a spectral measure if and only if $\big\{\frac{a_i}{\gcd(a_i,~b_i)},~\frac{b_i}{\gcd(a_i,~b_i)}\big\}\equiv\{1,~-1\}(\hspace{-0.5em}\mod{3})$ and $k_i\neq k_j$ for all $j>i>1$, where
	\begin{equation*}
		k_{n}:=\nu_3\Big(\frac{p_1p_2\cdots p_n}{3\gcd(a_n,~b_n)}\Big)=\nu_{3}(p_1p_2\cdots p_n)-\nu_3(3\gcd(a_n,~b_n)),\qquad n=1,~2,~3,~\cdots.
	\end{equation*}	
\end{theorem}

Nevertheless, although the admissibility assumption is no longer required, both of the above results impose a restriction that all digit sets have the same cardinality. In this paper, we remove this restriction and investigate the spectrality of Moran-type measures generated by digit sets with varying cardinalities.

In this paper, we study pairs of sequences $(\mathbb{P},\,\mathbb{D})$ satisfying the following conditions.
\begin{enumerate}	
	\item[\bfseries (C1$^{\prime}$)] For every $n\geqslant1$, $(D_n-D_n)\subset\{0,\,\pm1,\,\pm2,\,\cdots,\,\pm(|p_n|-1)\}$ and $\#D_n\in\{2,\,3\}$, where $\#D_n$ denotes the cardinality of the set $D_n$.
	\item[\bfseries (C2$^{\prime}$)] The set $\cup_{n=1}^{\infty}D_n$ is bounded.
\end{enumerate}
Without loss of generality, we may assume $D_n=\{a_{n},\,b_{n}\}$ if $\#D_n=2$ and $D_n=\{a_{n},\,b_{n},\,c_{n}\}$ if $\#D_n=3$, where $a_{n},\,b_{n},\,c_{n}\in\mathbb{Z}$.

Let $A\subset\mathbb{Z}$ be a subset of integers with $\#A\geqslant2$. We denote by $\gcd(A)$ the greatest common divisor of $A$, defined as $\gcd(A):=\max\{d\in\mathbb{N}:\,d\mid a~\mbox{for~all}~a\in A\}$. Clearly, $\gcd(A-a)$ is independent of the choice $a\in A$, i.e., $\gcd(A-a)=\gcd(A-b)$ for $a\neq b\in A$. Let
\begin{equation}\label{eq1.2-1}
	d_n:=\gcd(D_{n}-a_{n}),\qquad\forall\,n\geqslant1.
\end{equation}
If $0\in D_n$, it is clear $\gcd(D_{n}-a_{n})=\gcd(D_{n})$. Hence
\begin{equation}\label{eq1.2-2}
	d_n=\gcd(D_{n})
\end{equation}
when $0\in D_n$.

We denote by $q_n$ the cardinality of $D_n$ and by $\widetilde{q}_n$ the integer in the set $\{2,\,3\}$ that is not equal to $q_n$, i.e.,
\begin{equation*}
	q_n:=\#D_n \mbox{~~and~~} \widetilde{q}_n\in\{2,\,3\}\setminus \{q_n\},\qquad\forall n\geqslant1.
\end{equation*}
Moreover, we define $\min A$ and  $\max A$ as the minimum and maximum of $A$, respectively; that is,
\begin{equation*}
	\min A:=\min\{a:a\in A\}~~\mbox{and}~~\max A:=\max\{a:a\in A\}.
\end{equation*}

Assume $(\mathbb{P},\,\mathbb{D})$ satisfies {\bfseries (C1$^{\prime}$)}. For any integer $n\geqslant1$, it is well known that the zero set of the Fourier transform of $\delta_{p_1^{-1}p_2^{-1}\cdots p_n^{-1}D_n}$ is $\mathcal{Z}(\hat\delta_{p_1^{-1}p_2^{-1}\cdots p_n^{-1}D_n})=\frac{p_1p_2\cdots p_n(\mathbb{Z}\setminus q_n\mathbb{Z})}{d_nq_n}$. We denote by $k_n$ the number of $q_n$-factor in $\frac{p_1p_2\cdots p_n}{d_nq_n}$, i.e.,
\begin{equation}\label{eq1.2}
	k_n:=\nu_{q_n}\Big(\frac{p_1p_2\cdots p_n}{d_nq_n}\Big  )=\nu_{q_n}(p_1p_2\cdots p_n)-\nu_{q_n}(d_nq_n).
\end{equation}
If it is necessary to emphasize the cardinality of $D_n$, we write $k_n$ by $k_{q_n,n}$. It is clear that
\begin{equation}\label{eq1.3-1}
	k_n=\nu_{q_n}(p_1p_2\cdots p_n)-\nu_{q_n}(d_nq_n)=\nu_{q_n}(p_1p_2\cdots p_n)-\nu_{q_n}(d_n)-1.
\end{equation}
Moerover, there is a rational number $L_n$ such that $\nu_{q_n}(L_n)=0$ and
\begin{equation}\label{eq1.3}
	\frac{p_1p_2\cdots p_n}{d_nq_n}=q_n^{k_n}L_n,\qquad n=1,\,2,\,\cdots.
\end{equation}
It is easy to see that
\begin{equation}\label{eq1.4-1}
	\nu_{\widetilde{q}_n}(L_n)=\nu_{\widetilde{q}_n}(p_1p_2\cdots p_n)-\nu_{\widetilde{q}_n}(d_nq_n)=\nu_{\widetilde{q}_n}(p_1p_2\cdots p_n)-\nu_{\widetilde{q}_n}(d_n).
\end{equation}
%For instance, for $\mathcal{Z}(\hat\delta_{\frac{1}{2^5\cdot5}\{0,\,7\}})=\frac{2^5\cdot5(2\mathbb{Z}+1)}{2\cdot7}$, the primary factor is $2$,  the number of primary factors is $4$,  the secondary factor is $3$ and the number of secondary factors is $0$.

%In this paper, we show that the spectrality of infinite mixed convolutions of two-element and three-element digit sets not only depends on the numbers of principal factors $k_n$, and is also related to the number of secondary factors.

Now we state our main theorem.

\begin{theorem}\label{th1.3}
	Let $\mathbb{P}:=\{p_n\}_{n=1}^{\infty}$ be a sequence of integers and $\mathbb{D}:=\{D_n\}_{n=1}^{\infty}$ be a sequence of subsets of ${\mathbb Z}$. If $(\mathbb{P},\,\mathbb{D})$ satisfies the conditions {\bfseries (C1$^{\prime}$)} and {\bfseries (C2$^{\prime}$)}. Then $\mu_{\mathbb{P},\,\mathbb{D}}$ as defined by $\eqref{eq1.1}$ is a spectral measure if and only if  $\{\frac{b_{n}-a_{n}}{d_n},\frac{c_{n}-a_{n}}{d_n}\}\equiv\{-1,\,1\}(\hspace{-0.5em}\mod~3)$ when $q_n=3$ and, for any two distinct numbers $m,\,n\in\mathbb{Z}^+$, one has
	\begin{flalign}
		\rm(\romannumeral1)~&k_m\ne k_n~when~q_m=q_n,\label{eq1.4}\\
		\rm(\romannumeral2)~&either~ k_{q_m,m}\geqslant \nu_{q_m}(L_{q_n,n})~or~ k_{q_n,n}\geqslant\nu_{q_n}(L_{q_m,m})~when~q_m\ne q_n.\label{eq1.5}&
	\end{flalign}
\end{theorem}

\begin{remark}
	If $q_n=q_1$ for all integers $n>1$, only \eqref{eq1.4} needs to be considered. This means that our result contains the results of \cite{Deng2023}, \cite{Cao2024} and \cite{Xiong2023} as special cases.
\end{remark}

The following two examples, one failing to satisfy \eqref{eq1.5} and the other satisfying it, illustrate the significance of \eqref{eq1.5} for the spectrality of Moran-type measures with varying cardinalities.

\begin{example}
	Let $\mu:=\delta_{\frac{1}{2^23^4}\{0,\,1\}}\bigast\delta_{\frac{1}{2^23^4}\frac{1}{2^23}\{0,\,3,\,6\}}$. It is clear that $\mathcal{Z}(\hat\delta_{\frac{1}{2^23^4}\{0,\,1\}})=2\cdot3^4(2\mathbb{Z}+1)$ and $\mathcal{Z}(\hat\delta_{\frac{1}{2^23^4}\frac{1}{2^23}\{0,\,3,\,6\}})=3^3\cdot2^4(3\mathbb{Z}\pm1)$. Thus $k_{2,1}=1<4=\nu_{2}(L_2)$ and $k_{3,2}=3<4=\nu_{3}(L_1)$. This means that this example does not meet the conditions of the formula \eqref{eq1.5}. On the other hand, for any $z_1\in\mathcal{Z}(\hat\delta_{\frac{1}{2^23^4}\{0,\,1\}})$ and $z_2\in\mathcal{Z}(\hat\delta_{\frac{1}{2^23^4}\frac{1}{2^23}\{0,\,3,\,6\}})$, there exists integers $n_1\in(2\mathbb{Z}+1)$ and $n_2\in(3\mathbb{Z}\pm1)$ such that $z_1=2\cdot3^4\cdot n_1$ and $z_2=3^3\cdot2^4\cdot n_2$. Hence, $z_1+z_2=2\cdot3^3(3\cdot n_1+2^3\cdot n_2)$. Since $n_1\in(2\mathbb{Z}+1)$ and $n_2\in(3\mathbb{Z}\pm1)$, then $(3\cdot n_1+2^3\cdot n_2)\in[\mathbb{Z}\setminus(2\mathbb{Z}\cup3\mathbb{Z})]$. We now see $z_1+z_2\not\in\mathcal{Z}(\mu)$. According to the arbitrariness of $z_1$ and $z_2$, it is easy to see that $\mu$ is not a spectral measure.
\end{example}

\begin{example}
	Let $\mu:=\delta_{\frac{1}{2^23^4}\{0,\,1\}}\bigast\delta_{\frac{1}{2^23^4}\frac{1}{2^23}\{0,\,1,\,2\}}$. It is clear that $\mathcal{Z}(\hat\delta_{\frac{1}{2^23^4}\{0,\,1\}})=2\cdot3^4(2\mathbb{Z}+1)$ and $\mathcal{Z}(\hat\delta_{\frac{1}{2^23^4}\frac{1}{2^23}\{0,\,1,\,2\}})=3^4\cdot2^4(3\mathbb{Z}\pm1)$. Then, $k_{2,1}=1<4=\nu_{2}(L_2)$ and $k_{3,2}=4=\nu_{3}(L_1)$. It is easy to see that $2\cdot3^4\{0,\,3\}+3^4\cdot2^4\{0,\,1,\,2\}$ is a spectrum of $\mu$.
\end{example}

For convenience, we write
\begin{equation}\label{eq1.6-}
X_2:=\{n\geqslant1:\,q_n=2\},\qquad X_3:=\{n\geqslant1:\,q_n=3\},\qquad D:=\max\{d_n:\,n\geqslant1\}.
\end{equation}
For $n\geqslant0$, let
\begin{equation}\label{eq4.5-}
\mu_{>n}:=\delta_{p_{n+1}^{-1}D_{n+1}}\bigast\delta_{p_{n+1}^{-1}p_{n+2}^{-1}D_{n+2}}\bigast\cdots\bigast\delta_{p_{n+1}^{-1}\cdots p_{n+N}^{-1}D_{n+N}}\bigast\cdots.
\end{equation}

We now outline the structure of the paper. In Section \ref{sec2}, we will introduce some essential notations and lemmas. Moreover, we show that it suffices to prove Theorem \ref{th1.3} under a simplified condition, thereby reducing the problem to a simplified model. In Section \ref{sec3}, we will prove the necessity of Theorem \ref{th1.3}. In Sections \ref{sec5}, we present the construction of the spectrum of the probability measure generated by finitely many convolutions. In Section  \ref{sec7}, we prove the sufficiency of Theorem \ref{th1.3}. This proof relies on Proposition \ref{prop7.13}.

\section{Preliminaries}\label{sec2}

We  introduce some necessary results in the first part of this section, and the second part is devoted to simplifying $\{p_n\}_{n=1}^{\infty}$ and $\{D_n\}_{n=1}^{\infty}$.

Let $\{D_n\}_{n=1}^{\infty}$ be a sequence of subsets of ${\mathbb Z}$, and let $d_n$ denote the greatest common divisor of the elements in the set $D_n-a_n$ ($a_n\in D_n$).

We call
\begin{equation}\label{eq2.1}
	m_{D_n}(x):=\frac{1}{q_n}\sum_{d\in D_n}e^{2\pi idx},\qquad x\in\mathbb{R}
\end{equation}
the mask of $D_n$. Let $\widehat{\mu_{\mathbb{P},\,\mathbb{D}}}(x)$ be the Fourier transform of the measure $\mu_{\mathbb{P},\,\mathbb{D}}$ defined in \eqref{eq1.1}, we get
\begin{equation}\label{eq2.3}
	\widehat{\mu_{\mathbb{P},\,\mathbb{D}}}(x)=\prod_{n=1}^{\infty}m_{D_n}(p_1^{-1}\cdots p_n^{-1}x),\qquad\forall \, x\in\mathbb{R}.
\end{equation}

Let $\mathcal{Z}(f)$ be the set of all solutions of the equation $f(x)=0$. It is know that $\mathcal{Z}(m_{D})=\frac{\mathbb{Z}\setminus 2\mathbb{Z}}{2d}$ if $D=\{0,\,d\}$ with $d\neq0$. When $D$ is a three-element digit sets, similar results can also be obtained as in the following lemma.

\begin{lemma}[\cite{Xiong2023}]\label{le2.1}
	Let $D=\{0,\,a,\,b\}\subset\mathbb{Z}$, then
	\begin{enumerate}
		\item[\rm(\romannumeral1)] $\mathcal{Z}(m_{D})\ne\emptyset$ if and only if $\big\{\frac{a}{gcd(a,\,b)},\,\frac{b}{gcd(a,\,b)}\big\}\equiv\{1,\,-1\}(\mod~3)$;
		\item[\rm(\romannumeral2)] If $\mathcal{Z}(m_{D})\ne\emptyset$, then $\mathcal{Z}(m_{D})=\frac{\mathbb{Z}\setminus 3\mathbb{Z}}{3gcd(a,\,b)}$.
	\end{enumerate}
\end{lemma}

Since $m_{D_n-a_n}(x)=e^{-2\pi ia_nx}m_{D_n}(x)$ and the exponential factor is non-vanishing, it follows that $\mathcal{Z}(m_{D_n})=\mathcal{Z}(m_{D_n-a_n})$. According to $\eqref{eq2.1}$ and the above discussion, we have
\begin{equation}\label{eq2.4}
	\mathcal{Z}(m_{D_n})=\frac{\mathbb{Z}\setminus q_n\mathbb{Z}}{d_nq_n}
\end{equation}
and
\begin{equation}\label{eq2.5}
	m_{p_1^{-1}p_2^{-1}\cdots p_n^{-1}D_n}(x)=m_{D_n}(p_1^{-1}p_2^{-1}\cdots p_n^{-1}x)
\end{equation}
for $n\geqslant1$. It follows from \eqref{eq1.3}, \eqref{eq2.4} and \eqref{eq2.5} that
\begin{equation}\label{eq2.6}
	\mathcal{Z}(m_{p_1^{-1}p_2^{-1}\cdots p_n^{-1}D_n})=\frac{p_1p_2\cdots p_n}{d_nq_n}(\mathbb{Z}\setminus q_n\mathbb{Z})=q_n^{k_n}L_n(\mathbb{Z}\setminus q_n\mathbb{Z}).
\end{equation}
Furthermore, \eqref{eq2.3} shows
\begin{equation}\label{eq2.7}
	\mathcal{Z}(\widehat{\mu_{\mathbb{P},\,\mathbb{D}}})=\bigcup_{n=1}^{\infty}\mathcal{Z}(m_{p_1^{-1}p_2^{-1}\cdots p_n^{-1}D_n})=\bigcup_{n=1}^{\infty}\frac{p_1p_2\cdots p_n}{d_nq_n}(\mathbb{Z}\setminus q_n\mathbb{Z})=\bigcup_{n=1}^{\infty}q_n^{k_n}L_n(\mathbb{Z}\setminus q_n\mathbb{Z}).
\end{equation}

For any finite or countable subset $\Lambda\subset\mathbb{R}$, we write
\begin{equation*}
	E(\Lambda):=\{e^{2\pi i\lambda x}:\,\lambda\in\Lambda\}.
\end{equation*}
We say that $\Lambda$ is a bi-zero set of $\mu_{\mathbb{P},\,\mathbb{D}}$ if
\begin{equation}\label{eq2.8}
	(\Lambda-\Lambda)\setminus\{0\}\subset\mathcal{Z}(\widehat{\mu_{\mathbb{P},\,\mathbb{D}}}).
\end{equation}
It is clear that the orthogonality of $E(\Lambda)$ in $L^2(\widehat{\mu_{\mathbb{P},\,\mathbb{D}}})$ is equivalent to that $\Lambda$ is a bi-zero set of $\mu_{\mathbb{P},\,\mathbb{D}}$. Hence $\Lambda\subset\mathbb{R}$ is a uniformly discrete set since $|\widehat{\mu_{\mathbb{P},\,\mathbb{D}}}(x)|>0$ in a small interval $[-\varepsilon,\,\varepsilon]$. Without loss of generality, we always assume $0\in\Lambda$ and define
\begin{equation*}
	Q_\Lambda(\xi)=\sum_{\lambda\in\Lambda}|\widehat{\mu_{\mathbb{P},\,\mathbb{D}}}(\lambda+\xi)|^2,\qquad\forall\, x\in\mathbb R.
\end{equation*}

By using the Parseval identity, Jorgenson and Pederson (\cite{Jorgensen1998}) obtained the following basic criterion for the orthogonality and spectrality of $E_\Lambda$ in $L^2(\mu_{\mathbb{P},\,\mathbb{D}})$.

\begin{lemma}\label{le2.2}
The exponential function set $E_\Lambda$ is an orthogonal set of $L^2(\mu_{\mathbb{P},\,\mathbb{D}})$ if and only if $Q_\Lambda(\xi)\leqslant1$ for all $\xi\in\mathbb R$, and $E_\Lambda$ is an orthogonal basis of $L^2(\mu_{\mathbb{P},\,\mathbb{D}})$ if and only if $Q_\Lambda(\xi)=1$ for all $\xi\in\mathbb R$.
\end{lemma}

\begin{definition}
	Let $D$ and $S$ be two finite subsets of $\mathbb R$ with the same cardinality, we call $(D,S)$ is a compatible pair if
	\begin{equation*}
		\Big[\frac{1}{\sqrt{\#D}}e^{2\pi i ds}\Big]_{d\in D,s\in S}
	\end{equation*}
	is a unitary matrix.
\end{definition}

The following conclusion is well~known.

\begin{lemma}\label{le2.4}
	Let $D$ and $S$ be two finite subsets of $\mathbb R$ with the same cardinality, then the following statements are equivalent:
	\begin{enumerate}
		\item[\rm(\romannumeral1)] $(D,\,S)$ is a compatible pair,
		\item[\rm(\romannumeral2)] $m_{D}(s_{1}-s_{2})=0$ for any two distinct numbers $s_{1}$, $s_{2}\in S$,
		\item[\rm(\romannumeral3)]	$\sum_{s\in S}\lvert m_{D}(\xi+s)\rvert^{2}=1$ for any $\xi\in\mathbb {R}$.
	\end{enumerate}
\end{lemma}

%
%In particular, that $(p_n^{-1}D_n,\,C_n)$ is a compatible pair means that $\#C_n=\#D_n$ and all non-zero elements in $C_n-C_n$ are solutions of the equation $m_{D_n}(p_n^{-1}x)=0$. Equivalently, $C_n$ is a spectrum for the uniform probability measure $\delta_{p_n^{-1}D_n}$, i.e.,
%\begin{equation*}
%	\sum_{c\in C_n}|m_{D_n}(p_n^{-1}(x+c))|^2=\sum_{c\in C_n}|\frac{1}{\#D_n}\sum_{d\in D_n}e^{2\pi idp_n^{-1}(x+c)}|^2=1,\qquad\forall \, x\in\mathbb{R}.
%\end{equation*}

\begin{lemma}\textup{\cite[Proposition 2.2]{Cao2024}}\label{le5.2}
	Let $\nu$ be a probability measure and its support has finite cardinality $N$. If $E(\Lambda)$ is an orthogonal set of $L^2(\nu)$ and $\#\Lambda\geqslant N$, then $\Lambda$ is a spectrum of $\nu$ and $\#\Lambda=N$.
\end{lemma}

\begin{lemma}[\cite{Deng2022}]\label{le2.6}
	Let $p_{i,j}>0$ be positive numbers such that $\sum_{j=1}^{k}p_{i,j}=1~(i=1,\,2,\,\cdots,\,m)$ and $x_{i,j}\geqslant0$ be nonnegative numbers with $\sum_{i=1}^{m}\max\{x_{i,j}:\,1\leqslant j\leqslant k\}\leqslant1$. Then $\sum_{i=1}^{m}\sum_{j=1}^{k}p_{i,j}x_{i,j}=1$ if and only if $x_{i,1}=x_{i,2}=\cdots=x_{i,k}$ for $1\leqslant i\leqslant m$ and $\sum_{i=1}^{m}x_{i,1}=1$.
\end{lemma}

\begin{lemma}[\cite{Dai2014}]\label{le2.8}
	Let $\mu=\mu_0\bigast\mu_1$ be the convolution of the Borel probability measures $\mu_0$ and $\mu_1$, where $\mu_0$ and $\mu_1$ are not Dirac measures. Suppose that $\Lambda$ is a bi-zero set of $\mu_0$, then $\Lambda$ is also a bi-zero set of $\mu$, but cannot be a spectrum of $\mu$.
\end{lemma}

The following lemma shows that condition {\bfseries (C1$^{\prime}$)} in Theorem \ref{th1.3} can be replaced by a stricter condition.

\begin{lemma}[\cite{Deng2023}]\label{le2.5}
In order to prove Theorem \ref{th1.3}, one only need to consider the special case: $p_n>1$ and $0\in D_n\subset\{0,\,1,\cdots,\, p_n-1\}$ for all $n\geq 1$.
\end{lemma}

\begin{pf}
For any $n\geqslant1$, we define
\begin{equation*}
	E_n=:|p_1p_2\cdots p_n|(p_1^{-1}p_2^{-1}\cdots p_n^{-1}D_n-\min\{x:\,x\in p_1^{-1}p_2^{-1}\cdots p_n^{-1}D_n\}),\quad \forall\ n\ge1.
\end{equation*}
It is obvious that
\begin{equation*}
\mu_{\{p_n\},\{D_n\}}(~\cdot~)=\mu_{\{|p_n|\},\{E_n\}}\Big(~\cdot-\sum_{n=1}^{\infty}\min\{x:\,x\in p_1^{-1}p_2^{-1}\cdots p_n^{-1}D_n\}\Big).
\end{equation*}
Thus, Lemma \ref{le2.2} shows that $\mu_{\{p_n\},\{D_n\}}$ is a spectral measure if and only if $\mu_{\{|p_n|\},\{E_n\}}$ is.

Furthermore, note that either $E_n=D_n-\alpha_n$ for some $\alpha_n\in D_n$ or $E_n=\alpha_n-D_n$ for some $\alpha_n\in D_n$, so $d_n=\gcd E_n$. Hence the conditions \rm(\romannumeral1) and \rm(\romannumeral2) of Theorem \ref{th1.3} are the same when $(\{p_n\},\,\{D_n\})$ is replaced by $(\{|p_n|\},\,\{E_n\})$. Hence Lemma \ref{le2.5} is proven.		
\end{pf}

The following lemma show that the spectrality of $\mu$ is independent of $p_1\ne0$.

\begin{lemma}[\cite{Wu2022}]\label{le2.7}
	Let $\mu$ and $\nu$ be Borel probability measure with compact support on $\mathbb{R}$. If there exists $c\in\mathbb{R}\setminus\{0\}$ and $b\in\mathbb{R}$ such that
	\begin{equation*}
		\hat\nu(x)=e^{2\pi ibx }\hat\mu(cx),\qquad x\in\mathbb R.
	\end{equation*}
	Then $\nu$ is a spectral measure with spectrum $\Lambda$ if and only if $\mu$ is a spectral measure with spectrum $c\Lambda$.
\end{lemma}

Similar to the proof of \cite[Proposition 2.1]{Cao2024}, we have the following lemma.

\begin{lemma}\label{le5.1}
	In order to prove Theorem \ref{th1.3}, one only need to consider the special case: $\gcd(p_n,\,d_n)=1$ for all $n\geqslant1$.
\end{lemma}

\begin{pf}
By Lemma \ref{le2.5}, without loss of generality, we may assume that $0\in D_n$ for all $n\geqslant1$. In this case, we have
\begin{equation*}
	d_n=\gcd (D_n),\qquad\forall\,n\geqslant1.
\end{equation*}

Let $z_0:=1$ and define inductively
\begin{equation}\label{eq5.2-2}
	z_n:=\gcd(|z_{n-1}p_n|,\,|d_n|),\,~t_n:=\frac{z_{n-1}p_n}{z_n}~~\mbox{and}~~G_n:=\frac{1}{z_n}D_n,\qquad\forall \, n\geqslant1.
\end{equation}
It is clear $\delta_{t_1^{-1}t_2^{-1}\cdots t_{n}^{-1}G_n}=\delta_{p_1^{-1}p_2^{-1}\cdots p_n^{-1}D_n}$ for any $n\geqslant1$. Hence
\begin{equation*}
	\mu_{\mathbb{P},\,\mathbb{D}}=\mu_{\{t_n\},\,\{G_n\}}.
\end{equation*}
This means that $\mu_{\mathbb{P},\,\mathbb{D}}$ is a spectral measure if and only if $\mu_{\{t_n\},\{G_n\}}$ is.

Since $z_n\mid d_n$ and $d_n=\gcd(D_n)$, we get $G_n=\frac{1}{z_n}D_n\subset\mathbb{Z}$ for all $n\geqslant1$. Thus, the definition of $G_n$ shows that $0\in G_n$ and
\begin{equation*}
	d_n^{\mathbb{G}}:=\gcd(G_n)=\frac{d_n}{z_n}\cdot\gcd\Big(\frac{D_n}{d_n}\Big)=\frac{d_n}{z_n},\qquad\forall \, n\geqslant1.
\end{equation*}
Therefore, $\gcd(t_n,\,d_n^{\mathbb{G}})=1$ for any $n\geqslant1$. Note that $t_1^{-1}t_2^{-1}\cdots t_{n}^{-1}d_n^{\mathbb{G}}=p_1^{-1}p_2^{-1}\cdots p_n^{-1}d_n$, we see that condition \rm(\romannumeral1) and \rm(\romannumeral2) of Theorem \ref{th1.3} are the same when $(\{p_n\},\,\{D_n\})$ is replaced by $(\{t_n\},\,\{G_n\})$. Hence Lemma \ref{le5.1} is proven.	
\end{pf}

The following lemma shows that conditions \eqref{eq1.4} and \eqref{eq1.5} in Theorem \ref{th1.3} are independent of the choice of $p_1\neq0$.

\begin{lemma}\label{le2.8-1}
In order to prove Theorem \ref{th1.3}, it suffices to consider the special case: $\mathcal{Z}(\hat\mu_{\mathbb{P},\,\mathbb{D}})\subset\mathbb{Z}$ and so $k_n\geqslant0$ and $\nu_{\widetilde{q}_n}(L_n)\geqslant0$ for all $n\geqslant1$.
\end{lemma}

\begin{pf}
By the assumption that $\bigcup\limits_{n=1}^\infty D_n$ is bounded and so is a finite subset of $\mathbb{Z}$, we see that there is a positive integer $t_1\ge |p_1|$ such that $q_nd_n|t_1$ for all $n\ge 1$. Let $t_n:=p_n$ for all $n\geqslant2$ and $\mathbb{T}:=\{t_n\}_{n=1}^{\infty}$. Hence $\mathcal{Z}(\hat\mu_{\mathbb{T},\mathbb{D}})\subset\mathbb{Z}$. Furthermore, it is easy to see that $(\mathbb{T},\,\mathbb{D})$ satisfies the conditions {\bfseries (C1$^{\prime}$)} and {\bfseries (C2$^{\prime}$)}.

Lemma \ref{le2.7} shows $\mu_{\mathbb{P},\,\mathbb{D}}$ is a spectral measure if and only if $\mu_{\mathbb{T},\,\mathbb{D}}$ is. On the other hand, note
\[
\nu_{q}\Big(\frac{p_1p_2\cdots p_n}{d_nq_n}\Big)-\nu_{q}\Big(\frac{t_1t_2\cdots t_n}{d_nq_n}\Big)=\nu_{q}(p_1)-\nu_{q}(t_1),\quad\forall\ q\in\{2,\ 3\}\mbox{ and }n\ge1.
\]
Hence the conditions \rm(\romannumeral1) and \rm(\romannumeral2) of Theorem \ref{th1.3} hold for $(\{p_n\},\,\{D_n\})$ if and only if they hold for $(\{t_n\},\,\{D_n\})$. Our Lemma is proven. 		
\end{pf}

It follows from formula \eqref{eq1.2-2} and Lemmas \ref{le2.5}, \ref{le5.1} and \ref{le2.8-1} that, in order to prove Theorem \ref{th1.3}, we only need to consider the special case:
\begin{flalign}\label{eq2.12-3}
	&\left\{\begin{aligned}
		\text{\bfseries (C1)} \quad &\mbox{For~every}~n\geqslant1,\,0\in D_n\subset\{0,\,1,\cdots,p_n-1\}~\mbox{and}~\#D_n\in\{2,\,3\}.\\
		\text{\bfseries (C2)} \quad & \mbox{The~set}~\cup_{n=1}^{\infty}D_n~\mbox{is~bounded}.\\
		\text{\bfseries (C3)} \quad & \mathcal{Z}(\hat\mu_{\mathbb{P},\,\mathbb{D}})\subset\mathbb{Z}\mbox{ and }\gcd(p_n,\,d_n)=1 \mbox{ for all } n\geqslant1, \mbox{where }d_n=\gcd( D_n).
	\end{aligned}\right.&
\end{flalign}
We will always assume the above \eqref{eq2.12-3} in the following argument. Since $0\in D_n$ for all $n\geqslant1$, \eqref{eq1.2-2} holds for all $D_n$. Without loss of generality, we may replace the definition of $d_n$ in \eqref{eq1.2-1} with \eqref{eq1.2-2} in what follows. Given that $d_n=\gcd( D_n)$ implies $d_n^{-1}D_n\subset\mathbb{Z}$, one may verify that the function $m_{d_n^{-1}D_n}$ is an integral periodic function, i.e.,
\begin{equation}\label{eq2.2}
	m_{d_n^{-1}D_n}(x+L)=m_{d_n^{-1}D_n}(x),\qquad \forall \, L\in\mathbb{Z},\,n\geqslant1.
\end{equation}

Assume that the pair of sequences $(\mathbb{P},\,\mathbb{D})$ satisfies \eqref{eq2.12-3}. By {\bfseries (C1)}, we have
\begin{equation}\label{eq2.11}
	p_n\geqslant\bigg\{\begin{array}{ll}
		2,&\mbox{if } n\in X_2,\\
		3,&\mbox{if } n\in X_3
	\end{array}\bigg.
\end{equation}
for all $n\geqslant1$. Applying  {\bfseries (C1)} again, we obtain that there exist integer sequences $\{\alpha_n\}_{n\in X_3}$ and $\{\beta_n\}_{n\in X_3}$ such that $0<\alpha_n<\beta_n<p_n$ and $D_n=\{0,\,\alpha_n,\,\beta_n\}$ for all $n\in X_3$. It is clear that $0<\frac{\alpha_{n}}{d_{n}}<\frac{\beta_{n}}{d_{n}}<\frac{p_{n}}{d_{n}}$ for any $n\in X_3$. This means $\frac{p_{n}}{d_{n}}>2$ if $n\in X_3$ by noting $\{\frac{\alpha_{n}}{d_{n}},\,\frac{\beta_{n}}{d_{n}}\}\subset\mathbb{Z}$. Moreover, by {\bfseries (C1)}, we have $d_n<p_n$ for all $n\in X_2$. In conclusion, we get
\begin{equation}\label{eq2.10}
	\frac{d_{n}}{p_{n}}<\bigg\{\begin{array}{ll}
		1,&\mbox{if}~~n\in X_2,\\
		\frac{1}{2},&\mbox{if}~~n\in X_3.
	\end{array}\bigg.
\end{equation}

At the end of this section, we give a general result on the lower bound of the products of mask functions which will be used in Section 5.

\begin{lemma}\label{le2.9}
	Let $\{E_n\}_{n\geqslant1}$ be a sequence of subsets of $\mathbb{Z}$ such that $0\in E_n$, $\#E_n\in\{2,\,3\}$ and $\gcd(E_n)=1$ for every $n\geqslant1$. If $\cup_{n=1}^{\infty}E_n$ is bounded, then there exists a constant $\varepsilon_1>0$ such that
	\begin{equation*}
		\prod_{n=1}^{\infty}\lvert m_{E_n}(x_{n})\rvert\geqslant\varepsilon_1
	\end{equation*}
	for any sequence of real numbers $\{x_n\}_{n=1}^{\infty}\subset[-\frac{1}{6},\,\frac{1}{6}]$ satisfying $\sum_{n=1}^{\infty}|x_n|\leqslant 1$.	
\end{lemma}

\begin{pf}
	Since $0\in E_n$ and $\gcd(E_n)=1$, we see that $\mathcal{Z}(m_{E_n})=\frac{2\mathbb{Z}+1}{2}$ if $\#E_n=2$ and $\mathcal{Z}(m_{E_n})=\frac{3\mathbb{Z}\pm1}{3}$ if $\#E_n=3$.
	
	Since $\cup_{n=1}^{\infty}E_n$ is bounded, the family $\{m_{E_n}(x):\,n\geqslant1\}$ is finite. Note that $m_{E_n}(0)=1$ for every $n\geqslant1$, we get
	\begin{equation}\label{eq2.14}
		\alpha_1:=\inf\{m_{E_n}(x):\,|x|\leqslant\frac{1}{6},\,n\geqslant1\}>0.
	\end{equation}
	
	Since the family $\{m_{E_n}(x):\,n\geqslant1\}$ is finite, we see that there exists a constant $c_1$ such that
	\begin{equation*}
		\lvert 1 -m_{E_n}(x)\rvert=\lvert m_{E_n}(0)-m_{E_n}(x)\rvert\leqslant c_1\lvert x\rvert,\qquad\forall \, n\geqslant1,\,~ x\in(-1,\,1).
	\end{equation*}
	Note an elementary inequality $1-x\geqslant e^{-2x}$ for all $0\leqslant x\leqslant\frac{1}{2}$, we have
	\begin{equation*}
		\lvert m_{E_n}(x_{n})\rvert\geqslant1-c_1\lvert x_{n}\rvert\geqslant e^{-2c_1\lvert x_{n}\rvert},\quad\mbox{if}~~2c_1\lvert x_{n}\rvert\leqslant1.
	\end{equation*}
	Note that $\sum_{n=1}^{\infty}|x_n|\leqslant 1$ shows $\#\{n:\,2c_1\lvert x_{n}\rvert>1\}\leqslant2c_1$, we have
	\begin{equation*}
		\begin{split}
			\prod_{n=1}^{\infty}\lvert m_{E_n}(x_{n})\rvert=&\prod_{n:\,2c_1\lvert x_{n}\rvert>1}\lvert m_{E_n}(x_{n})\rvert\cdot\prod_{n:\,2c_1\lvert x_{n}\rvert\leqslant1} \lvert m_{E_n}(x_{n})\rvert\\
			\geqslant&\alpha_1^{2c_1}\cdot\prod_{n:\,2c_1\lvert x_{n}\rvert\leqslant1} e^{-2c_1\lvert x_{n}\rvert}\\
			\geqslant&\alpha_1^{2c_1}\cdot e^{-2c_1}=:\varepsilon_1>0.
		\end{split}
	\end{equation*}
	This completes the proof.
\end{pf}

\section{Proof of the necessity Theorem \ref{th1.3}}\label{sec3}

Assume that the pair $(\mathbb{P},\,\mathbb{D})$ satisfies \eqref{eq2.12-3}. Let $\mu_{\mathbb{P},\,\mathbb{D}}$ be the measure defined by \eqref{eq1.1}, and let $\Lambda$ be a spectrum of $\mu_{\mathbb{P},\,\mathbb{D}}$. Given that $\Lambda$ is a bi-zero set of $\mu_{\mathbb{P},\,\mathbb{D}}$, it follows from \eqref{eq2.8} and {\bfseries (C3)} that $\Lambda\subset\mathbb{Z}$. Define
\begin{equation}\label{eq3.1}
	\Lambda_{i,\,j}^{(n)}:=\mathbb{Z}\cap\Big[\frac{1}{q_n^{k_n+1}L_n}\Big(\Lambda-(i+q_n^{k_n}L_nj)\Big)\Big]
\end{equation}
for every $n\geqslant1$, $0\leqslant i\leqslant q_n^{k_n}L_n-1$ and $0\leqslant j\leqslant q_n-1$. For $n\geqslant1$, $0\leqslant i\leqslant q_n^{k_n}L_n-1$ and $0\leqslant j\leqslant q_n-1$, it is clear $0\leqslant i+q_n^{k_n}L_nj\leqslant q_n^{k_n+1}L_n-1$, which implies that the family $\{\Lambda_{i,\,j}^{(n)}\}_{0\leqslant i\leqslant q_n^{k_n}L_n-1,\,0\leqslant j\leqslant q_n-1}$ is pairwise disjoint. That is, whenever $i_1\neq i_2$ or $j_1\neq j_2$, we have $\Lambda_{i_1,\,j_1}^{(n)}\cap\Lambda_{i_2,\,j_2}^{(n)}=\emptyset$.

According to the definition of $\Lambda_{i,\,j}^{(n)}$ given in \eqref{eq3.1}, we can decompose $\Lambda$ as follows
\begin{equation}\label{eq3.2}
	\displaystyle \Lambda=\bigcup_{i=0}^{q_n^{k_n}L_n-1}\bigcup_{j=0}^{q_n-1}\Big(i+q_n^{k_n}L_nj+q_n^{k_n+1}L_n\Lambda_{i,\,j}^{(n)}\Big),
\end{equation}
where $i+q_n^{k_n}L_nj+q_n^{k_n+1}L_n\Lambda_{i,\,j}^{(n)}$ is empty when $\Lambda_{i,\,j}^{(n)}=\emptyset$.

For any integer $n\geqslant 1$, we define
\begin{equation}\label{eq4.8}
	\chi_n:=\delta_{p_1^{-1}p_2^{-1}\cdots p_n^{-1}D_n}.
\end{equation}
Therefore, the measure $\mu_{\mathbb{P},\,\mathbb{D}}$ defined in \eqref{eq1.1} can be rewritten as $\mu_{\mathbb{P},\,\mathbb{D}}=\chi_1\bigast\chi_2\bigast\cdots\chi_n\cdots$. Let ${\mathcal A}=\{s_1,\ s_2,\ \cdots,\ s_m\}$ be a finite subset of positive integers. Without loss of generality, assume that  $s_1<s_2<\cdots<s_m$. We define
\begin{equation}\label{eq4.8+}
	\chi_{\mathcal A}=\chi_{s_1}\bigast\chi_{s_2}\bigast\cdots\bigast\chi_{s_m},
\end{equation}
\begin{equation}\label{eq4.9}
	\begin{split}
		\omega_{\mathcal A}:=&\chi_1\bigast\chi_2\bigast\cdots\bigast\chi_{s_1-1}\\
		\bigast&\chi_{s_1+1}\bigast\chi_{s_1+2}\bigast\cdots\bigast\chi_{s_2-1}\\
		\bigast&\chi_{s_2+1}\bigast\chi_{s_2+2}\bigast\cdots\bigast\chi_{s_3-1}\\
		\bigast&\cdots\bigast\\
		\bigast& \chi_{s_m+1}\bigast\chi_{s_m+2}\bigast\cdots\bigast\chi_{s_m+n}\bigast\cdots.
	\end{split}
\end{equation}
It is clear that $\omega_{\mathcal A}$ is the probability measure obtained  by deleting $\chi_{s_1}$, $\chi_{s_2}$, $\cdots$, $\chi_{s_m}$ from  the convolution representation \eqref{eq1.1} and so $\mu_{\mathbb{P},\,\mathbb{D}}=\chi_{\mathcal A}\bigast\omega_{\mathcal A}$.

\begin{lemma}\label{le3.2}
	Assume that the pair $(\mathbb{P},\,\mathbb{D})$ satisfies \eqref{eq2.12-3}. Let $\mu_{\mathbb{P},\,\mathbb{D}}$ be the measure defined by \eqref{eq1.1} and $\Lambda$ be a spectrum of $\mu_{\mathbb{P},\,\mathbb{D}}$. Take $0\leqslant j_i\leqslant q_n-1$ for each $0\leqslant i\leqslant q_n^{k_n}L_n-1$, and write
	\begin{equation*}
		\Gamma_{j_0,\,j_1,\cdots,j_{q_n^{k_n}L_n-1}}:=\bigcup_{i=0}^{q_n^{k_n}L_n-1}\Big(i+q_n^{k_n}L_nj_i+q_n^{k_n+1}L_n\Lambda_{i,\,j_i}^{(n)}\Big).
	\end{equation*}
	Then, $E(\Gamma_{j_0,\,j_1,\cdots,j_{q_n^{k_n}L_n-1}})$ is either an empty set or an orthogonal family of $L^2(\omega_{n})$, where $\omega_{n}$ is defined as in \eqref{eq4.9}.
\end{lemma}

\begin{pf}
	Assume that the set $\Gamma_{j_0,\,j_1,\cdots,j_{q_n^{k_n}L_n-1}}$ contains at least two elements, otherwise the conclusion holds trivially. Let $\alpha$ and $\beta$ be distinct elements in set $\Gamma_{j_0,\,j_1,\cdots,j_{q_n^{k_n}L_n-1}}$, we can write $\alpha=i_1+q_n^{k_n}L_nj_{i_1}+q_n^{k_n+1}L_n\lambda_1$ and $\beta=i_2+q_n^{k_n}L_nj_{i_2}+q_n^{k_n+1}L_n\lambda_2$ for some $0\leqslant i_1,\,i_2\leqslant q_n^{k_n}L_n-1$, $\lambda_1\in\Lambda_{i_1,\,j_1}^{(n)}$ and $\lambda_2\in\Lambda_{i_2,\,j_2}^{(n)}$, where either $i_1\ne i_2$ or $\lambda_1\ne\lambda_2$. It is clear that
	\begin{equation*}
		\alpha-\beta=i_1-i_2+q_n^{k_n}L_n(j_{i_1}-j_{i_2})+q_n^{k_n+1}L_n(\lambda_1-\lambda_2).
	\end{equation*}
	By \eqref{eq3.2}, it is easy to check that both $\alpha$ and $\beta$ belong to $\Lambda$. Given that $\Lambda$ is a bi-zero set of $\mu_{\mathbb{P},\,\mathbb{D}}$, it follows that
	\begin{equation}\label{eq3.6-1}
		0=\widehat{\mu_{\mathbb{P},\,\mathbb{D}}}(\alpha-\beta)=\hat\chi_{n}(\alpha-\beta)\hat\omega_n(\alpha-\beta).
	\end{equation}
	
	$\rm(a)$ If $i_1=i_2$, then $\lambda_1\ne\lambda_2$ and $\alpha-\beta=q_n^{k_n+1}L_n(\lambda_1-\lambda_2)\not\in q_n^{k_n}L_n(\mathbb{Z}\setminus q_n\mathbb{Z})=\mathcal{Z}(\hat\chi_n)$. By \eqref{eq3.6-1}, we get $\hat\omega_n(\alpha-\beta)=0$.
	
	$\rm(b)$ If $i_1\ne i_2$, we get $0<|i_1-i_2|\leqslant q_n^{k_n}L_n-1$, which yields $(q_n^{k_n}L_n)\nmid(\alpha-\beta)$. This means that $\alpha-\beta\not\in q_n^{k_n}L_n(\mathbb{Z}\setminus q_n\mathbb{Z})=\mathcal{Z}(\hat\chi_n)$. By \eqref{eq3.6-1}, we get $\hat\omega_n(\alpha-\beta)=0$.
	
	By the arbitrariness of $\alpha$ and $\beta$, we see that $E(\Gamma_{j_0,\,j_1,\cdots,j_{q_n^{k_n}L_n-1}})$ is an orthogonal family of $L^2(\omega_n)$. The proof is completed.
\end{pf}

\begin{lemma}\label{le3.3}
Assume that the pair $(\mathbb{P},\,\mathbb{D})$ satisfies \eqref{eq2.12-3}. Let $\mu_{\mathbb{P},\,\mathbb{D}}$ be the measure defined by \eqref{eq1.1}. If $\Lambda$ is a spectrum of $\mu_{\mathbb{P},\,\mathbb{D}}$. Then, the set $\Gamma_{j_0,\,j_1,\cdots,j_{q_n^{k_n}L_n-1}}$ is a spectrum of $\omega_n$ for any $0\leqslant j_0,\,j_1,\,\cdots,\,j_{q_n^{k_n}L_n-1}\leqslant q_n-1$, where $\omega_n$ is defined as in \eqref{eq4.9} and $\Gamma_{j_0,\,j_1,\cdots,j_{q_n^{k_n}L_n-1}}$ is defined in Lemma \ref{le3.2}.
\end{lemma}

\begin{pf}
Since $\Lambda$ is a spectrum of $\mu_{\mathbb{P},\,\mathbb{D}}$ with $0\in\Lambda$. By Lemma \ref{le2.2} and \eqref{eq3.2}, we have
\begin{equation}\label{eq3.5-2}
	1=Q_\Lambda(x)=\sum_{i=0}^{q_n^{k_n}L_n-1}\sum_{j=0}^{q_n-1}\sum_{\lambda\in\Lambda_{i,\,j}^{(n)}}\Big|\widehat{\mu_{\mathbb{P},\,\mathbb{D}}}\Big(x+i+q_n^{k_n}L_nj+q_n^{k_n+1}L_n\lambda\Big)\Big|^2
\end{equation}
for any $ x\in\mathbb{R}$, where $\sum_{\lambda\in i+q_n^{k_n}L_nj+q_n^{k_n+1}L_n\Lambda_{i,\,j}^{(n)}}|\widehat{\mu_{\mathbb{P},\,\mathbb{D}}}(x+\frac{p_1}{dq}\lambda)|^2=0$ when $\Lambda_{i,\,j}^{(n)}=\emptyset$. By \eqref{eq1.3} and \eqref{eq4.8}, we have $\hat\chi_n(\cdot)=m_{d_n^{-1}D_n}(\frac{1}{q_n^{k_n+1}L_n}~\cdot)$ for each $n\geqslant1$. Hence, according to \eqref{eq2.2}, \eqref{eq2.5}, \eqref{eq3.2} and \eqref{eq3.5-2}, we get
\begin{equation*}
	\begin{split}
		1=&\sum_{i=0}^{q_n^{k_n}L_n-1}\sum_{j=0}^{q_n-1}\sum_{\lambda\in\Lambda_{i,\,j}^{(n)}}\Big|m_{d_n^{-1}D_n}\Big(\frac{1}{q_n^{k_n+1}L_n}x+\frac{1}{q_n^{k_n+1}L_n}\big(i+q_n^{k_n}L_nj+q_n^{k_n+1}L_n\lambda\big)\Big)\Big|^2\\
		&\hspace{18em}\cdot\Big|\omega_{n}\Big(x+i+q_n^{k_n}L_nj+q_n^{k_n+1}L_n\lambda\Big)\Big|^2\\
		=&\sum_{i=0}^{q_n^{k_n}L_n-1}\sum_{j=0}^{q_n-1}\Big|m_{d_n^{-1}D_n}\Big(\frac{1}{q_n^{k_n+1}L_n}x+\frac{1}{q_n^{k_n+1}L_n}\big(i+q_n^{k_n}L_nj\big)\Big)\Big|^2\\
		&\hspace{16em}\cdot\sum_{\lambda\in\Lambda_{i,\,j}^{(n)}}\Big|\omega_{n}\Big(x+i+q_n^{k_n}L_nj+q_n^{k_n+1}L_n\lambda\Big)\Big|^2
	\end{split}
\end{equation*}
for any $ x\in\mathbb{R}$ and $n\geqslant1$. For $0\leqslant i\leqslant q_n^{k_n}L_n-1$ and $0\leqslant j\leqslant q_n-1$, we write $p_{i,j}:=\big|m_{d_n^{-1}D_n}\big(\frac{1}{q_n^{k_n+1}L_n}x+\frac{1}{q_n^{k_n+1}L_n}\big(i+q_n^{k_n}L_nj\big)\big)\big|^2$ and $x_{i,j}:=\big|\omega_{n}\big(x+i+q_n^{k_n}L_nj+q_n^{k_n+1}L_n\lambda\big)\big|^2$. For any $i\in\{0,\,1,\cdots,\,q_n^{k_n}L_n-1\}$, define
\begin{equation}\label{eq3.5-1}
	\Delta_{n,i}:=\{i+q_n^{k_n}L_nj:\,0\leqslant j\leqslant q_n-1\}.
\end{equation}
By \eqref{eq2.4}, we have $\frac{a-b}{q_n^{k_n+1}L_n}\in\mathcal{Z}(m_{d_n^{-1}D_n})$ for any $a\neq b\in\Delta_{n,i}$. Since $\#D_n=q_n=\#\Delta_{n,i}$. It follows from Lemma \ref{le2.4} \rm(\romannumeral2) that $(d_n^{-1}D_n,\,\frac{1}{q_n^{k_n+1}L_n}\Delta_{n,i})$ is a compatible pair for each $i\in\{0,\,1,\cdots,\,q_n^{k_n}L_n-1\}$. Thus, by Lemma \ref{le2.4} \rm(\romannumeral3), we get $\sum_{j=0}^{q_n-1}p_{i,j}=1$ for each $i\in\{0,\,1,\cdots,\,q_n^{k_n}L_n-1\}$. Moreover, for any $j_0,\,j_1,\,\cdots,\,j_{q_n^{k_n}L_n-1}\in\{0,\,1,\cdots,\,q_n-1\}$, Lemma \ref{le3.2} show that $E(\Gamma_{j_0,\,j_1,\cdots,j_{q_n^{k_n}L_n-1}})$ is either an empty set or an orthogonal family of $L^2(\omega_n)$. By Lemma \ref{le2.2}, we get that $\sum_{i=0}^{q_n^{k_n}L_n-1}x_{i,j_i}\leqslant1$ for any $j_0,\,j_1,\,\cdots,\,j_{q_n^{k_n}L_n-1}\in\{0,\,1,\cdots,\,q_n-1\}$. Therefore, it follows that $\sum_{i=0}^{q_n^{k_n}L_n-1}\max\{x_{i,0},\,x_{i,1},\,\cdots,\,x_{i,q_n-1}\}\leqslant1$.

Choose an irrational number $x\in\mathbb{R}\setminus\mathbb{Z}$. By \eqref{eq2.4}, we get $p_{i,j}>0$. Thus, $p_{i,j}$ and $x_{i,j}$ satisfy the assumption in Lemma \ref{le2.6}. This means that
\begin{equation}\label{eq3.5}
	x_{i,0}=x_{i,1}=\cdots=x_{i,q_n-1},\qquad 0\leqslant i\leqslant q_n^{k_n}L_n-1,
\end{equation}
and
\begin{equation}\label{eq3.6}
	1=\sum_{i=0}^{q_n^{k_n}L_n-1}x_{i,0}.
\end{equation}
Noting that $\hat\omega_{n}$ is a continuous functions on $\mathbb{R}$, then the  equations \eqref{eq3.5} and \eqref{eq3.6} hold for all $x\in\mathbb{R}$. Furthermore, \eqref{eq3.5} and \eqref{eq3.6} implies that
\begin{equation*}
	\sum_{i=0}^{q_n^{k_n}L_n-1}x_{i,j_i}\equiv1,\qquad \forall \, x\in\mathbb{R},\,j_0,\,j_1,\,\cdots,\,j_{q_n^{k_n}L_n-1}\in\{0,\,1,\cdots,\,q_n-1\}.
\end{equation*}
Combining the above with Lemma \ref{le2.2}, the conclusion hold.	
\end{pf}

Assume that the pair $(\mathbb{P},\,\mathbb{D})$ satisfies \eqref{eq2.12-3}. Let $\mu_{\mathbb{P},\,\mathbb{D}}$ be the measure defined by \eqref{eq1.1}, and let $\Lambda$ be a spectrum of $\mu_{\mathbb{P},\,\mathbb{D}}$. For any $n\geqslant1$ and $0\leqslant i\leqslant L_n-1$, define $E_{n,i}:=\{i+q_n^{k_n}L_nj:\,\Lambda_{i,\,j}^{(n)}\ne\emptyset,\,0\leqslant j\leqslant q_n-1\}$. We see that
\begin{equation}\label{eq4.14}
	\mbox{either}~~E_{n,i}=\Delta_{n,i} \mbox{~~or~~} E_{n,i}=\emptyset
\end{equation}
by using \eqref{eq3.5}, where $\Delta_{n,i}$ is defined in \eqref{eq3.5-1}. If $0\in\Lambda$, we have $0\in\Lambda_{0,\,0}^{(n)}$. Thus,
\begin{equation}\label{eq4.1-1}
	E_{n,0}=\Delta_{n,0}=q_n^{k_n}L_n\{0,1,\cdots,q_n-1\}
\end{equation}
and so $q_n^{k_n}L_nj+q_n^{k_n+1}L_n\Lambda_{0,\,j}^{(n)}\neq\emptyset$ for all $0\leqslant j\leqslant q_n-1$ when $0\in\Lambda$. Since $q_n\geqslant2$, we may  take $j=1$. It follows from \eqref{eq2.6} and \eqref{eq3.2} that $q_n^{k_n}L_n+q_n^{k_n+1}L_n\Lambda_{0,\,1}^{(n)}$ is contained in both $\mathcal{Z}(\hat\chi_n)$ and $\Lambda$. Hence $\mathcal{Z}(\hat\chi_{n})\cap\Lambda\neq\emptyset$. We thus obtain the following corollary.

\begin{coro}\label{co3.6}
Assume that the pair $(\mathbb{P},\,\mathbb{D})$ satisfies \eqref{eq2.12-3}. Let $\mu_{\mathbb{P},\,\mathbb{D}}$ be the measure defined by \eqref{eq1.1}. If $\Lambda$ is a spectrum of $\mu_{\mathbb{P},\,\mathbb{D}}$ with $0\in\Lambda$, then $\mathcal{Z}(\hat\chi_{n})\cap\Lambda\neq\emptyset$ for any $n\geqslant1$.
\end{coro}

\begin{lemma}\label{le4.1}
Assume that the pair $(\mathbb{P},\,\mathbb{D})$ satisfies \eqref{eq2.12-3}. Let $\mu_{\mathbb{P},\,\mathbb{D}}$, defined by \eqref{eq1.1}, be a spectral measure. If $\Lambda$ is a spectrum of $\mu_{\mathbb{P},\,\mathbb{D}}$ with $0\in\Lambda$, then there exists an integer $n\geqslant1$ such that $k_{q_n,n}=\min\{k_{q_n,i}:\,i\in X_{q_n}\}$ and
\begin{equation}\label{eq4.2}
	q_n^{-k_{q_n,n}}\Lambda\subset\mathbb{Z}.
\end{equation}	
\end{lemma}

\begin{pf}
We can label the elements of $X_2$ as $s_{1},\ s_{2},\ \cdots,\ s_{m},\ \cdots$ in such a way that
\begin{equation}\label{eq4.3}
	\left\{\begin{array}{l}
		X_2=\{s_{1},\ s_{2},\ \cdots,\ s_{m},\ \cdots\},\\
		\mbox{either } k_{2,s_{j}}<k_{2,s_{j+1}}\mbox{ or, both }k_{2,s_{j}}=k_{2,s_{j+1}}\mbox{ and }s_{j}<s_{j+1}\mbox{ for all }j\ge1.
	\end{array}\right.
\end{equation}
Similarly, reorder the elements of the set $X_3$  such that
\begin{equation}\label{eq4.4}
	\left\{\begin{array}{l}
		X_3=\{t_{1},\ t_{2},\ \cdots,\ t_{m},\ \cdots\},\\
		\mbox{either } k_{3,t_{j}}<k_{3,t_{j+1}}\mbox{ or, both } k_{3,t_{j}}=k_{3,t_{j+1}}\mbox{ and }t_{j}<t_{j+1}\mbox{ for all }j\ge1.
	\end{array}\right.
\end{equation}	
	
Since $\Lambda$ is a spectrum of $\mu_{\mathbb{P},\,\mathbb{D}}$ with $0\in\Lambda$, we have $\Lambda\subset\mathcal{Z}(\widehat{\mu_{\mathbb{P},\,\mathbb{D}}})\cup\{0\}$. By \eqref{eq2.7} and \eqref{eq2.12-3}, we have $\mathcal{Z}(\widehat{\mu_{\mathbb{P},\,\mathbb{D}}})\subset(2^{k_{2,s_{1}}}\mathbb{Z}\cup3^{k_{3,t_{1}}}\mathbb{Z})$. Therefore,
\begin{equation}\label{eq4.5}
	\Lambda\subset\mathcal{Z}(\widehat{\mu_{\mathbb{P},\,\mathbb{D}}})\cup\{0\}\subset (2^{k_{2,s_{1}}}\mathbb{Z}\cup3^{k_{3,t_{1}}}\mathbb{Z}).
\end{equation}
If $\Lambda\not\subset2^{k_{2,s_{1}}}\mathbb{Z}$ and $\Lambda\not\subset3^{k_{3,t_{1}}}\mathbb{Z}$, \eqref{eq4.5} shows that there exist $\lambda$, $\gamma\in\Lambda\setminus\{0\}$ such that
$\lambda\in (3^{k_{3,t_{1}}}\mathbb{Z})\setminus(2^{k_{2,s_{1}}}\mathbb{Z})$ and $\gamma\in (2^{k_{2,s_{1}}}\mathbb{Z})\setminus(3^{k_{3,t_{1}}}\mathbb{Z})$. Hence $\lambda-\gamma\notin(2^{k_{2,s_{1}}}\mathbb{Z}\cup3^{k_{3,t_{1}}}\mathbb{Z})$, so \eqref{eq4.5} shows $\lambda-\gamma\notin\mathcal{Z}(\widehat{\mu_{\mathbb{P},\,\mathbb{D}}})$, which contradicts to our assumption that $\Lambda$ is a spectrum of $\mu_{\mathbb{P},\,\mathbb{D}}$.
The lemma is proven.
\end{pf}

With reference to the above notation, we have the following Theorem:

\begin{theorem}\label{th4.6}
Assume that the pair $(\mathbb{P},\,\mathbb{D})$ satisfies \eqref{eq2.12-3}. Let $\mu_{\mathbb{P},\,\mathbb{D}}$, defined by \eqref{eq1.1}, be a spectral measure. If $\Lambda$ is a spectrum of $\mu_{\mathbb{P},\,\mathbb{D}}$ with $0\in\Lambda$. Then, there exists an integer $n\in\{s_{1},\,t_{1}\}$ satisfies \eqref{eq4.2}. Furthermore, the following statements hold.
\begin{enumerate}
	\item[\rm(\romannumeral1)] $\Lambda_n:=[(q_n^{1+k_{q_n,n}}{\mathbb Z})\cap\Lambda]$ is a spectrum of the measure $\omega_{n}$,
	\item[\rm(\romannumeral2)] $k_{q_n,j}>k_{q_n,n}$  for every $j\in (X_{q_n}\setminus\{n\})$,
\end{enumerate}
where $\omega_{n}$ is defined in \eqref{eq4.9} and, the sequences $\{s_{j}\}_{j=1}^{\infty}$ and $\{t_{j}\}_{j=1}^{\infty}$ are defined in \eqref{eq4.3} and \eqref{eq4.4}, respectively.	
\end{theorem}

\begin{pf}
\rm(\romannumeral1) Lemma \ref{le4.1} shows that there exists an integer $n\in\{s_{1},\,t_{1}\}$ such that
\begin{equation}\label{eq4.12-1}
	\Lambda\subset q_n^{k_n}\mathbb{Z}.
\end{equation}
By \eqref{eq3.1} and \eqref{eq4.12-1}, we see that $\Lambda_{i,\,j}^{(n)}=\emptyset$ when $i\not\in q_n^{k_n}\mathbb{Z}$, where $0\leqslant i\leqslant q_n^{k_n}L_n-1$. This means that
\begin{equation}\label{eq4.12-4}
	i+q_n^{k_n}L_nj+q_n^{k_n+1}L_n\Lambda_{i,\,j}^{(n)}=\emptyset
\end{equation}
when $i\not\in q_n^{k_n}\mathbb{Z}$, where $0\leqslant i\leqslant q_n^{k_n}L_n-1$.
Therefore, the decomposition \eqref{eq3.2} of $\Lambda$ becomes
\begin{equation}\label{eq4.7}
	\Lambda=\bigcup_{i=0}^{L_n-1}\bigcup_{j=0}^{q_n-1}\Big(q_n^{k_n}i+q_n^{k_n}L_nj+q_n^{k_n+1}L_n\Lambda_{q_n^{k_n}i,\,j}^{(n)}\Big),
\end{equation}
where $q_n^{k_n}i+q_n^{k_n}L_nj+q_n^{k_n+1}L_n\Lambda_{q_n^{k_n}i,\,j}^{(n)}=\emptyset$ when $\Lambda_{q_n^{k_n}i,\,j}^{(n)}=\emptyset$. For each $i\in\{0,\,1,\cdots,L_n-1\}$, we choose $j_i$ as follows: If $n\in X_2$, then
\begin{equation}\label{eq4.12}
	j_i=\left\{\begin{array}{ll}
		0, &\mbox{ if } i\in2\mathbb{Z}, \\
		1, &\mbox{ if } i\in2\mathbb{Z}+1.
	\end{array}\right.
\end{equation}
If $n\in X_3$, then
\begin{equation}\label{eq4.13}
	j_i= \left\{\begin{array}{ll}
		0, &\mbox{ if } i\in3\mathbb{Z}, \\
		1, &\mbox{ if } i\in3\mathbb{Z}+1 \mbox{ and } L_{n}\in3\mathbb{Z}-1,\\
		2, &\mbox{ if } i\in3\mathbb{Z}+1 \mbox{ and } L_{n}\in3\mathbb{Z}+1,\\
		1, &\mbox{ if } i\in3\mathbb{Z}-1 \mbox{ and } L_{n}\in3\mathbb{Z}+1,\\
		2, &\mbox{ if } i\in3\mathbb{Z}-1 \mbox{ and } L_{n}\in3\mathbb{Z}-1.
	\end{array}\right.
\end{equation}
By \eqref{eq4.12}, \eqref{eq4.13} and the fact $L_n\in(\mathbb{Z}\setminus q_n\mathbb{Z})$, we get
\begin{equation}\label{eq4.12-3}
	i+L_nj_i\in q_n\mathbb{Z},\qquad\forall \, i\in\{0,\,1,\,\cdots,\,L_n-1\}
\end{equation}
Define $\Gamma_{j_0,j_1,\cdots,j_{L_{n}-1}}:=\cup_{i=0}^{L_{n}-1}\big(q_n^{k_n}i+q_n^{k_n}L_nj_i+q_n^{k_n+1}L_n\Lambda_{q_n^{k_n}i,\,j_i}^{(n)}\big)$. Since $j_0=0$ and $0\in\Lambda_{0,\,0}^{(n)}$, we have $0\in\Gamma_{j_0,j_1,\cdots,j_{L_{n}-1}}$. Moreover, it follows from \eqref{eq4.12-3} that $\Gamma_{j_0,j_1,\cdots,j_{L_{n}-1}}\subset q_n^{k_n+1}\mathbb{Z}$. In conclusion, we have
\begin{equation}\label{eq4.12-5}
	0\in\Gamma_{j_0,j_1,\cdots,j_{L_{n}-1}}\subset [(q_n^{k_n+1}\mathbb{Z})\cap\Lambda].
\end{equation}	
	
For any $\lambda\in[(q_n^{k_n+1}\mathbb{Z})\cap\Lambda]$, \eqref{eq4.7} shows that there exist $i_{\lambda}\in\{0,\,1,\cdots,L_n-1\}$ and $j_{\lambda}\in\{0,\,1,\cdots,q_n-1\}$ such that $\lambda\in q_n^{k_n}i_{\lambda}+q_n^{k_n}L_nj_{\lambda}+q_n^{k_n+1}L_n\Lambda_{q_n^{k_n}i_{\lambda},\,j_{\lambda}}^{(n)}$ and $i_{\lambda}+L_nj_{\lambda}\in q_n\mathbb{Z}$. Since $L_n\in(\mathbb{Z}\setminus q_n\mathbb{Z})$ and $i_{\lambda}+L_nj_{\lambda}\in q_n\mathbb{Z}$, the pair $(i_{\lambda},\,j_{\lambda})$ must satisfy \eqref{eq4.12} when $n\in X_2$ and \eqref{eq4.13} when $n\in X_3$. It follows that $\lambda\in\Gamma_{j_0,j_1,\cdots,j_{L_{n}-1}}$. According to the arbitrariness of $\lambda$, we have $[(q_n^{k_n+1}\mathbb{Z})\cap\Lambda]\subset\Gamma_{j_0,j_1,\cdots,j_{L_{n}-1}}$. Combining this with \eqref{eq4.12-5}, we have $\Gamma_{j_0,j_1,\cdots,j_{L_{n}-1}}=[(q_n^{k_n+1}\mathbb{Z})\cap\Lambda]$.

Moreover, it follows from Lemma \ref{le3.3} and \eqref{eq4.12-4} that $\Gamma_{j_0,j_1,\cdots,j_{L_{n}-1}}$ is a spectrum of $\omega_{n}$. Therefore, statement \rm(\romannumeral1) is proved.

\rm(\romannumeral2) Since $\Gamma_{j_0,j_1,\cdots,j_{L_{n}-1}}$ is a spectrum of $\omega_{n}$ and $\omega_{n}=\omega_{n,\,j}\bigast\chi_j$ for each $j\in(X_{q_n}\setminus\{n\})$. Hence, by \eqref{eq2.8} and Lemma \ref{le2.8}, we obtain that for each $j\in(X_{q_n}\setminus\{n\})$, there exist distinct integers $\lambda_{j,1}$, $\lambda_{j,2}\in\Gamma_{j_0,j_1,\cdots,j_{L_{n}-1}}$ such that
\begin{equation}\label{eq4.12-6}
	\lambda_{j,1}-\lambda_{j,2}\in\mathcal{Z}(\hat\chi_j).
\end{equation}
Moreover, by \eqref{eq4.12-5}, we get $\lambda_{j,1}-\lambda_{j,2}\in q_n^{k_n+1}\mathbb{Z}$. Hence, by \eqref{eq2.6} and \eqref{eq4.12-6}, we obtain that $k_j\geqslant k_n+1$ for each $j\in(X_{q_n}\setminus\{n\})$. This finishes the proof.	
\end{pf}

Recall that in \eqref{eq4.3}, the elements of $X_2$ are arranged as
$X_2=\{s_{1},\ s_{2},\ \cdots,\ s_{m},\ \cdots\}$ such that
\begin{equation}\label{eq4.15}
	k_{2,s_{1}}\leqslant k_{2,s_{2}}\leqslant k_{2,s_{3}}\leqslant\cdots\leqslant k_{2,s_{n}}\leqslant\cdots,
\end{equation}
and for all $j\geqslant1$, if $k_{2,s_{j}}=k_{2,s_{j+1}}$, then $s_{j}<s_{j+1}$. Similarly, in \eqref{eq4.4}, we index the elements of $X_3$ as $t_{1},\ t_{2},\ \cdots,\ t_{m},\ \cdots$ such that $X_3=\{t_{1},\ t_{2},\ \cdots,\ t_{m},\ \cdots\}$,
\begin{equation}\label{eq4.16}
	k_{3,t_{1}}\leqslant k_{3,t_{2}}\leqslant k_{3,t_{3}}\leqslant\cdots\leqslant k_{3,t_{n}}\leqslant\cdots,
\end{equation}
and for all $j\geqslant1$, if $k_{3,t_{j}}=k_{3,t_{j+1}}$, then $t_{j}<t_{j+1}$.

In reference \cite{Deng2023} (or \cite{Xiong2023}), the authors assume $q_1=q_2=\cdots=q_n=\cdots =2$ (or $=3$). This implies that one of $X_3$ and $X_2$ is empty. It follows that in Theorem \ref{th4.6}, $n$ can only take the value $s_1$ (or $t_1$). Hence, we have $k_{2,s_1}<k_{2,s_2}$ (or $k_{3,t_1}<k_{3,t_2}$) and $\omega_{s_1}$ (or $\omega_{t_1}$) is a spectral measure. Replace $\mu_{\mathbb{P},\,\mathbb{D}}$ in Theorem \ref{th4.6} by $\omega_{s_1}$ (or $\omega_{t_1}$), we have $k_{2,s_2}<k_{2,s_3}$ (or $k_{3,t_2}<k_{3,t_3}$) and $\omega_{s_1,s_2}$ (or $\omega_{t_1,t_2}$) is a spectral measure. By induction, one easily checks that $k_{2,s_1}<k_{2,s_2}<\cdots<k_{2,s_m}<\cdots$ (or  $k_{3,t_1}<k_{3,t_2}<\cdots<k_{3,t_m}<\cdots$). The authors of \cite{Deng2023} (or \cite{Xiong2023}) thus obtained the result presented in their work.

From the above discussion, we see that under the assumption $q_1=q_2=\cdots=q_n=\cdots$ made in reference \cite{Deng2023} (or \cite{Xiong2023}), the decomposition path of $X_2$ (or $X_3$) must be exactly $s_1\to s_2\to\cdots\to s_m\to\cdots$ (or $t_1\to t_2\to\cdots\to  t_m\to\cdots$). In the present work, however, the condition $q_1=q_2=\cdots=q_n=\cdots$ may fail; that is, both sets $X_2$ and $X_3$ can be nonempty. Moreover, Lemma \ref{le4.1} only guarantees the existence of some $n$ that satisfies \eqref{eq4.2}, but does not determine whether $n=s_1$ or $n=t_1$. Consequently, there can be various possible decomposition paths for set $X_2\cup X_3$, and some of these paths may not traverse $X_2\cup X_3$. For instance, when $\#X_2=\infty$ and $\#X_3=\infty$, the paths $s_1\to s_2\to\cdots\to s_m\to\cdots$, $t_1\to t_2\to\cdots\to t_m\to\cdots$ or $s_1\to s_2\to t_1\to t_2\to\cdots\to t_m\to\cdots$ do not traverse $X_2\cup X_3$. To prove statement \rm(\romannumeral1) of the necessity in Theorem \ref{th1.3}, we need to find a path that does traverse $X_2\cup X_3$. For this purpose, we present the following two theorems.

\begin{theorem}\label{th4.7}
Assume that the pair $(\mathbb{P},\,\mathbb{D})$ satisfies \eqref{eq2.12-3}. Let $\mu_{\mathbb{P},\,\mathbb{D}}$ be the measure defined by \eqref{eq1.1}, and let $\Lambda$ be a spectrum of $\mu_{\mathbb{P},\,\mathbb{D}}$ with $0\in\Lambda$. If $2^{-k_{2,s_1}}\Lambda\subset\mathbb{Z}$, then there exist an integer $m:=m(s_1)\geqslant1$ and a subset $\Gamma\subset3^{k_{3,t_{1}}}\mathbb{Z}$ with $0\in\Gamma$ such that the following statements hold.
\begin{enumerate}
	\item[\rm(\romannumeral1)] $k_{2,s_1}<k_{2,s_2}<\cdots<k_{2,s_m}<k_{2,s_{m+1}}$,
	\item[\rm(\romannumeral2)] $\Gamma$ is a spectrum of $\omega_{s_1,s_2,\cdots,s_m}$,
\end{enumerate}
where the sequences $\{s_{j}\}_{j=1}^{\infty}$ and $\{t_{j}\}_{j=1}^{\infty}$ are defined in \eqref{eq4.15} and \eqref{eq4.16}, respectively.	
\end{theorem}

\begin{pf}
By \eqref{eq2.8} and Lemma \ref{le2.8}, there exist $\lambda_1$, $\lambda_2\in\Lambda$ with $\lambda_1\neq\lambda_2$ such that
\begin{equation*}
	\lambda_1-\lambda_2\in[\mathcal{Z}(\hat\chi_{t_{1}})\setminus\mathcal{Z}(\hat\omega_{t_{1}})].
\end{equation*}
Note that $0\in\Lambda-\lambda_2$ and $\Lambda-\lambda_2$ is also a spectrum of $ \mu_{\mathbb{P},\,\mathbb{D}}$, without loss of generality, we can assume $\lambda_2=0$ and fix the $\lambda_1$ in the following argument.

By Theorem \ref{th4.6} \rm(\romannumeral1), we obtain that $(2^{1+k_{2,s_1}}{\mathbb Z})\cap\Lambda$ is a spectrum of $\omega_{s_1}$, so \eqref{eq2.8} shows $k_{2,s_2}>k_{2,s_1}$.

Since we have assumed $\lambda_2=0$, \eqref{eq3.2} shows $\lambda_1\in i_1+q_{s_1}^{k_{s_1}}L_{s_1}j_1+q_{s_1}^{k_{s_1}+1}L_{s_1}\Lambda_{i_1,\,j_1}^{(s_1)}$ for some $0\leqslant i_1\leqslant q_n^{k_n}L_n-1$ and $0\leqslant j_1\leqslant q_n-1$. Since $\lambda_1\in[\mathcal{Z}(\hat\chi_{t_{1}})\setminus\mathcal{Z}(\hat\omega_{t_{1}})]\subset[\mathcal{Z}(\hat\chi_{t_{1}})\setminus\mathcal{Z}(\hat\chi_{s_{1}})]$, we see that either $i_1\neq0$, or $i_1=0$ and $j_1=0$. (Otherwise, if $i_1=0$ and $j_1\neq0$, we see that $\lambda_1\in q_{s_1}^{k_{s_1}}L_{s_1}(\mathbb{Z}\setminus q\mathbb{Z})=\mathcal{Z}(\hat\chi_{s_{1}})$, which contradicts $\lambda_1\not\in\mathcal{Z}(\hat\chi_{s_{1}})$.) Hence, Lemma \ref{le3.3} shows that there is a spectrum $\Omega_{s_1}$ of $\omega_{s_1}$ such that $\{0,\,\lambda_1\}\subset\Omega_{s_1}$.

If $\Omega_{s_1}\subset3^{k_{3,t_{1}}}\mathbb{Z}$, then our conclusions \rm(\romannumeral1) and \rm(\romannumeral2) hold with $m=1$ and $\Gamma=\Omega_{s_1}$. Otherwise, applying Lemma \ref{le4.1} to $\omega_{s_1}$ shows $\Omega_{s_1}\subset2^{k_{2,s_{2}}}\mathbb{Z}$. Then, applying the above argument (replace the spectral pair $(\mu_{\mathbb{P},\,\mathbb{D}},\,\Lambda)$ by  $(\omega_{s_1},\Omega_{s_1})$) shows that $k_{2,s_{2}}<k_{2,s_{3}}$ and there is a spectrum $\Omega_{s_1,s_2}$ of $\omega_{s_1,s_2}$ such that $\{0,\,\lambda_1\}\subset\Omega_{s_1,s_2}$.

If $\Omega_{s_1,s_2}\subset3^{k_{3,t_{1}}}\mathbb{Z}$, then our conclusions \rm(\romannumeral1) and \rm(\romannumeral2) hold with $m=2$ and $\Gamma=\Omega_{s_1,s_2}$. Otherwise, applying the above argument (replace the spectral pair $(\mu_{\mathbb{P},\,\mathbb{D}},\,\Lambda)$ by  $(\omega_{s_1,s_2},\Omega_{s_1,s_2})$) shows that $\Omega_{s_1,s_2}\subset2^{k_{2,s_{3}}}\mathbb{Z}$, $k_{2,s_{3}}<k_{2,s_{4}}$ and there is a spectrum $\Omega_{s_1,s_2,s_3}$ of $\omega_{s_1,s_2,s_3}$ such that $\{0,\,\lambda_1\}\subset\Omega_{s_1,s_2,s_3}$.

Since $\lambda_1\not\in2^l\mathbb{Z}$ for some $l>k_{2,s_{1}}$, we see that the integer $m$ and the set $\Gamma$ exist.	
\end{pf}

\begin{theorem}\label{th4.7+1}
Assume that the pair $(\mathbb{P},\,\mathbb{D})$ satisfies \eqref{eq2.12-3}. Let $\mu_{\mathbb{P},\,\mathbb{D}}$ be the measure defined by \eqref{eq1.1}, and let $\Lambda$ be a spectrum of $\mu_{\mathbb{P},\,\mathbb{D}}$ with $0\in\Lambda$. If $3^{-k_{3,t_1}}\Lambda\subset\mathbb{Z}$, then there exist an integer $m:=m(t_1)\geqslant1$ and a subset $\Gamma\subset 2^{k_{2,s_1}}\mathbb{Z}$ with $0\in\Gamma$ such that the following statements hold.
\begin{enumerate}
	\item[\rm(\romannumeral1)] $k_{3,t_1}<k_{3,t_2}<\cdots<k_{3,t_m}<k_{3,t_{m+1}}$,
	\item[\rm(\romannumeral2)] $\Gamma$ is a spectrum of $\omega_{t_1,t_2,\cdots,t_m}$,
\end{enumerate}
where the sequences $\{s_{j}\}_{j=1}^{\infty}$ and $\{t_{j}\}_{j=1}^{\infty}$ are defined in \eqref{eq4.15} and \eqref{eq4.16}, respectively.	
\end{theorem}

\begin{pf}
	The proof of Theorem \ref{th4.7+1} is almost the same as Theorem \ref{th4.7}.
\end{pf}

\begin{theorem}\label{th4.7+2}
Assume that the pair $(\mathbb{P},\,\mathbb{D})$ satisfies \eqref{eq2.12-3}. Let $\mu_{\mathbb{P},\,\mathbb{D}}$ be the measure defined by \eqref{eq1.1}. If $\mu_{\mathbb{P},\,\mathbb{D}}$ is a spectral measure, then the sequences in \eqref{eq4.15} and \eqref{eq4.16} are both strictly increasing, i.e.,
\begin{equation*}
	k_{2,s_{1}}<k_{2,s_{2}}<k_{2,s_{3}}\cdots<k_{2,s_{n}}<\cdots\quad\mbox{and}\quad k_{3,t_{1}}<k_{3,t_{2}}<k_{3,t_{3}}\cdots<k_{3,t_{n}}<\cdots.
\end{equation*}	
\end{theorem}

\begin{pf}
Since $\mu_{\mathbb{P},\,\mathbb{D}}$ is a spectral measure, there is a subset $\Lambda\subset\mathbb{Z}$ such that $0\in\Lambda$ and $\Lambda$ is a spectrum of $\mu_{\mathbb{P},\,\mathbb{D}}$. Lemma \ref{le4.1} shows that at least one of $2^{-k_{2,s_{1}}}\Lambda\subset\mathbb{Z}$ and $3^{-k_{3,t_{1}}}\Lambda\subset\mathbb{Z}$ holds.

Without loss of generality, we assume $2^{-k_{2,s_{1}}}\Lambda\subset\mathbb{Z}$.

By Theorem \ref{th4.7}, one can find an integer $m_1:=m(s_1)\geqslant1$ and a subset $\Gamma_1\subset\mathbb{Z}$ with $0\in\Gamma_1$ such that
\begin{numcases}{}
	k_{2,s_1}<k_{2,s_2}<\cdots<k_{2,s_{m_1}}<k_{2,s_{m_1+1}},\label{eq4.47}\\
	\Gamma_1~\mbox{is~a~spectrum~of}~\omega_{s_1,s_2,\cdots,s_{m_1}},\label{eq4.48}\\
	3^{-k_{3,t_{1}}}\Gamma_1\subset\mathbb{Z}\label{eq4.49}.
\end{numcases}
From \eqref{eq4.48} and \eqref{eq4.49}, it follows that the pair $(\omega_{s_1,s_2,\cdots,s_{m_1}},\,\Gamma_1)$ satisfies the conditions of Theorem \ref{th4.7+1}. Replace the spectral pair $(\mu_{\mathbb{P},\,\mathbb{D}},\,\Lambda)$ in Theorem \ref{th4.7+1} by $(\omega_{s_1,s_2,\cdots,s_{m_1}},\,\Gamma_1)$, there is an integer $m_2:=m(t_1)\geqslant1$ and a subset $\Gamma_2\subset\mathbb{Z}$ with $0\in\Gamma_2$ such that
\begin{numcases}{}
	k_{3,t_1}<k_{3,t_2}<\cdots<k_{3,t_{m_2}}<k_{3,t_{m_2+1}},\label{eq4.50}\\
	\Gamma_2~\mbox{is~a~spectrum~of}~\omega_{s_1,s_2,\cdots,s_{m_1},\,t_1,t_2,\cdots,t_{m_2}},\label{eq4.51}\\
	2^{-k_{2,s_{m_1+1}}}\Gamma_2\subset\mathbb{Z}\label{eq4.52}.
\end{numcases}
By \eqref{eq4.51} and \eqref{eq4.52}, we get that the pair $(\omega_{s_1,s_2,\cdots,s_{m_1},\,t_1,t_2,\cdots,t_{m_2}},\,\Gamma_2)$  satisfies the conditions of Theorem \ref{th4.7}. Replace the pair $(\mu_{\mathbb{P},\,\mathbb{D}},\,\Lambda)$ in Theorem \ref{th4.7} by $(\omega_{s_1,s_2,\cdots,s_{m_1},\,t_1,t_2,\cdots,t_{m_2}},\,\Gamma_2)$, there exist an integer $m_3:=m(s_{m_1+1})\geqslant1$ and a subset $\Gamma_3\subset\mathbb{Z}$ with $0\in\Gamma_3$ such that
\begin{numcases}{}
	k_{2,s_{m_1+1}}<k_{2,s_{m_1+2}}<\cdots<k_{2,s_{m_3+1}}<k_{2,s_{m_3+1}},\label{eq4.53}\\
	\Gamma_3~\mbox{is~a~spectrum~of}~\omega_{s_1,s_2,\cdots,s_{m_3},\,t_1,t_2,\cdots,t_{m_2}},\label{eq4.54}\\
	3^{-k_{3,t_{m_2+1}}}\Gamma_3\subset\mathbb{Z}\label{eq4.55}.
\end{numcases}
It follows from \eqref{eq4.54} that $(\omega_{s_1,s_2,\cdots,s_{m_3},\,t_1,t_2,\cdots,t_{m_2}},\,\Gamma_3)$ is a spectral pair. Replace the spectral pair $(\mu_{\mathbb{P},\,\mathbb{D}},\,\Lambda)$ in Theorem \ref{th4.7+1} by $(\omega_{s_1,s_2,\cdots,s_{m_3},\,t_1,t_2,\cdots,t_{m_2}},\,\Gamma_3)$. By \eqref{eq4.55} and Theorem \ref{th4.7+1}, we can find an integer $m_4:=m(t_{m_2+1})\geqslant1$ and a subset $\Gamma_4\subset\mathbb{Z}$ with $0\in\Gamma_4$ such that
\begin{numcases}{}
	k_{3,t_{m_2+1}}<k_{3,t_{m_2+2}}<\cdots<k_{3,t_{m_4}}<k_{3,t_{m_4+1}},\label{eq4.56}\\
	\Gamma_4~\mbox{is~a~spectrum~of}~\omega_{s_1,s_2,\cdots,s_{m_3},\,t_1,t_2,\cdots,t_{m_4}},\nonumber\\
	2^{-k_{2,s_{m_3+1}}}\Gamma_4\subset\mathbb{Z}\nonumber.
\end{numcases}
By continuing the procedure, we see that the conclusion of the lemma follows from \eqref{eq4.47}, \eqref{eq4.50}, \eqref{eq4.53} and \eqref{eq4.56}. The lemma was proved in this case.	
\end{pf}

{\bf Proof of the necessity Theorem \ref{th1.3}}. For any $n\in X_3$, by Lemma \ref{le2.8}, we have $\mathcal{Z}(\hat\chi_n)\ne\emptyset$. Combining this with Lemma \ref{eq2.1}, we obtain $\{\frac{a_n}{d_n},\,\frac{b_n}{d_n}\}\equiv\{1,\,-1\}(\mod~3)$.

Since $\{s_n\}_{n=1}^{\infty}=X_2$. It follows from Theorem \ref{th4.7+2} that
\begin{equation}\label{eq4.57}
	k_{2,i}\neq k_{2,j},\qquad\forall \, i\neq j\in X_2
\end{equation}
and
\begin{equation}\label{eq4.58}
	k_{3,i}\neq k_{3,j},\qquad\forall \, i\neq j\in X_3.
\end{equation}
We now see that, in the necessity part of Theorem \ref{th1.3}, formula \eqref{eq1.4} holds.

Next, we will prove that the \eqref{eq1.5} in the necessity part of Theorem \eqref{th1.3} is also true. Suppose on the contrary that there exist positive integers $m\in X_2$ and $n\in X_3$ such that
\begin{equation}\label{eq4.59}
	\nu_{2}(L_{n})>k_{2,m}~~\mbox{ and }~~\nu_{3}(L_{m})>k_{3,n}.
\end{equation}
It follows from \eqref{eq2.6} and \eqref{eq4.59} that
\begin{equation}\label{eq4.60}
	\mathcal{Z}(\hat\chi_{m})\cap\mathcal{Z}(\hat\chi_{n})=\emptyset.
\end{equation}

Since $\mu_{\mathbb{P},\,\mathbb{D}}$ is a spectral measure, there is a subset $\Lambda\subset\mathbb{Z}$ such that $0\in\Lambda$ and $\Lambda$ is a spectrum of $\mu_{\mathbb{P},\,\mathbb{D}}$. According to Corollary \ref{co3.6}, there exist integers $\lambda,\,\beta\in\Lambda$ such that
$\lambda\in\mathcal{Z}(\hat\chi_{m})$ and $\beta\in\mathcal{Z}(\hat\chi_{n})$. By \eqref{eq4.60}, we have
\begin{equation*}
	\lambda\ne\beta.
\end{equation*}
From \eqref{eq2.6}, we have $\lambda:=2^{k_{2,m}}L_{m}\lambda'$ and $\beta:=3^{k_{3,n}}L_{n}\beta'$ for some $\lambda'\in(2\mathbb{Z}+1)$ and $\beta'\in(3\mathbb{Z}\pm1)$. It follows that
\begin{equation}\label{eq4.61}
	\lambda-\beta=2^{k_{2,m}}\cdot3^{k_{3,n}}\Big(\frac{L_{m}}{3^{k_{3,n}}}\lambda'-\frac{L_{n}}{2^{k_{2,m}}}\beta'\Big).
\end{equation}
Since $m\in X_2$ and $n\in X_3$, the definitions of $L_m$ and $L_n$ show that $L_m\in(2\mathbb{Z}+1)$ and $L_n\in(3\mathbb{Z}\pm1)$. Hence, the assumption \eqref{eq4.59} implies that $\frac{L_{m}}{3^{k_{3,n}}}\in [(3\mathbb{Z}\setminus\{0\})\cap(2\mathbb{Z}+1)]$ and $\frac{L_{n}}{2^{k_{2,m}}}\in[(2\mathbb{Z}\setminus\{0\})\cap(3\mathbb{Z}\pm1)]$. Thus, we have $\frac{L_{m}}{3^{k_{3,n}}}\lambda'-\frac{L_{n}}{2^{k_{2,m}}}\beta'\in[\mathbb{Z}\setminus(2\mathbb{Z}\cup3\mathbb{Z})]$ by noting that $\lambda'\in(2\mathbb{Z}+1)$ and $\beta'\in(3\mathbb{Z}\pm1)$. Combining this with \eqref{eq4.61}, we obtain
\begin{equation}\label{eq4.62}
	\lambda-\beta\in2^{k_{2,m}}\cdot3^{k_{3,n}}[\mathbb{Z}\setminus(2\mathbb{Z}\cup3\mathbb{Z})].
\end{equation}

Since $\nu_{3}(L_{m})>k_{3,n}$, it follows from \eqref{eq2.6} and \eqref{eq4.62} that  $\lambda-\beta\not\in\mathcal{Z}(\hat\chi_{m})$. Moreover, by \eqref{eq2.6}, \eqref{eq4.57} and \eqref{eq4.62}, we have $\lambda-\beta\not\in\mathcal{Z}(\hat\chi_{j})$ for any $j\in(X_{2}\setminus\{m\})$. Thus,
\begin{equation}\label{eq4.63}
	\lambda-\beta\not\in\cup_{i\in X_{2}}\mathcal{Z}(\hat\chi_{i}).
\end{equation}
On the other hand, it follows from \eqref{eq2.6}, \eqref{eq4.59} and \eqref{eq4.62} that $\lambda-\beta\not\in\mathcal{Z}(\hat\chi_{n})$. Moreover, by \eqref{eq2.6}, \eqref{eq4.58} and \eqref{eq4.62}, we have $\lambda-\beta\not\in\mathcal{Z}(\hat\chi_{j})$ for any $j\in(X_{3}\setminus\{n\})$. Hence,
\begin{equation}\label{eq4.64}
	\lambda-\beta\not\in\cup_{i\in X_{3}}\mathcal{Z}(\hat\chi_{i}).
\end{equation}
By \eqref{eq4.63} and \eqref{eq4.64}, we have $\lambda-\beta\not\in\mathcal{Z}(\widehat{_{\mathbb{P},\,\mathbb{D}}})$, in contradiction to the assumption  that $\Lambda$ is a spectrum of $\mu_{\mathbb{P},\,\mathbb{D}}$. Therefore, the assumption in \eqref{eq4.59} is false, and the desired formula \eqref{eq1.5} holds. In conclusion, the necessity of Theorem \ref{th1.3} is proved.

\section{Spectra of Finite Convolutions}\label{sec5}

As discussed in Section \ref{sec2}, we will always assume {\bfseries (C1)}, {\bfseries (C2)} and {\bfseries (C3)} in the following part.
Since the sequence $\{d_n\}_{n=1}^{\infty}$ is bounded. Then, there is an integer $G\geqslant1$ such that
\begin{equation}\label{eq5.6}
	G=\max\{\nu_2(2d_n),\,\nu_3(3d_n):\,n\geqslant1\}.
\end{equation}
For any integer $n\ne0$ and $m\in\{2,\, 3\}$, we denote by $\theta_{m}(n)$ the non $m$-factor part of $n$, i.e.,
\begin{equation}\label{eq5.7}
	\theta_{m}(n):=\frac{n}{m^{\nu_m(n)}}.
\end{equation}

%T in Theorem 1, we introduce the notion of absorption for sets. Let $B$ and $C$ be two nonempty sets. We say that $B$
%is absorbed by $C$ if
%\begin{equation*}
%	B+C\subset C.
%\end{equation*}

This section is devoted to construction of spectra of finite convolutions: $\chi_{k_1}\ast\chi_{k_2}\ast\cdots\ast\chi_{k_\ell}$. Note that, for $n>k\ge1$, without the existence of Hadamard triples, $\lambda_1\in{\mathcal Z}(\widehat{\chi}_n)$ and $\lambda_2\in{\mathcal Z}(\widehat{\chi}_k)$ cann't guarantee $\lambda_1-\lambda_2\in{\mathcal Z}(\widehat{\mu}_{\mathbb{P},\,\mathbb{D}})$. On the other hand, however, Corollary \ref{co3.6} shows that $\Lambda\cap{\mathcal Z}(\widehat{\chi}_n)\ne\emptyset$ for all $n\ge1$ and all spectrum $\Lambda$ of ${\mu}_{\mathbb{P},\,\mathbb{D}}$ satisfying $0\in\Lambda$. Hence, in order to prove the sufficiency, the first thing is to choose some suitable subset $U_n\subset{\mathcal Z}(\widehat{\chi}_n)$ so that any spectrum $C_n$ for each $\chi_n$ satisfying $C_n\subset U_n\cup\{0\}$ will ensure that $C_{n_1}+C_{n_2}+\cdots +C_{n_k}$ is a spectrum of $\chi_{n_1}\ast\chi_{n_2}\ast\cdots\chi_{n_k}$ for any subset $\{n_1,\ n_2,\ \cdots,\ n_k\}$. \eqref{eq2.6} shows that the $U_n$ should be of the form $q_n^{k_n}\theta_{q_n}(p_1p_2\cdots p_{n'})(\mathbb{Z}\setminus q_n\mathbb{Z})$ for some $n'\ge n$. We now consider the choice of the $n'$.
	
Let
\begin{equation}\label{eq5.8}
	l_{q_n,n}:=\max\{j\in X_{q_n}:\,n\leqslant j,\,k_{q_n,n}\geqslant k_{q_j,j}\},\quad\forall\ n\ge1.
\end{equation}
It is easy to see that $q_n=q_{l_{q_n,n}}$ and $n\leqslant l_{q_n,n}$. Moreover, by \eqref{eq1.4}, we have $k_{q_n,n}>k_{q_n,l_{q_n,n}}$ if $n<l_{q_n,n}$. Then $\lambda_1\in q_n^{k_n}\theta_{q_n}(p_1p_2\cdots p_{l_{q_n,n}})(\mathbb{Z}\setminus q_n\mathbb{Z}){\mathcal Z}(\widehat{\chi}_n)$ and $\lambda_2\in{\mathcal Z}(\widehat{\chi}_{l_{q_n,n}})$ guarantee $\lambda_1-\lambda_2\in{\mathcal Z}(\widehat{\chi}_{l_{q_n,n}})$. The above properties will be frequently invoked in the following sections. The following example illustrates the role of $l_{q_n,n}$ in constructing the spectrum of $\mu_{\mathbb{P},\,\mathbb{D}}$.	
\begin{example}
	Let $p_1=2^5$, $p_2=5$, $D_1=\{0,\,1\}$, $D_2=\{0,\,2^2\}$ and, $p_n=2$ and $D_n=\{0,\,1\}$ for all $n\geqslant3$. By \eqref{eq1.2} and \eqref{eq5.8}, we have $k_{2,1}=4$, $k_{2,2}=2$, $l_{2,1}=2$ and  $l_{2,2}=2$. Obviously, $\{0,\,2^4\theta_2(p_1\cdots p_{l_{2,1}})\}+\{0,\,2^2\theta_2(p_1\cdots p_{l_{2,2}})\}$ is a spectrum of $\chi_1*\chi_2$.
	
	%, but we do not have  $\mathcal{Z}(\widehat{\chi_1*\chi_2})\subset\mathcal{Z}(\hat\delta_{\frac{1}{p_1}D_1})\pm\mathcal{Z}(\hat\delta_{\frac{1}{p_1}\frac{1}{p_2}D_2})$.
\end{example}

In order to make sure that $\lambda_1\in U_n$ and $\lambda_2\in U_k$ guarantee $\lambda_1-\lambda_2\in{\mathcal Z}(\widehat{\mu}_{\mathbb{P},\,\mathbb{D}})$ when $q_n\neq q_k$, we also need to consider the following $\alpha_{q_n,n}$, $\beta_{q_n,n}$ and $b_{q_n,n}$.
For any $n\geqslant1$, let
\begin{equation}\label{eq5.9}
	\alpha_{q_n,n}:= \begin{cases}
		0, &\mbox{ if } \{j\in X_{\widetilde{q}_n}:\,\nu_{\widetilde{q}_n}(L_{n})>k_{\widetilde{q}_n,j}\}=\emptyset, \\
		\max\{j\in X_{\widetilde{q}_n}:\,\nu_{\widetilde{q}_n}(L_{n})>k_{\widetilde{q}_n,j}\}, &\mbox{ else}.
	\end{cases}
\end{equation}
The following example illustrates the role of $\alpha_{q_n,n}$ in the spectrum of $\mu_{\mathbb{P},\,\mathbb{D}}$.
\begin{example}
	Let $p_1=2^53^4$, $p_2=5$, $D_1=\{0,\,1\}$, $D_2=\{0,\,2,\,2^2\}$ and, $p_n=2$ and $D_n=\{0,\,1\}$ for all $n\geqslant3$. Moreover, define $\varphi_2:=\delta_{\frac{1}{p_1}D_1}\bigast\delta_{\frac{1}{p_1}\frac{1}{p_2}D_2}$. It follows from \eqref{eq1.2} and \eqref{eq1.3} that $k_{2,1}=4$, $k_{3,2}=3$, $\nu_3(L_1)=4$ and $\nu_2(L_2)=4$. Since $\nu_3(L_1)>k_{3,2}$. By \eqref{eq5.9}, we have $\alpha_{2,1}=2$. Obviously, $\{0,\,2^4\theta_2(p_1\cdots p_{\alpha_{2,1}})\}+\{0,\,3^32^4\times5,\,-3^32^4\times5\}$ is a spectrum of $\varphi_2$.
\end{example}

Moreover, for any $n\geqslant1$, write
\begin{equation}\label{eq5.10}
	\beta_{q_n,n}:=\max\{\alpha_{q_n,j}:\,j\geqslant1,\, q_j=q_n,\,k_{q_j,j}\leqslant k_{q_n,n}\}.
\end{equation}	
The following example illustrates the role of $\beta_{q_n,n}$ in the construction of the spectrum of $\mu_{\mathbb{P},\,\mathbb{D}}$.
\begin{example}
Let $p_1=2^53^4$, $p_2=3^2$, $p_3=17$, $D_1=\{0,\,1\}$, $D_2=\{0,\,2^2\}$, $D_3=\{0,\,2^3,\,2^4\}$ and $p_n=2$ and $D_n=\{0,\,1\}$ for all $n\geqslant4$.. Moreover, define
\begin{equation*}
	\varphi_3:=\delta_{\frac{1}{p_1}D_1}\bigast\delta_{\frac{1}{p_1}\frac{1}{p_2}D_2}\bigast\delta_{\frac{1}{p_1}\frac{1}{p_2}\frac{1}{p_3}D_3}=\delta_{\frac{1}{2^53^4}\cdot\{0,\,1\}}\bigast\delta_{\frac{1}{2^53^4}\frac{1}{3^2}\cdot\{0,\,2^2\}}\bigast\delta_{\frac{1}{2^53^4}\frac{1}{3^2}\frac{1}{17}\cdot\{0,\,2^3,\,2^4\}}.
\end{equation*}
It follows from \eqref{eq1.2} and \eqref{eq1.3} that $k_{2,1}=4$, $k_{2,2}=2$, $k_{3,3}=5$, $\nu_3(L_1)=4$, $\nu_3(L_2)=6$ and $\nu_2(L_3)=2$. Since $\nu_3(L_2)>k_{3,3}$. By \eqref{eq5.9}, we have $\alpha_{2,2}=3$. Hence, we have $\beta_{2,1}=\alpha_{2,2}=3$ by noting $k_{2,2}<k_{2,1}$. It is easy to check that $\{0,\,2^4\theta_2(p_1\cdots p_{\beta_{2,1}})\}+\{0,\,2^2\theta_2(p_1\cdots p_{\alpha_{2,2}})\}+\{0,\,3^52^2\cdot17,\,-3^52^2\cdot17\}$ is a spectrum of $\varphi_3$.	
\end{example}

\begin{remark}
The role of $l_{q_n,n}$ in the construction of the prespectrum is obvious. However, the role of $\beta_{q_n,n}$ is not obvious. We will now illustrate the reason why $\beta_{q_n,n}$ defines $b_{q_n,n}$ instead of $\alpha_{q_n,n}$ by the following example.
\begin{example}
	Let $p_1=2^63^4$, $p_2=3^2$, $p_3=89$, $p_{2n}=2^2$ for all $n\geqslant2$ and  $p_{2n+1}=3^2$ for all $n\geqslant2$. Assume further $D_1=\{0,\,1,\,2\}$, $D_2=\{0,\,2^2,\,2\cdot2^2\}$, $D_3=\{0,\,3^3\}$, $D_{2n}=\{0,\,1\}$ for all $n\geqslant2$ and $D_{2n+1}=\{0,\,1,\,2\}$ for all $n\geqslant2$. Let $\mu=\bigast\limits_{n=1}^{\infty}\delta_{\frac{1}{p_1}\frac{1}{p_2}\cdots\frac{1}{p_n}D_n}$. Then,
	\begin{table}[H]
		\centering
		\begin{tabular}{c|c|c}
			$\delta_{\frac{1}{2^63^5}\{0,\,1,\,2\}}$ &  $\delta_{\frac{1}{2^63^5}\frac{1}{3^2}\{0,\,2^2,\,2\cdot2^2\}}$ & $\delta_{\frac{1}{2^63^5}\frac{1}{3^2}\frac{1}{89}\{0,\,3^3\}}$ \\
			$\mathcal{Z}(\hat\chi_1)=3^42^6(3\mathbb{Z}\pm1)$  & $\mathcal{Z}(\hat\chi_2)=3^62^4(3\mathbb{Z}\pm1)$ &  $\mathcal{Z}(\hat\chi_3)=2^53^4(2\mathbb{Z}+1)$ \\
			$k_{3,1}=4$  & $k_{3,2}=6$ &$k_{2,3}=5$ \\
			$\nu_{2}(L_{3,1})=6$  & $\nu_{2}(L_{3,2})=4$ & $\nu_{3}(L_{2,3})=4$ \\
			$b_{3,1}=3$  & $b_{3,2}=2$ & $b_{2,3}=3$ \\
		\end{tabular}
	\end{table}
	If we let $b_{q_n,n}:=\max\{l_{q_n,n},\,\alpha_{q_n,n}\}$. Choose $x_1=3^42^689\in3^4\theta_{3}(p_1\cdots p_{b_{3,1}})(3\mathbb{Z}\pm1)$, $x_2=3^62^6\in3^6\theta_{3}(p_1\cdots p_{b_{3,2}})(3\mathbb{Z}\pm1)$ and $x_3=2^53^789\in2^5\theta_{3}(p_1\cdots p_{b_{2,3}})(2\mathbb{Z}+1)$. It is clear that $x_1+x_2+x_3\not\in\mathcal{Z}(\widehat{\mu_{\mathbb{P},\,\mathbb{D}}})$. This means that, if we define $b_{q_n,n}:=\max\{l_{q_n,n},\,\alpha_{q_n,n}\}$, then Theorem \ref{th5.11} \rm(\romannumeral2) in the following does not hold.
\end{example}	
\end{remark}

Finally, we define
\begin{equation}\label{eq5.11}
	b_{q_n,n}:=\max\{l_{q_n,n},\,\beta_{q_n,n}\}.
\end{equation}	

The following example will illustrate the role of $b_{q_n,n}$ in the construction of the spectrum of a convolution measure.
\begin{example}
Let $p_1=2^53^4$, $p_2=5$, $p_3=3^2$, $p_4=37$, $D_1=\{0,\,1\}$, $D_2=\{0,\,1,\,2\}$, $D_3=\{0,\,2^2\}$, $D_4=\{0,\,2^4,2^5\}$ and, $p_{n}=2^2$ and $D_n=\{0,\,1\}$ for all $n\geqslant5$. Let  $\varphi_4:=\delta_{\frac{1}{p_1}D_1}\bigast\delta_{\frac{1}{p_1}\frac{1}{p_2}D_2}\bigast\delta_{\frac{1}{p_1}\frac{1}{p_2}\frac{1}{p_3}D_3}\bigast\delta_{\frac{1}{p_1}\frac{1}{p_2}\frac{1}{p_3}\frac{1}{p_4}D_4}$. It follows that
\begin{table}[H]\small
\begin{tabular}{c|c|c|c}
	$\delta_{\frac{1}{2^53^4}\{0,\,1\}}$ & $\delta_{\frac{1}{2^53^4}\frac{1}{17}\{0,\,2^3,\,2^4\}}$ & $\delta_{\frac{1}{2^53^4}\frac{1}{17}\frac{1}{3^2}\{0,\,2^2\}}$ & $\delta_{\frac{1}{2^53^4}\frac{1}{17}\frac{1}{3^2}\frac{1}{37}\{0,\,2^4,\,2^5\}}$\\
	$\mathcal{Z}(\hat\chi_1)=2^43^4(2\mathbb{Z}+1)$\hspace{-0.5em}   & \hspace{-0.5em}$\mathcal{Z}(\hat\chi_2)=3^32^2\times17(3\mathbb{Z}\pm1)$\hspace{-0.5em} & \hspace{-0.5em}$\mathcal{Z}(\hat\chi_3)=2^23^6\times17(2\mathbb{Z}+1)$\hspace{-0.5em} & \hspace{-0.5em}$\mathcal{Z}(\hat\chi_4)=3^52\times17\times37(2\mathbb{Z}+1)$ \\
	$k_{2,1}=4$  &   $k_{3,2}=3$ &  $k_{2,3}=2$ & $k_{3,4}=5$ \\
	$\nu_{3}(L_{1})=4$  &   $\nu_{2}(L_{2})=2$ &  $\nu_{3}(L_{3})=6$ & $\nu_{2}(L_{4})=1$\\
	$b_{2,1}=4$  &   $b_{3,2}=2$ &  $b_{2,3}=4$ & $b_{3,4}=4$ \\
\end{tabular}	
\end{table}
\hspace{-2em}It is easy to see that  $\{0,\,2^4\theta_2(p_1\cdots p_{b_{2,1}})\}+\{0,\,3^3\theta_3(p_1\cdots p_{b_{3,2}}),\,-3^3\theta_3(p_1\cdots p_{b_{3,2}})\}+\{0,\,2^2\theta_2(p_1\cdots p_{b_{2,3}})\}+\{0,\,3^5\theta_3(p_1\cdots p_{b_{3,4}}),\,-3^5\theta_3(p_1\cdots p_{b_{3,4}})\}$ is a spectrum of $\varphi_4$.	
\end{example}

\begin{lemma}\label{le5.8}
	Assume that the pair $(\mathbb{P},\,\mathbb{D})$ satisfies \eqref{eq1.4}, \eqref{eq1.5} and \eqref{eq2.12-3}. Then, $b_{q_n,n}<\infty$ for all $n\geqslant1$.
\end{lemma}

\begin{pf}
	By \eqref{eq2.12-3}, we have $k_{q_n,n}\geqslant0$ for every $n\geqslant1$. Moreover, by \eqref{eq1.4}, we can see that $k_{q_n,i}\ne k_{q_n,j}$ for any distinct integers $i,\,j\in X_{q_n}$ and $n\geqslant1$. Hence, for any $n\geqslant1$, we have $\#\{j\in X_{q_n}:\,k_{q_n,j}<k_{q_n,n}\}=\#\{k_{q_n,j}:\,j\in X_{q_n},\,k_{q_n,j}<k_{q_n,n}\}\leqslant k_{q_n,n}$. Thus,
	\begin{equation*}
		l_{q_n,n}<\infty.
	\end{equation*}
	
	On the other hand, it follows from \eqref{eq2.12-3} that $\nu_{\widetilde{q}_n}(L_n)\geqslant0$ for any $n\geqslant1$. Moreover, it follows from \eqref{eq1.4} that $k_{\widetilde{q}_n,i}\ne k_{\widetilde{q}_n,j}$ for any $i\neq j\in X_{\widetilde{q}_n}$ and $n\geqslant1$.  Therefore, we get $\#\{j\in X_{\widetilde{q}_n}:\,k_{\widetilde{q}_n,j}<\nu_{\widetilde{q}_n}(L_n)\}=\#\{k_{\widetilde{q}_n,j}:\,j\in X_{\widetilde{q}_n},\,k_{\widetilde{q}_n,j}<\nu_{\widetilde{q}_n}(L_n)\}\leqslant \nu_{\widetilde{q}_n}(L_n)$. Hence,
	\begin{equation*}
		\alpha_{q_n,n}<\infty.
	\end{equation*}
	Combining the above two inequalities, we have $\beta_{q_n,n}<\infty$. Thus, by \eqref{eq5.11}, we have $b_{q_n,n}<\infty$. Given the arbitrariness of $n$, the lemma is proven.
\end{pf}

%, where $\nu_{q_{n}}(p_{n+1}p_{n+2}\cdots p_m):=0$ when $n=m$

\begin{lemma}\label{le5.9}
Assume that the pair $(\mathbb{P},\,\mathbb{D})$ satisfies \eqref{eq1.4}, \eqref{eq1.5} and \eqref{eq2.12-3}. Let $n,\,m$ be two distinct positive integers, then
\begin{enumerate}
	\item[\rm(\romannumeral1)] $\#\{i\in X_{q_n}:\,n<i\leqslant l_{q_n,n}\}\leqslant2(G-1)$ for all $n\geqslant1$.
	\item[\rm(\romannumeral2)] If $n<m$, $q_n=q_m$ and $k_{q_n,n}>k_{q_n,m}$, then $G\geqslant2$ and $\nu_{q_{n}}(p_{n+1}p_{n+2}\cdots p_m)\leqslant G-2$.
	\item[\rm(\romannumeral3)] If $n<m$, $q_n\neq q_m$ and $k_{q_n,n}\geqslant\nu_{q_{n}}(L_{m})$, then $G\geqslant2$ and $\nu_{q_{n}}(p_{n+1}p_{n+2}\cdots p_m)\leqslant G-2$.		
	\item[\rm(\romannumeral4)] If $q_m=q_n$ and $k_{q_n,m}>k_{q_n,n}$, then $l_{q_n,m}\geqslant l_{q_n,n}$.
	\item[\rm(\romannumeral5)] If $q_m=q_n$ and $k_{q_n,m}>k_{q_n,n}$, then $b_{q_n,m}\geqslant b_{q_n,n}$.
	\item[\rm(\romannumeral6)] If $q_m\neq q_n$ and $\nu_{q_n}(L_{m})>k_{q_n,n}$, then $b_{q_m,m}\geqslant b_{q_n,n}$.
	\item[\rm(\romannumeral7)] For any positive integers $z_1$ and $z_2$ with $z_1\leqslant z_2$. If $q\in\{2,\,3\}$ and $\#\{i\in X_{q}:\,z_1\leqslant i\leqslant z_2\}\geqslant G$, then $\nu_{q}(p_{z_1}p_{z_1+1}\cdots p_{z_2})\geqslant\#\{i\in X_{q}:\,z_1\leqslant i\leqslant z_2\}-G$.
	\item[\rm(\romannumeral8)] $\nu_{q_n}(p_{n+1}p_{n+2}\cdots p_{b_{q_n,n}})\leqslant G-1$, where $\nu_{q_n}(p_{n+1}\cdots p_{b_{q_n,n}}):=0$ when $n=b_{q_n,n}$.
	\item[\rm(\romannumeral9)] For any integer $n\geqslant1$, we have $k_{q_n,n}\geqslant\nu_{q_n}(L_{\beta_{q_n,n}})$.
\end{enumerate}	
\end{lemma}	

\begin{pf}
	\rm(\romannumeral1) For any $n\geqslant1$, it suffices to prove that \rm(\romannumeral1) holds when $n<l_{q_n,n}$, since the case $n=l_{q_n,n}$ is trivial.
	
	For any integers $n\geqslant1$ and $i\geqslant1$, if $n<i\leqslant l_{q_n,n}$ and $q_i=q_n$, then
	\begin{equation*}
		k_{q_n,i}=\nu_{q_n}(p_{1}p_{2}\cdots p_{i})-\nu_{q_n}(q_nd_{i})\geqslant \nu_{q_n}(p_{1}p_{2}\cdots p_{n})-G\geqslant k_{q_n,n}+1-G
	\end{equation*}
	and
	\begin{equation*}
		k_{q_n,i}=\nu_{q_n}(p_{1}p_{2}\cdots p_{i})-\nu_{q_n}(d_{i})-1\leqslant \nu_{q_n}(p_{1}p_{2}\cdots p_{l_{q_n,n}})-1\leqslant k_{q_n,l_{q_n,n}}+G-1\leqslant k_{q_n,n}+G-2
	\end{equation*}
	by using \eqref{eq1.2} and \eqref{eq5.6}. Moreover, by \eqref{eq1.4}, we get that $k_{q_n,s}\ne k_{q_n,t}$ for any $s\neq t\in X_{q_n}$. Hence, we have
	\begin{equation*}
		\begin{split}
			\#\{i\in X_{q_n}:\,n<i\leqslant l_{q_n,n}\}=&\#\{k_{q_n,i}:\,i\in X_{q_n},\,n<i\leqslant l_{q_n,n}\}\\
			\leqslant&\{j\in\mathbb{N}:\,k_{q_n,n}+1-G\leqslant j\leqslant k_{q_n,n}+G-2\}\leqslant2(G-1).
		\end{split}
	\end{equation*}
	This proves \rm(\romannumeral1).
	
	\rm(\romannumeral2) Since $n<m$ and $k_{q_n,n}>k_{q_n,m}$, we have $k_{q_n,n}>k_{q_n,m}$. Hecne,
	\begin{equation}\label{eq5.17}
		0<k_{q_n,n}-k_{q_n,m}=-\nu_{q_{n}}(p_{n+1}p_{n+2}\cdots p_m)-\nu_{q_{n}}(d_n)-1+\nu_{q_{n}}(d_m)+1.
	\end{equation}
	Thus, we have $\nu_{q_{n}}(d_m)>0$, which yields $G\geqslant\nu_{q_{n}}(q_{n}d_m)\geqslant2$. Moreover, by \eqref{eq5.17}, we have $\nu_{q_{n}}(p_{n+1}p_{n+2}\cdots p_m)<\nu_{q_{n}}(d_m)\leqslant G-1$. Thus, \rm(\romannumeral2) holds.
	
	\rm(\romannumeral3) Since $n<m$ and $k_{q_n,n}\geqslant\nu_{q_{n}}(L_{m})$, we have
	\begin{equation*}
		0\leqslant k_{q_n,n}-\nu_{q_{n}}(L_{m})=-\nu_{q_{n}}(p_{n+1}\cdots p_m)-\nu_{q_{n}}(d_n)-1+\nu_{q_{n}}(d_m).
	\end{equation*}
	Therefore, we have $\nu_{q_{n}}(d_m)\geqslant1$, which yields $G\geqslant\nu_{q_{n}}(q_{n}d_m)\geqslant2$. Hence, we have $\nu_{q_{n}}(p_{n+1}\cdots p_m)\leqslant\nu_{q_{n}}(d_m)-1\leqslant G-2$. This proves \rm(\romannumeral3).
	
	\rm(\romannumeral4) By \eqref{eq1.4}, we have $k_{q_n,i}\ne k_{q_n,j}$ for any distinct integers $i,\,j\in X_{q_n}$. The definitions of $l_{q_n,n}$ and the assumption $k_{q_n,m}> k_{q_n,n}$ shows $k_{q_n,m}>k_{q_n,l_{q_n,n}}$. Note that $l_{q_n,m}$ is the largest integer $s$ such that $k_{q_n,m}\geqslant k_{q_n,s}$ and $s\in X_{q_n}$, we see that $l_{q_n,m}\geqslant l_{q_n,n}$. Hence, \rm(\romannumeral4) is established.
	
	\rm(\romannumeral5) Since $k_{q_n,m}>k_{q_n,n}$. According to the definition of $\beta_{q_n,m}$, we have $\beta_{q_n,m}\geqslant \beta_{q_n,n}$. On the other hand, by \rm(\romannumeral4), we have  $l_{q_n,m}\geqslant l_{q_n,n}$.  Hence, we have $b_{q_n,m}=\max\{l_{q_n,m},\,\beta_{q_n,m}\}\geqslant \max\{l_{q_n,n},\,\beta_{q_n,n}\}=b_{q_n,n}$.

	\rm(\romannumeral6) Choose positive integers $m$ and $n$ with $q_m\ne q_n$ satisfying
	\begin{equation}\label{eq5.19}
		\nu_{q_n}(L_{m})>k_{q_n,n}.
	\end{equation}
	
	By \eqref{eq5.19}, we have $\nu_{q_n}(L_{m})>k_{q_n,j}$ for all $j\in X_{q_n}$ with $k_{q_n,n}\geqslant k_{q_n,j}$. Hence, by \eqref{eq1.5}, we get $k_{q_m,m}\geqslant\nu_{q_m}(L_{j})$ for all $j\in X_{q_n}$ with $k_{q_n,n}\geqslant k_{q_n,j}$. Moreover, by \eqref{eq5.9}, we have $\nu_{\widetilde{q}_n}(L_{j})>k_{\widetilde{q}_n,\alpha_{q_n,j}}$ for all $j\in X_{q_n}$. Since $q_m=\widetilde{q}_n$, we have
	\begin{equation*}
		k_{q_m,m}>k_{q_m,\alpha_{q_n,j}}
	\end{equation*}
	for all $j\in X_{q_n}$ with $k_{q_n,n}\geqslant k_{q_n,j}$. It follows from \eqref{eq5.8} that $l_{q_m,m}\geqslant\alpha_{q_n,j}$ for all $j\in X_{q_n}$ with $k_{q_n,n}\geqslant k_{q_n,j}$. Hence, by \eqref{eq5.10}, we get $l_{q_m,m}\geqslant\beta_{q_n,n}$.
	
	On the other hand, according to the definition of $l_{q_n,n}$, we have $k_{q_n,n}\geqslant k_{q_n,l_{q_n,n}}$. It follows from \eqref{eq5.19} that $\nu_{q_n}(L_{m})>k_{q_n,l_{q_n,n}}$. Hence, by the definitions of $\alpha_{q_m,m}$, we have $\alpha_{q_m,m}\geqslant l_{q_n,n}$. Since $\beta_{q_m,m}\geqslant\alpha_{q_m,m}$, we have $\beta_{q_m,m}\geqslant l_{q_n,n}$. Thus, we obtain $b_{q_m,m}=\max\{l_{q_m,m},\,\beta_{q_m,m}\}\geqslant\max\{l_{q_n,n},\,\beta_{q_n,n}\}=b_{q_n,n}$, so \rm(\romannumeral6) is true.

	\rm(\romannumeral7) Let $s$ and $t$ be integers satisfying $z_1\leqslant s,\,t\leqslant z_2$, and such that $k_{q,s}=\min\{k_{q,i}:\,z_1\leqslant i\leqslant z_2\}$ and $k_{q,t}=\max\{k_{q,i}:\,z_1\leqslant i\leqslant z_2\}$. Then \eqref{eq1.2}, \eqref{eq1.4} and \eqref{eq5.6} shows
	\begin{equation*}
		\begin{split}
			\#\{i\in X_{q}:\,z_1\leqslant i\leqslant z_2\}\leqslant& k_{q,t}-k_{q,s}+1\\
			=&\nu_{q}(p_{z_1}\cdots p_t)-\nu_{q}(p_{z_1}\cdots p_s)-\nu_{q}(d_t)+\nu_{q}(d_s)+1\\
			\leqslant&\nu_{q}(p_{z_1}\cdots p_{z_2})+\nu_{q}(d_s)+1\\
			\leqslant&\nu_{q}(p_{z_1}\cdots p_{z_2})+G.
		\end{split}
	\end{equation*}
	It is now evident that \rm(\romannumeral7) holds.
	
	\rm(\romannumeral8)  It suffices to consider the case $n<b_{q_n,n}$. According to the definition of $b_{q_n,n}$, we get that either $b_{q_n,n}=l_{q_n,n}$ or $b_{q_n,n}=\beta_{q_n,n}$. If $b_{q_n,n}=l_{q_n,n}$, then $k_{q_n,n}>k_{q_n,l_{q_n,n}}$. Hence, by \rm(\romannumeral2), we get
	\begin{equation*}
		\nu_{q_n}(p_{n+1}\cdots p_{b_{q_n,n}})\leqslant G-1.
	\end{equation*}

	If $b_{q_n,n}=\beta_{q_n,n}$. According to the definition of $\beta_{q_{n},n}$, there exists $n'\in X_{q_n}$ with $k_{q_n,n'}\leqslant k_{q_n,n}$ such that $\alpha_{q_n,n'}=\beta_{q_n,n}=b_{q_n,n}$, $q_{\alpha_{q_n,n'}}=\widetilde{q}_{n}$ and $\nu_{\widetilde{q}_{n}}(L_{n'})>k_{\widetilde{q}_{n},\alpha_{q_n,n'}}$. Hence, by \eqref{eq1.5}, we have $k_{q_n.n'}\geqslant\nu_{q_n}(L_{\alpha_{q_n,n'}})$. Combining this result with the condition $k_{q_n,n}\geqslant k_{q_n,n'}$, then $k_{q_n,n}\geqslant\nu_{q_n}(L_{\alpha_{q_n,n'}})$. Hence, by \rm(\romannumeral3), we get
	\begin{equation*}
		\nu_{q_n}(p_{n+1}\cdots p_{b_{q_n,n}})\leqslant G-1.
	\end{equation*}
	We have thus shown \rm(\romannumeral8).

	\rm(\romannumeral9) According to the definition of $\beta_{q_n,n}$, we get that $\beta_{q_n,n}\in X_{\widetilde{q}_n}$ and there exists an integer $j\in X_{q_n}$ such that $k_{q_n,j}\leqslant k_{q_n,n}$, $\alpha_{q_n,j}=\beta_{q_n,n}$ and
	\begin{equation*}
		\nu_{\widetilde{q}_n}(L_j)>k_{\widetilde{q}_n,\alpha_{q_n,j}}.
	\end{equation*}
	Hence, by \eqref{eq1.5}, we have $k_{q_n,j}\geqslant\nu_{q_n}(L_{\alpha_{q_n,j}})$. Since $k_{q_n,j}\leqslant k_{q_n,n}$, we get $k_{q_n,n}\geqslant\nu_{q_n}(L_{\alpha_{q_n,j}})$. Since $\alpha_{q_n,j}=\beta_{q_n,n}$. It follows that $k_{q_n,n}\geqslant\nu_{q_n}(L_{\beta_{q_n,n}})$, so \rm(\romannumeral9) is true.
\end{pf}

Let
\begin{equation}\label{eq5.28}
	U_n:=q_n^{k_n}\theta_{q_n}(p_1p_2\cdots p_{b_{q_n,n}})(\mathbb{Z}\setminus q_n\mathbb{Z}),\qquad n\geqslant1.
\end{equation}
Since $b_{q_n,n}\geqslant n$ and $L_n=\theta_{q_n}(\frac{p_1p_2\cdots p_{n}}{d_nq_n})$, we get $\theta_{q_n}(p_1p_2\cdots p_{b_{q_n,n}})\in L_n(\mathbb{Z}\setminus q_n\mathbb{Z})$. It follows from \eqref{eq2.6} that
\begin{equation}\label{eq5.30}
	U_n\subset\mathcal{Z}(\hat\chi_n),\qquad\forall \, n\geqslant1.
\end{equation}

\begin{theorem}\label{th5.11}
Assume that the pair $(\mathbb{P},\,\mathbb{D})$ satisfies \eqref{eq1.4}, \eqref{eq1.5} and \eqref{eq2.12-3}. Then, we have the following statements.
\begin{enumerate}
	\item[\rm(\romannumeral1)] Let $q\in\{2,\,3\}$ and suppose further that $\mathcal{A}$ is a finite nonempty subset of $X_q$. Then, for any $x_j\in U_j$ with $j\in\mathcal{A}$, we have $\sum_{j\in\mathcal{A}}x_j\in U_{j_1}$, where $j_1\in\mathcal{A}$ satisfies $k_{q,j_1}=\min\{k_{q,j}:\,j\in\mathcal{A}\}$.
	\item[\rm(\romannumeral2)] Let $\mathcal{A}$ be a finite nonempty subset of positive integers. Then, for any $x_j\in U_j$ with $j\in\mathcal{A}$, we have $\widehat{\mu_{\mathbb{P},\,\mathbb{D}}}(\sum_{j\in\mathcal{A}}x_j)=0$.
	\item[\rm(\romannumeral3)] Let $\mathcal{A}$ be a finite nonempty subset of positive integers. For any $j\in\mathcal{A}$ and $x_j\in U_j$, define $C_j=\{0,\,x_j\}$ if $q_j=2$ and $C_j=\{0,\,x_j,\,y_j\}$ if $q_j=3$. Then, $\Gamma:=\sum_{j\in \mathcal{A}}C_j$ is a spectrum of $\bigast\limits_{j\in \mathcal{A}}\delta_{p_1^{-1}p_2^{-1}\cdots p_{j}^{-1}D_j}$.
	\item[\rm(\romannumeral4)] Let $\mathcal{A}$ and $\Gamma$ be as defined in (\romannumeral3). For any $\lambda\in\Gamma$, there exists an integer sequence $\{c_j\}_{j\in\mathcal{A}}$ with $c_j\in C_j$ such that $\lambda=\sum_{j\in\mathcal{A}}c_j$. If $\lambda\neq0$, then there exists $j_{\lambda}\in\mathcal{A}$ such that $c_{j_{\lambda}}\neq0$. For any $\lambda\in(\Gamma\setminus\{0\})$, define $n_{\lambda}:=\max(\{b_{q_{j_{\lambda}},j_{\lambda}}\}\cup\mathcal{A})$. Take $z_{\lambda}\in p_1p_2\cdots p_{n_{\lambda}}\mathbb{Z}$ with $z_0=0$. Then, the set $\{\lambda+z_{\lambda}:\, \lambda\in\Gamma\}$ is also a spectrum of the probability measure $\bigast\limits_{j\in \mathcal{A}}\delta_{p_1^{-1}p_2^{-1}\cdots p_{j}^{-1}D_j}$. Furthermore, we have $\{\lambda+z_{\lambda}:\, \lambda\in\Gamma\}\subset\sum_{j\in \mathcal{A}}(\{0\}\cup U_j)$.
\end{enumerate}		
\end{theorem}

%\item[\rm(\romannumeral4)] Let $\mathcal{A}$ be a finite nonempty subset of positive integers, and let $\Gamma$ be as defined in (\romannumeral3). For any $\lambda:=\sum_{j\in \mathcal{A}}c_j\in(\Gamma\backslash\{0\})$, there exists $j_{\lambda}\in\mathcal{A}$ such that $c_{j_{\lambda}}\neq0$, where $c_j\in C_j$. We choose $z_{\lambda}\in p_1p_2\cdots p_{b_{q_{j_{\lambda}},j_{\lambda}}}\mathbb{Z}$ with $z_0=0$. Then, the set $\{\lambda+z_{\lambda}:\, \lambda\in\Gamma\}$ is also a spectrum of the probability measure $\bigast\limits_{j\in \mathcal{A}}\delta_{p_1^{-1}p_2^{-1}\cdots p_{j}^{-1}D_j}$. Furthermore, we have $\{\lambda+z_{\lambda}:\, \lambda\in\Gamma\}\subset\sum_{j\in \mathcal{A}}(\{0\}\cup U_j)$.	

\begin{pf}
\rm(\romannumeral1) Without loss of generality, assume that $\mathcal{A}$ contains at least two elements; otherwise the statement \rm(\romannumeral1) holds trivially. We arrange the elements of $\mathcal{A}$ as $j_1,\,j_2,\cdots,\,j_t$ in such a way that $k_{q,j_1}<k_{q,j_2}<\cdots<k_{q,j_t}$. By Lemma \ref{le5.9} \rm(\romannumeral5), we have $b_{q,j_1}\leqslant b_{q,j_2}\leqslant\cdots\leqslant b_{q,j_t}$. Hence, we have
\begin{equation*}
	\theta_{q}(p_1p_2\cdots p_{b_{q,j_1}})\mid\theta_{q}(p_1p_2\cdots p_{b_{q,j_i}}),\qquad 1\leqslant i\leqslant t.
\end{equation*}
Moreover, for any $1\leqslant i\leqslant t$ and $x_{j_i}\in U_{j_i}$, one can find an integer $z_{j_i}\in(\mathbb{Z}\setminus q\mathbb{Z})$ such that $x_{j_i}=q^{k_{q,j_i}}\theta_{q}(p_1p_2\cdots p_{b_{q,j_i}})z_{j_i}$. Hence,
\begin{equation*}
	\begin{split}
		\sum_{j\in\mathcal{A}}x_j=\sum_{i=1}^tx_{j_i}
		=&q^{k_{q,j_1}}\theta_{q}(p_1p_2\cdots p_{b_{j_1}})z_{j_1}+\sum_{i=2}^t q^{k_{q,j_i}}\theta_{q}(p_1p_2\cdots p_{b_{j_i}})z_{j_i}\\
		=&q^{k_{q,j_1}}\theta_{q}(p_1\cdots p_{b_{j_1}})\Big(z_{j_1}+\sum_{i=2}^t q^{k_{q,j_i}-k_{q,j_1}}\theta_{q}(p_{b_{j_1}+1}\cdots p_{b_{j_i}})z_{j_i}\Big).
	\end{split}
\end{equation*}
Thus, we have $\sum_{j\in\mathcal{A}}x_j\in U_{j_1}$ by noting $k_{q,j_i}>k_{q,j_1}$ for $2\leqslant i\leqslant t$. This completes the proof of conclusion \rm(\romannumeral1).

\rm(\romannumeral2) If $\mathcal{A}\subset X_2$ or $\mathcal{A}\subset X_3$, by \rm(\romannumeral1), we have $\sum_{j\in\mathcal{A}}x_j\in U_{j_1}$ and so $\widehat{\mu_{\mathbb{P},\,\mathbb{D}}}(\sum_{j\in\mathcal{A}}x_j)=0$.

If both $\mathcal{A}\cap X_2$ and $\mathcal{A}\cap X_3$ are non-empty. Suppose $\mathcal{A}\cap X_2=\{i_1,\,i_2,\cdots,\,i_t\}$ with $k_{2,i_1}<k_{2,i_2}<\cdots<k_{2,i_t}$ and $\mathcal{A}\cap X_3=\{j_1,\,j_2,\cdots,\,j_s\}$ with $k_{3,j_1}<k_{3,j_2}<\cdots<k_{3,j_s}$. By \rm(\romannumeral1), one can find $z_2\in(2\mathbb{Z}+1)$ and $z_3\in(3\mathbb{Z}\pm1)$ such that
\begin{equation}\label{eq5.31}
	\sum_{h\in(\mathcal{A}\cap X_2)}x_h=2^{k_{2,i_1}}\theta_{2}(p_1p_2\cdots p_{b_{2,i_1}})z_2\in U_{i_1}~~\mbox{ and }~~\sum_{h\in(\mathcal{A}\cap X_3)}x_h=3^{k_{3,j_1}}\theta_{3}(p_1p_2\cdots p_{b_{3,j_1}})z_3\in U_{j_1}.
\end{equation}

Note that $i_1\in X_2$ and $j_1\in X_3$, we have $i_1\neq j_1$. We only consider the case $i_1<j_1$, since the case $i_1>j_1$ is almost the same. Since $j_1\leqslant b_{3,j_1}$, we get $i_1\leqslant b_{3,j_1}$.

If $k_{3,j_1}\geqslant\nu_{3}(L_{i_1})$, then
\begin{equation*}
	3^{k_{3,j_1}}\theta_{2}(\theta_{3}(p_1p_2\cdots p_{b_{3,j_1}}))z_3\in L_{i_1}\mathbb{Z}
\end{equation*}
by noting $L_{i_1}=3^{\nu_{3}(L_{i_1})}\theta_{2}\big(\theta_{3}\big(\frac{p_1p_2\cdots p_{i_1}}{d_{i_1}}\big)\big)$. Moreover, it follows from \eqref{eq1.3-1} and the assumption $i_1\leqslant b_{3,j_1}$ that $k_{2,i_1}<\nu_{2}(p_1p_2\cdots p_{b_{3,j_1}})$. Hence, it follows from \eqref{eq5.31} that
\begin{equation}\label{eq5.32-1}
	\sum_{h\in(\mathcal{A}\cap X_3)}x_h=2^{\nu_{2}(p_1p_2\cdots p_{b_{3,j_1}})}3^{k_{3,j_1}}\theta_{2}(\theta_{3}(p_1p_2\cdots p_{b_{3,j_1}}))z_3\in2^{k_{2,i_1}+1}L_{i_1}\mathbb{Z}
\end{equation}
if $k_{3,j_1}\geqslant\nu_{3}(L_{i_1})$. By \eqref{eq5.31} and \eqref{eq5.32-1}, we get
\begin{equation*}
	\sum_{h\in\mathcal{A}}x_h
	=\sum_{h\in(\mathcal{A}\cap X_2)}x_h+\sum_{h\in(\mathcal{A}\cap X_3)}x_h\in2^{k_{2,i_1}}L_{i_1}(2\mathbb{Z}+1)=\mathcal{Z}(\hat\chi_{i_1}).
\end{equation*}

If $\nu_{3}(L_{i_1})>k_{3,j_1}$, then $b_{2,i_1}\geqslant b_{3,j_1}$ by Lemma \ref{le5.9} \rm(\romannumeral6). Hence, we get $b_{2,i_1}\geqslant j_1$ by noting $b_{3,j_1}\geqslant j_1$. Moreover, it follows from \eqref{eq1.4} and the assumption $\nu_{3}(L_{i_1})>k_{3,j_1}$ that $k_{2,i_1}\geqslant\nu_{2}(L_{j_1})$. Thus, similar to \eqref{eq5.32-1}, we get
\begin{equation*}
	\sum_{h\in(\mathcal{A}\cap X_2)}x_h=3^{\nu_{3}(p_1p_2\cdots p_{b_{2,i_1}})}2^{k_{2,i_1}}\theta_{2}(\theta_{3}(p_1p_2\cdots p_{b_{2,i_1}}))z_2\in3^{k_{3,j_1}+1}L_{j_1}\mathbb{Z}.
\end{equation*}
Therefore, \eqref{eq5.31} shows                                  $\sum_{h\in\mathcal{A}}x_h
=\sum_{h\in(\mathcal{A}\cap X_2)}x_h+\sum_{h\in(\mathcal{A}\cap X_3)}x_h\in3^{k_{3,j_1}}L_{j_1}(3\mathbb{Z}\pm1)=\mathcal{Z}(\hat\chi_{j_1})$.

Consequently, $\sum_{h\in\mathcal{A}}x_h\in\mathcal{Z}(\hat\mu)$. This proves \rm(\romannumeral2).

\rm(\romannumeral3) For any $j\in\mathcal{A}$, it is clear $U_j=-U_j$ and, we get  $x_j-(-x_j)\in U_j$ when $q_j=3$.

For any distinct elements $\lambda_1,\,\lambda_2\in\Gamma$, they can be written as $\lambda_1=\sum_{j\in\mathcal{A}}a_j$ and $\lambda_2=\sum_{j\in\mathcal{A}}c_j$, where $a_j,\,c_j\in (\{0\}\cup U_j)$. Hence, $\lambda_1-\lambda_2=\sum_{j\in\mathcal{A}}(a_j-c_j)$. According to the above discussion, we have $a_j-c_j\in U_j$ if $a_j\neq c_j$. Since $\lambda_1\ne\lambda_2$, then $\{j\in\mathcal{A}:\,a_j\neq c_j\}\ne\emptyset$. Hence $\lambda_1-\lambda_2=\sum_{j\in\mathcal{A},\, a_j\neq c_j}(a_j-c_j)\in\sum_{j\in\mathcal{A},\, a_j\neq c_j} U_j$. By \rm(\romannumeral2), we have $\lambda_1-\lambda_2\in\mathcal{Z}(\hat\mu)$.

On the other hand, we see that the support of the measure $\bigast\limits_{j\in \mathcal{A}}\delta_{p_1^{-1}p_2^{-1}\cdots p_{j}^{-1}D_j}$ has cardinality at most $2^{\#(\mathcal{A}\cap X_2)}3^{\#(\mathcal{A}\cap X_3)}$ and $\#\Gamma=2^{\#(\mathcal{A}\cap X_2)}3^{\#(\mathcal{A}\cap X_3)}$. Lemma \ref{le5.2} show that $\Gamma$ is a spectrum of $\bigast\limits_{j\in \mathcal{A}}\delta_{p_1^{-1}p_2^{-1}\cdots p_{j}^{-1}D_j}$. The proof of \rm(\romannumeral3) is finished.

\rm(\romannumeral4) For any $j\in\mathcal{A}$ and $\lambda\in\Gamma$, we have $j\leqslant n_{\lambda}$, so $p_1^{-1}p_2^{-1}\cdots p_{j}^{-1}z_{\lambda}\in\mathbb{Z}$. Since $p_1p_2\cdots p_j$ is a period of the function $\hat\delta_{p_1^{-1}p_2^{-1}\cdots p_{j}^{-1}D_j}$, we get $\hat\delta_{p_1^{-1}p_2^{-1}\cdots p_{j}^{-1}D_j}(x+\lambda+z_{\lambda})=\hat\delta_{p_1^{-1}p_2^{-1}\cdots p_{j}^{-1}D_j}(x+\lambda)$ for all $x\in\mathbb{R}$, $j\in\mathcal{A}$ and $\lambda\in\Gamma$. Hence,
\begin{equation*}
	\sum_{t\in\{\lambda+z_{\lambda}:\,\lambda\in\Gamma\}}\Big\lvert\prod_{j\in\mathcal{A}}\hat\delta_{p_1^{-1}p_2^{-1}\cdots p_{j}^{-1}D_j}(x+t)\Big\rvert^2=\sum_{\lambda\in\Gamma}\Big\lvert\prod_{j\in\mathcal{A}}\hat\delta_{p_1^{-1}p_2^{-1}\cdots p_{j}^{-1}D_j}(x+\lambda)\Big\rvert^2=1,\qquad \forall \, x\in\mathbb{R}, \end{equation*}
where the second equality follows from \rm(\romannumeral3) and Lemma \ref{le2.2}. By the arbitrariness of $x$, the first conclusion in \rm(\romannumeral4) is proven by using Lemma \ref{le2.2}.

If $\lambda=0$, it is clear $\lambda+z_{\lambda}=0\in(\{0\}\cup U_j\})$. Else $\lambda\ne0$, we have $k_{q_{j_{\lambda}},j_{\lambda}}<\nu_{q_{j_{\lambda}}}(z_{\lambda})$ and $\theta_{q_{j_{\lambda}}}(p_1p_2\cdots p_{b_{q_{j_{\lambda}},\,j_{\lambda}}})\mid\theta_{q_{j_{\lambda}}}(z_{\lambda})$ by noting $z_{\lambda}\in p_1p_2\cdots p_{b_{q_{j_{\lambda}},j_{\lambda}}}\mathbb{Z}$. Hence  $c_{j_{\lambda}}+z_{\lambda}\in U_{j_{\lambda}}$. It follows that $\lambda+z_{\lambda}=\sum_{j\in\mathcal{A}}c_j+z_{\lambda}=(c_{j_{\lambda}}+z_{\lambda})+\sum_{j\in(\mathcal{A}\setminus\{j_{\lambda}\})}c_j\in\sum_{j\in\mathcal{A}}(\{0\}\cup U_j)$. The second conclusion in \rm(\romannumeral4) is proven.	
\end{pf}

%\section{The Low Bound of $\hat\mu_{>n}$}\label{sec6}

\section{Proof of the Sufficiency of Theorem \ref{th1.3}}\label{sec7}

We first introduce the ``equi-positive", ``admissible" and ``equi-continuous" defined in \cite{AFL2019}.

\begin{definition}[\cite{AFL2019}]
Let $\Phi$ be a collection of probability measure on $[0,\,1]$.
\begin{enumerate}
	\item[\rm(\romannumeral1)] We say that $\Phi$ is an equi-positive family if there exists $\epsilon_0>0$ such that for all $x\in[0,\,1]$ and all $\nu\in\Phi$, there exists an integer $\psi_{x,\nu}\in\mathbb{Z}$ such that
	\begin{equation*}
		|\hat\nu(x+\psi_{x,\nu})|>\epsilon_0.
	\end{equation*}
	\item[\rm(\romannumeral2)] We say that $\Phi$ is an admissible family if for all $\nu\in\Phi$ or any possible weak limits $\nu$ of a sequence in $\Phi$, the set $\{x\in\mathbb{R}:\,\hat\nu(x+n)=0 \mbox{ for~all } n\in\mathbb{Z}\}$ is empty.
	\item[\rm(\romannumeral3)] We say a family of function $\Psi$ defined on $\mathbb{R}$ equi-continuous, if for $\epsilon_1>0$, there is a number $\delta>0$ such that $|f(x)-f(y)|\leqslant\epsilon_1$ holds for all $f\in\Psi$ and all $x$ and $y$ with $|x-y|<\delta$.
\end{enumerate}	
\end{definition}

\begin{prop}\label{prop5.13}
Assume that the pair $(\mathbb{P},\,\mathbb{D})$ satisfies \eqref{eq1.4}, \eqref{eq1.5} and \eqref{eq2.12-3}. Let $\mu_{\mathbb{P},\,\mathbb{D}}$ be the measure defined by \eqref{eq1.1}. Then, there are small constants $\varepsilon_2>0$ and $\omega_2>0$ such that for all $x\in\mathbb{R}$ and $n\geqslant0$, there is an integer $h_{x,n}\in\mathbb{Z}$ with $h_{0,n}=0$, such that
\begin{equation*}
	|\hat\mu_{>n}(y+x+h_{x,n})|>\frac{1}{2}\varepsilon_2,\qquad\forall \, n\geqslant0,\,~ y\in[-\omega_2,\,\omega_2].
\end{equation*}	
\end{prop}	

\begin{pf}
For any integer $n\geqslant1$, let $c_n=\max\{a\in D_n:\,n\geqslant1\}$. By {\bfseries (C1)} and \eqref{eq2.12-3}, we see that $q_n-1\leqslant c_n\leqslant p_n-1$ for any $n\geqslant1$. Hence, by  \eqref{eq2.12-3}, we get that the support of $\mu_{>n}$ satisfies
\begin{equation*}
	\sup(\mu_{>n})\subseteq\Big[0,\,\sum_{k=1}^{\infty}p_{n+1}^{-1}p_{n+2}^{-1}\cdots p_{n+k}^{-1}c_{n+k}\Big]\subseteq[0,\,1],\quad \forall\ n\geqslant0.
\end{equation*}
Thus $\Phi:=\{\mu_{>n}:\,n=0,\,1,\,2,\cdots\}$ is a family of probability measures on $[0,\,1]$.

Let $B:=\max\{c_n:\,n\geqslant1\}$. Since $d_n\geqslant1$ for any $n>0$ and the set $\cup_{n\geqslant1}D_n$ is bounded, we get that $2\leqslant B+1<\infty$.

If $p_{n+1}\geqslant2c_{n+1}+2$, then
\begin{equation*}
	\sum_{k=1}^{\infty}p_{n+1}^{-1}p_{n+2}^{-1}\cdots p_{n+k}^{-1}c_{n+k}\leqslant\frac{1}{2}\sum_{k=1}^{\infty}(c_{n+1}+1)^{-1}\cdots (c_{n+k}+1)^{-1}c_{n+k}\leqslant\frac{1}{2}<1-\frac{1}{2(B+1)}.
\end{equation*}

If $c_{n+1}<p_{n+1}<2c_{n+1}+2$. Note that $c_n\geqslant1$ and $p_n\geqslant2$ for all $n\geqslant1$, we have
\begin{equation*}
	\begin{split}
		\frac{1}{3}\leqslant p_{n+1}^{-1}c_{n+1}\leqslant& p_{n+1}^{-1}c_{n+1}+\sum_{k=3}^{\infty}p_{n+1}^{-1}p_{n+2}^{-1}\cdots p_{n+k}^{-1}c_{n+k}\\
		=&p_{n+1}^{-1}\Big(c_{n+1}+p_{n+2}^{-1}\sum_{k=3}^{\infty}p_{n+3}^{-1}\cdots p_{n+k}^{-1}c_{n+k}\Big)\\
		\leqslant&(c_{n+1}+1)^{-1}\Big(c_{n+1}+\frac{1}{p_{n+2}}\Big)\\
		\leqslant&(c_{n+1}+1)^{-1}\Big(c_{n+1}+\frac{1}{2}\Big)\\
		\leqslant&1-\frac{1}{2(B+1)}.
	\end{split}
\end{equation*}
Hence, by the fact that
\begin{equation*}
	\mu_{>n}\Big(\Big[p_{n+1}^{-1}c_{n+1},\, p_{n+1}^{-1}c_{n+1}+\sum_{k=3}^{\infty}p_{n+1}^{-1}p_{n+2}^{-1}\cdots p_{n+k}^{-1}c_{n+k}\Big]\Big)\geqslant\frac{1}{9},
\end{equation*}
we see that either $\mu_{>n}([0,\,1-\frac{1}{2(B+1)}])=1$ or $\mu_{>n}([\frac{1}{3},\,1-\frac{1}{2(B+1)}])\geqslant\frac{1}{9}$ for every $n\geqslant0$, which implies that any subsequence of $\Phi$ cannot converge to the probability measure $\frac{1}{2}(\delta_0+\delta_1)$, where $\delta_a$  denote the Dirac measure at the point $a$. Hence, \cite[Theorem 4.4]{AFL2019} shows that $\Phi$ is an admissible family. Furthermore, \cite[Theorem 3.5]{AFL2019} shows that $\Phi$ is an equi-positive family. Therefore, there exists $\varepsilon_2>0$ such that for all $x\in[0,\,1]$ and all $n\geqslant0$, there exists an integer $h_{x,n}\in\mathbb{Z}$ such that
\begin{equation}\label{eq5.38}
	|\hat\mu_{>n}(x+h_{x,n})|>\varepsilon_2.
\end{equation}	
	
Note that $|e^{i\xi}-1|\leqslant\xi$ for all $\xi\in\mathbb{R}$ and $\Phi$ is a family of probability measure on $[0,\,1]$. We have
\begin{equation*}
	\begin{split}
		|\hat\mu_{>n}(x)-\hat\mu_{>n}(y)|=\ &  \Big|\int\nolimits_{\mathbb{R}}( e^{2\pi ixt}-e^{2\pi iyt}) d\mu_{>n}(t)\Big|
		\leqslant\int\nolimits_{\mathbb{R}}\big|e^{2\pi i(x-y)t}-1\big|\, d\mu_{>n}(t)\\
		\leqslant\ & 2\pi |x-y|\int\nolimits_{\mathbb{R}} |t|\, d\mu_{>n}(t)
		\leqslant 2\pi C|x-y|,
	\end{split}
\end{equation*}
for any $n\geqslant0$ and $x,\,y\in\mathbb{R}$, where $C:=\int\nolimits_{\mathbb{R}} |t|\, d\mu_{>n}(t)<\infty$. Therefore, for any $\varepsilon>0$, taking $\delta=\frac{\varepsilon}{2\pi C}$ ensures that $|\hat\mu_{>n}(x)-\hat\mu_{>n}(y)|<\varepsilon$ for any $x,\,y\in\mathbb{R}$ satisfying $|x-y|<\delta$ and $n\geqslant1$. This means that $\hat\Phi:=\{\hat\mu_{>n}:\,n\geqslant1\}$ is equi-continuous. Hence, for the $\varepsilon_2$ in \eqref{eq5.38}, there exists $\omega_2>0$ such that for any $x\in\mathbb{R}$ and $n\geqslant0$, we have
\begin{equation*}
	|\hat\mu_{>n}(x+y+h_{x,n})-\hat\mu_{>n}(x+h_{x,n})|\leqslant\frac{1}{2}\varepsilon_2,\qquad\forall \, y\in[-\omega_2,\,\omega_2].
\end{equation*}
Combining this inequality and \eqref{eq5.38} implies
\begin{equation*}
	|\hat\mu_{>n}(x+y+h_{x,n})|\geqslant|\hat\mu_{>n}(x+h_{x,n})|-\frac{1}{2}\varepsilon_2\geqslant\frac{1}{2}\varepsilon_2,\qquad\forall \, y\in[-\omega_2,\,\omega_2].
\end{equation*}	
	
Finally, the fact that $\hat\mu_{>n}(0)=1$ for any $n\geqslant0$ implies that we may set $h_{x,n}=0$ when $x=0$. This completes the proof.
\end{pf}

For any integers $n_1$ and $n_2$ with $1\leqslant n_1<n_2$, we denote by $\mathcal{S}[n_1,\,n_2]$ the set of integers from $n_1+1$ to $n_2$, i.e.,
\begin{equation}\label{eq5.39}
	\mathcal{S}[n_1,\, n_2]:=\{i:\,n_1+1\leqslant i\leqslant n_2\}.
\end{equation}
We denote by $\mathcal{B}[n_1,\, n_2]$ the set of indices $i$ with $n_1<i\leqslant n_2$ such that $b_{q_i,\,i}>n_2$, i.e.,
\begin{equation*}
	\mathcal{B}[n_1,\, n_2]:=\{i\in\mathcal{S}[n_1,\, n_2]:\,b_{q_i,\,i}>n_2\}.
\end{equation*}
Moreover, let $H(n)$ represent the maximum value of $b_{q_i,i}$ over $\mathcal{S}[0,\, n]$, i.e.,
\begin{equation}\label{eq5.40}
	H(n):=\max\{b_{q_i,i}:\,1\leqslant i\leqslant n\},\qquad\forall n\geqslant1.
\end{equation}

\begin{prop}\label{prop5.16}
Assume that the pair $(\mathbb{P},\,\mathbb{D})$ satisfies \eqref{eq1.4}, \eqref{eq1.5} and \eqref{eq2.12-3}. Let $\mu_{\mathbb{P},\,\mathbb{D}}$ be the measure defined by \eqref{eq1.1}. Then, for any spectrum $\Lambda$ of $\bigast\limits_{i=n_1+1}^{n_2}\delta_{p_{1}^{-1}p_{2}^{-1}\cdots p_{i}^{-1}D_{i}}$ given in Theorem \ref{th5.11} \rm(\romannumeral3), one can find a subset $\{z_{\lambda}\}_{\lambda\in\Lambda}\subset\mathbb{Z}$ with $z_{0}=0$ so that
\begin{equation*}
	\Big|\hat\mu_{>H}\Big(y+\frac{\lambda}{p_1p_2\cdots p_H}+z_{\lambda}\Big)\Big|>\frac{1}{2}\varepsilon_2,\qquad\forall \, y\in[-\omega_2,\,\omega_2],\,\,\lambda\in\Lambda,
\end{equation*}
where $H:=H(n_2)$ and, $\varepsilon_2$ and $\omega_2$ are given in Proposition \ref{prop5.13}. Furthermore, the set $\{\lambda+p_1p_2\cdots p_Hz_{\lambda}:\,\lambda\in\Lambda\}\subset\sum_{i=n_1+1}^{n_2}(\{0\}\cup U_i)$ is also a spectrum of $\bigast\limits_{i=n_1+1}^{n_2}\delta_{p_{1}^{-1}p_{2}^{-1}\cdots p_{i}^{-1}D_{i}}$.	
\end{prop}	

\begin{pf}
Let $\Lambda$ be a spectrum of $\bigast\limits_{i=n_1+1}^{n_2}\delta_{p_{1}^{-1}p_{2}^{-1}\cdots p_{i}^{-1}D_{i}}$ as described in Theorem \ref{th5.11} \rm(\romannumeral3). By Proposition \ref{prop5.13}, we obtain that for any $\lambda\in\Lambda$, there exists $h_{\frac{\lambda}{p_1p_2\cdots p_H},\,H}$ with $h_{0,\,H}:=0$ such that
\begin{equation*}
	\Big|\hat\mu_{>H}\Big(y+\frac{\lambda}{p_1p_2\cdots p_H}+h_{\frac{\lambda}{p_1p_2\cdots p_H},\,H}\Big)\Big|\geqslant\frac{1}{2}\varepsilon_2,\qquad\forall \, y\in[-\omega_2,\,\omega_2],\quad\lambda\in\Lambda.
\end{equation*}
The first result of the proposition is established by taking $z_{\lambda}:=h_{\frac{\lambda}{p_1p_2\cdots p_H},\,H}$ for all $\lambda\in\Lambda$.

By the definition of $H$, we get $H=\max(\{b_{q_i,i}:\,i\in\mathcal{S}[n_1,\,n_2]\}\cup\mathcal{S}[n_1,\,n_2])$. Hence, it follows from Theorem \ref{th5.11} \rm(\romannumeral4) that the set $\{\lambda+p_1p_2\cdots p_Hz_{\lambda}:\,\lambda\in\Lambda\}\subset\sum_{i=n_1+1}^{n_2}(\{0\}\cup U_i)$ is also a spectrum of $\bigast\limits_{i=n_1+1}^{n_2}\delta_{p_{1}^{-1}p_{2}^{-1}\cdots p_{i}^{-1}D_{i}}$.
\end{pf}

%We denote by $\mathcal{A}:=\{i\in\mathcal{S}[n_1,\, n_2]:\,b_{q_i,\,i}>n_2\}$ the set of indices $i$ with $n_1<i\leqslant n_2$ such that $b_{q_i,\,i}>n_2$.

In order to prove the sufficiency of Theorem \ref{th1.3}, after choosing $C_j$ as Theorem \ref{th5.11} \rm(\romannumeral3), we need to prove that there exist infinite integers $\{n_j\}_{j=1}^n$ and $z_{\lambda}\in\mathbb{Z}$ with $z_{0}=0$ such that
\begin{equation*}
	\inf_{j\geqslant1}\Big\{\Big|\hat\mu_{>n_j}\Big(y+\frac{\lambda}{p_1p_2\cdots p_{n_j}}+z_{\lambda}\Big)\Big|:\,y\in[-\omega,\,\omega],\,\lambda\in\sum_{i=n_{j-1}+1}^{n_j}C_i\Big\}>0
\end{equation*}
for some constant $\omega>0$.

If the set $\{n\geqslant1:\,b_{q_j,j}\leqslant b_{q_n,n}~\mbox{for~all}~j\leqslant b_{q_n,n}\}$ is infinite, we can take $\{n_j\}_{j=1}^n$ to be this one, since $H(n_j)=n_j$ in this case. But the following Example \ref{e.p.5.3} shows that this set may be finite, so we also need to prove
\begin{equation*}
	\inf_{j\geqslant1}\Big\{\prod_{n=n_j+1}^{H(n_j)}\Big|\hat\delta_{p_{n_j+1}^{-1}p_{n_j+2}^{-1}\cdots p_n^{-1}D_n}\Big(y+\frac{\lambda}{p_1p_2\cdots p_{n_j}}+z_{\lambda}\Big)\Big|:\,y\in[-\omega,\,\omega],\,\lambda\in\sum_{i=n_{j-1}+1}^{n_j}C_i\Big\}>0.
\end{equation*}
This is the most complicated part of the proof of the sufficiency in Theorem \ref{th1.3}. We will address this by dividing the product into several parts and proving their uniform lower bounds in different cases. Specifically, Lemma \ref{le5.9} \rm(\romannumeral7) and \rm(\romannumeral8) show that $\#(X_q\cap\mathcal{S}[n_2,\,H(n_2)])<2G$ for some $q\in\{2,\ 3\}$. Let $X_q\cap\mathcal{S}[n_2,\,H(n_2)]=\{\ell_1,\ \ell_2,\ \cdots,\ \ell_r\}$, then $X_q\cap \mathcal{S}[n_2,\,H(n_2)]$ can be decomposed into at most $2G+1$ subsets of $X_{\widetilde{q}}$ with consecutive integers.

\begin{example}\label{e.p.5.3}
	Let $p_1=2^2$, $p_2=3^2$ and, $p_{4n+3}=3^2$, $p_{4n+4}=5^4$, $p_{4n+5}=2^2$ and $p_{4n+6}=5^2$ for all $n\geqslant0$. Further, assume that $D_{1}=\{0,\,1\}$, $D_{2}=\{0,\,1,\,2\}$ and, $D_{4n+3}=\{0,\,2,\,2\times2\}$, $D_{4n+5}=\{0,\,2\times3^4\}$, $D_{4n+5}=\{0,\,3\}$ and $D_{4n+6}=\{0,\,2\times3,\,2\times2\times3\}$ for all $n\geqslant0$. It is clear that $k_{2,1}=1$, $k_{3,2}=1$, $\nu_{3}(L_1)=0$, $\nu_{2}(L_2)=2$ and, $k_{3,4n+3}=2n+3$, $k_{2,4n+4}=2n$, $k_{2,4n+5}=2n+3$, $k_{3,4n+6}=2n+2$, $\nu_{2}(L_{4n+3})=2n+1$, $\nu_{3}(L_{4n+4})=2n$, $\nu_{3}(L_{4n+5})=2n+3$ $\nu_{3}(L_{4n+6})=2n+3$ for all $n\geqslant1$.
	
	For any integer $n\geqslant0$, it is easy to see that $4n+5<4n+6=b_{3,4n+3}$, but $b_{2,4n+5}=4(n+1)+4>4n+6=b_{3,4n+3}$. Similarly, we obtain $4n+3<4n+4=b_{2,4n+4}<4n+6=b_{3,4n+3}$, $4(n+1)+3<4(n+1)+4=b_{2,4n+5}<4(n+1)+6=b_{3,4(n+1)+3}$ and $4n+5<4n+6=b_{2,4n+6}<4(n+1)+4=b_{2,4n+5}$. This means that for any integer $n\geqslant2$, there exists an integer $m\leqslant b_{q_n,n}$ such that $b_{q_m,m}>b_{q_n,n}$. Hence, set $\{n\geqslant1:\,j\leqslant b_{q_n,n},\,b_{q_j,j}\leqslant b_{q_n,n}\}$ is finite.
\end{example}

The set $\{n\geqslant1:\,b_{q_j,j}\leqslant b_{q_n,n}~\mbox{for~every}~j\leqslant b_{q_n,n}\}$ is not infinite means the following assumption {\bfseries (C6)} holds:
\begin{enumerate}	
	\item[\bfseries (C6)] There exists an integer $N_4>0$ such that $H(n)>n$ for any $n\geqslant N_4$, where $H(n)$ is defined in \eqref{eq5.40}.
\end{enumerate}

For a real number $r$ and an integer $q\in\{2,\ 3\}$, we use $\left\Vert r\right\Vert_{\frac{1}{q}}$ to denote the distance from $r$ to the set $\frac{\mathbb{Z}\setminus q\mathbb{Z}}{q}$, i.e.
\begin{equation}\label{eq6.2-1}
	\left\Vert r\right\Vert_{\frac{1}{2}}=\min_{z\in\mathbb{Z}}\Big\{\Big|r-z-\frac{1}{2}\Big|\Big\}.
\end{equation}
\begin{equation}\label{eq6.2}
	\left\Vert r\right\Vert_{\frac{1}{3}}=\min_{z\in\mathbb{Z}}\Big\{\Big|r-z-\frac{1}{3}\Big|,\,\Big|r-z+\frac{1}{3}\Big|\Big\}.
\end{equation}
Since the set $\cup_{i=1}^{\infty}D_i$ is bounded, and for every $i>0$, the function $m_{d_i^{-1}D_i}$ is continuous and has period $1$. Hence
\begin{equation}\label{eq6.3}
	\varepsilon(a):=\inf\left\{m_{d_i^{-1}D_i}(x):\,i\ge1,\,\left\Vert x\right\Vert_{\frac{1}{q_i}}\geqslant a\right\}>0,\quad\forall\ a\in\Big(0,\,\frac{1}{3}\Big).
\end{equation}
Clearly, the function $\varepsilon(~\cdot~)$ is increasing.

The proof of the following lemma is straightforward, requiring only the formula for the sum of a geometric series. However, this lemma will be used repeatedly in what follows and highlights the crucial role of equation \eqref{eq1.4} in the proof of the sufficiency part of Theorem \ref{th1.3}.

%\begin{equation*}
%	\Big|\sum_{i\in B}\frac{d_{M+j}c_i}{p_1p_2\cdots p_{M+j}}\Big|\leqslant\Bigg\{\begin{array}{ll}
	%		\frac{\theta_q(d_{M+j})}{\theta_q(p_{M+1}p_{M+2}\cdots p_{M+j})}\cdot\frac{q}{q-1}\cdot q^{k_{q,i_1}-k_{q,M+j}-1},&\mbox{if}~M+j\in X_q,\\
	%		\frac{\theta_q(d_{M+j})}{\theta_q(p_{M+1}p_{M+2}\cdots p_{M+j})}\cdot\frac{q}{q-1}\cdot q^{k_{q,i_1}-\nu_{q}(L_{M+j})},&\mbox{if}~M+j\in X_{\widetilde{q}}
	%	\end{array}
%\end{equation*}

\begin{lemma}\label{le5.14}
Assume that the pair $(\mathbb{P},\,\mathbb{D})$ satisfies \eqref{eq1.4}, \eqref{eq1.5} and \eqref{eq2.12-3}. Let $q\in\{2,\,3\}$. Assume that $B\subset X_q$ is a finite nonempty set and $M\geqslant\max B$ is an integer. Let $\{c_i\}_{i\in B}$ be an integer sequence satisfying $|c_i|\leqslant q^{k_{q,i}}\theta_q(p_1p_2\cdots p_M)$ for every $i\in B$. Then, we have
\begin{equation*}
	\Big|\sum_{i\in B}\frac{c_i}{p_1p_2\cdots p_{M+j}}\Big|\leqslant
	\frac{1}{\theta_q(p_{M+1}p_{M+2}\cdots p_{M+j})}\cdot\frac{q}{q-1}\cdot q^{k_{q,i_1}-\nu_q(p_1p_2\cdots p_{M+j})},\qquad\forall\, j\geqslant0,
\end{equation*}
where $i_1\in B$ satisfies $k_{q,i_1}=\max\{k_{q,i}:\,i\in B\}$ and $\theta_q(p_{M+1}p_{M+2}\cdots p_{M+j}):=1$ if $j=0$.		
\end{lemma}

\begin{pf}
{\bfseries (C3)} shows $k_{q,i}\geqslant0$, \eqref{eq1.4} shows that $k_{q,i}\ne k_{q,j}$ for any $i\neq j\in X_q$. Hence
\begin{equation}\label{eq5.39-1}
	\sum_{i\in B}q^{k_{q,i}}\leqslant\sum_{l=0}^{k_{q,i_1}}q^{l}\leqslant\frac{q^{k_{q,i_1}+1}}{q-1}.
\end{equation}
Moreover, by the assumption $|c_i|\leqslant q^{k_{q,i}}\theta_q(p_1p_2\cdots p_M)$ for every $i\in B$, we get
\begin{equation*}
	\Big|\sum_{i\in B}\frac{c_i}{p_1p_2\cdots p_{M+j}}\Big|\leqslant\frac{1}{\theta_q(p_{M+1}p_{M+2}\cdots p_{M+j})}\cdot\sum_{i\in B}q^{k_{q,i}-\nu_q(p_1p_2\cdots p_{M+j})},\qquad\forall\, j\geqslant0.
\end{equation*}
Combining the above formula with \eqref{eq5.39-1} leads to the conclusion of the lemma.	
\end{pf}

%The following lemma illustrates another crucial role of equation \eqref{eq1.4} in the proof of the sufficiency part of Theorem \ref{th1.3}.

\begin{lemma}\label{le5.15}
	Assume that the pair $(\mathbb{P},\,\mathbb{D})$ satisfies \eqref{eq1.4}, \eqref{eq1.5} and \eqref{eq2.12-3}. Let $q\in\{2,\,3\}$. Suppose $B\subset X_q$ is a finite nonempty set and $M\geqslant1$ is an integer. Let $\{z_i\}_{i\in B}\subset(\mathbb{Z}\setminus q\mathbb{Z})$ be an integer sequence satisfying $z_i\geqslant M$ for all $i\in B$. If $\sum_{i\in B}\frac{d_{M+j}q^{k_{q,i}}\theta_q(p_1p_2\cdots p_{z_i})}{p_1p_2\cdots p_{M}}\not\in\mathbb{Z}$, then
	\begin{equation*}
		\left\{\begin{array}{ll}k_{q,i_2}<k_{q,M}~~\mbox{and}~~\Big\Vert\sum_{i\in B}\frac{d_{M+j}q^{k_{q,i}}\theta_q(p_1p_2\cdots p_{z_i})}{p_1p_2\cdots p_{M}}\Big\Vert_{\frac{1}{q}}\geqslant q^{k_{q,i_2}-k_{q,M}-1},&\mbox{if}~M\in X_q,\\
			k_{q,i_2}<\nu_{\widetilde{q}}(L_{M})~~\mbox{and}~~\Big\Vert\sum_{i\in B}\frac{d_{M+j}q^{k_{q,i}}\theta_q(p_1p_2\cdots p_{z_i})}{p_1p_2\cdots p_{M}}\Big\Vert_{\frac{1}{\widetilde{q}}}\geqslant \frac{1}{\widetilde{q}}\cdot q^{k_{q,i_2}-\nu_{q}(L_{M})},&\mbox{if}~M\in X_{\widetilde{q}},
		\end{array}\right.
	\end{equation*}
	where $i_2\in B$ satisfies $k_{q,i_2}=\min\{k_{q,i}:\,i\in B\}$.
\end{lemma}

\begin{pf}
Since $z_i\geqslant M$, we get $\frac{d_{M+j}q^{k_{q,i}}\theta_q(p_1p_2\cdots p_{z_i})}{p_1p_2\cdots p_{M}}\in q^{k_{q,i}-k_{q,M}-1}(\mathbb{Z}\backslash q\mathbb{Z})$ if $M\in X_q$ and $\frac{d_{M+j}q^{k_{q,i}}\theta_q(p_1p_2\cdots p_{z_i})}{p_1p_2\cdots p_{M}}\in q^{k_{q,i}-\nu_{q}(L_{M})}(\mathbb{Z}\backslash q\mathbb{Z})$ if $M\in X_{\widetilde{q}}$. If $B\backslash\{i_2\}$ is nonempty, \eqref{eq1.4} shows $k_{q,i}>k_{q,i_2}$ for all $i\in (B\backslash\{i_2\})$. It follows that
\begin{equation}\label{eq6.88-1}
	\left\{\begin{array}{ll}\sum_{i\in B}\frac{d_{M+j}q^{k_{q,i}}\theta_q(p_1p_2\cdots p_{z_i})}{p_1p_2\cdots p_{M}}\in q^{k_{q,i_2}-k_{q,M}-1}(\mathbb{Z}\backslash q\mathbb{Z}),&\mbox{if}~M\in X_q,\\
		\sum_{i\in B}\frac{d_{M+j}q^{k_{q,i}}\theta_q(p_1p_2\cdots p_{z_i})}{p_1p_2\cdots p_{M}}\in q^{k_{q,i_2}-\nu_{q}(L_{M})}(\mathbb{Z}\backslash q\mathbb{Z}),&\mbox{if}~M\in X_{\widetilde{q}}.
	\end{array}\right.	
\end{equation}
Hence, it follows from \eqref{eq1.4}, \eqref{eq6.2-1} and \eqref{eq6.2} that the lemma holds.	
\end{pf}

Since the sequence $\{d_n\}_{n=1}^{\infty}$ is bounded, there exists an integer $c_3>0$ and $c_4>0$ such that
\begin{equation}\label{eq6.24}
	2^{c_3}>3^{5G}D,
\end{equation}
where $D$ is defined as in \eqref{eq1.6-}. It is clear that $c_3>5G$.

\begin{prop}\label{prop6.10}
	Assume that the pair $(\mathbb{P},\,\mathbb{D})$ satisfies \eqref{eq1.4}, \eqref{eq1.5}, \eqref{eq2.12-3} and {\bfseries (C6)}. Let $n_1$ and $n_2$ be integers satisfying $1\leqslant n_1<n_2$. If $\sup\{\nu_{2}(p_{n}),\,\nu_{3}(p_{n}):\,n>n_2\}<c_3$, then for any integer $\mathcal{N}>n_2$, one can find a spectrum $\Lambda:=\Lambda(n_1,\,n_2,\,\mathcal{N})$ of $\bigast\limits_{i\in\mathcal{B}[n_1,\,n_2]}\delta_{p_1^{-1}p_2^{-1}\cdots p_i^{-1}D_i}$ satisfying Theorem \ref{th5.11} (\romannumeral3) such that
	\begin{equation*}
		\prod_{j=1}^{\mathcal{N}-n_2}\Big|m_{D_{n_2+j}}\Big(\frac{\lambda}{p_1p_2\cdots p_{n_2+j}}\Big)\Big|\geqslant\Big[\varepsilon\Big(\frac{1}{3^{2(\mathcal{N}-n_2+1)c_3}}\Big)\Big]^{\mathcal{N}-n_2},\qquad\forall \,\lambda\in\Lambda,
	\end{equation*}
	where $k_{q_n,n}$, $l_{q_n,n}$, $U_n$, $\mathcal{S}$, $H$  and $c_3$ are defined in \eqref{eq1.2}, \eqref{eq5.8}, \eqref{eq5.28}, \eqref{eq5.39}, \eqref{eq5.40} and \eqref{eq6.24}, respectively.		
\end{prop}

\begin{pf}
	Let $q$ and $\widetilde{q}$ be two integers such that $\{q,\,\widetilde{q}\}=\{2,\,3\}$. By {\bfseries (C6)}, we get $H(n_2)>n_2$, so $\mathcal{B}[n_1,\,n_2]\neq\emptyset$.  Let $\mathcal{B}:=\mathcal{B}[n_1,\,n_2]$, and partition it into two sets by setting
	\begin{equation}\label{eq6.88}
		\mathcal{B}_{q}:=\mathcal{B}[n_1,\,n_2]\cap X_{q}~~\mbox{and}~~
		\mathcal{B}_{\widetilde{q}}:=\mathcal{B}[n_1,\,n_2]\cap X_{\widetilde{q}}.
	\end{equation}

	For any $j\in\mathcal{B}$, we get $n_2<b_{q_j,j}$, so $\nu_{q_j}(p_{j+1}p_{j+2}\cdots p_{n_2})\leqslant\nu_{q_j}(p_{j+1}p_{j+2}\cdots p_{b_{q_j,j}})$. It follows from Lemma \ref{le5.9} \rm(\romannumeral8) that
	\begin{equation}\label{eq6.93-1}
		\nu_{q_j}(p_{j+1}p_{j+2}\cdots p_{n_2})\leqslant G-1,\qquad\forall\, j\in\mathcal{B}.
	\end{equation}
	
	For all $i\in\mathcal{B}_{q}$ and $j\geqslant1$ with $n_2+j\in X_q$, we get
	By \eqref{eq1.3-1} and \eqref{eq6.93-1}, we get
	\begin{equation}\label{eq6.93}
		\begin{split}
			k_{q,n_2+j}-k_{q,i}=&\nu_{q}(p_{i+1}p_{i+2}\cdots p_{n_2})+\nu_{q}(p_{n_2+1}p_{n_2+2}\cdots p_{n_2+j})-\nu(qd_{n_2+j})+\nu(qd_{i})\\
			<&G-1+(n_2+j-n_2)c_3+G\\
			\leqslant&(j+1)c_3,
		\end{split}
	\end{equation}
	where the first equality comes from \eqref{eq1.3-1}, and the first inequality comes from \eqref{eq6.93-1} and the assumption $\nu_{q}(p_{n_2+j})<c_3$ for all $j>n_0+n_1$. Also, we have
	\begin{equation}\label{eq6.94}
		\nu_{q}(L_{n_2+j})-k_{q,i}=\nu_{q}(p_{i+1}p_{i+2}\cdots p_{n_2+j})-\nu_{q}(d_{n_2+j})+\nu_{q}(qd_{i})<(j+1)c_3
	\end{equation}
	for all $i\in\mathcal{B}_{q}$ and $j\geqslant1$ with $n_2+j\in X_{\widetilde{q}}$.

	Analogous to \eqref{eq6.93} and \eqref{eq6.94}, one can derive the corresponding relations
	\begin{equation}\label{eq6.96}
		k_{\widetilde{q},n_2+j}-k_{\widetilde{q},i}<(j+1)c_3
	\end{equation}
	for all $i\in\mathcal{B}_{\widetilde{q}}$ and $j\geqslant1$ with $n_2+j\in X_{\widetilde{q}}$, and
	\begin{equation}\label{eq6.97}
		\nu_{\widetilde{q}}(L_{n_2+j})-k_{\widetilde{q},i}<(j+1)c_3
	\end{equation}
	for all $i\in\mathcal{B}_{\widetilde{q}}$ and $j\geqslant1$ with $n_2+j\in X_q$.

	We first construct the spectrum $\Lambda$. For any $i\in\mathcal{B}$, choose $\eta_i\in\mathbb{Z}$ such that $\eta_i\geqslant\max\{b_{q_i,i},\,\mathcal{N}\}$. For any $i\in\mathcal{B}$, we take
	\begin{equation}\label{eq6.98}
		x_i=\begin{cases}
			q^{k_{q,i}}\theta_{q}(p_1p_2\cdots p_{\eta_i}), & i \in\mathcal{B}_{q}\\
			\widetilde{q}^{k_{\widetilde{q},i}}\theta_{\widetilde{q}}(p_1p_2\cdots p_{\eta_i}), & i\in\mathcal{B}_{\widetilde{q}}
		\end{cases}
	\end{equation}
	and
	\begin{equation}\label{eq6.99}
		C_i:=\begin{cases}
			\{0,\,x_i\},  & \mbox{if } i\in X_2\\
			\{0,\,x_i,\,-x_i\},  & \mbox{if } i\in X_3.
		\end{cases}
	\end{equation}
	Let $\Lambda:=\sum_{i\in\mathcal{B}}C_i$. By Theorem \ref{th5.11} \rm(\romannumeral4), we see that $\Lambda$ is a spectrum of $\bigast\limits_{i\in\mathcal{B}}\delta_{p_1^{-1}p_2^{-1}\cdots p_{i}^{-1}D_i}$.

	{\bf Claim}. For any $1\leqslant j\leqslant\mathcal{N}-n_2$. If $n_2+j\in X_{q}$, then
	\begin{equation}\label{eq6.104}
		\Big\Vert\frac{d_{n_2+j}\lambda}{p_1p_2\cdots p_{n_2+j}}\Big\Vert_{\frac{1}{q}}>\frac{1}{3^{2(j+1)c_3}},\qquad\forall\,\lambda\in\Lambda.
	\end{equation}
	
	{\bf Proof of Claim}. Choose $\lambda=\sum_{i\in\mathcal{B}}c_i\in\Lambda$, where $c_i\in C_i$. If $\sum_{i\in\mathcal{B}_{q}}\frac{d_{n_2+j}c_i}{p_1p_2\cdots p_{n_2+j}}\not\in\mathbb{Z}$, then the set $\{i\in\mathcal{B}_{q}:\,c_i\neq0\}$ is nonempty, so there exists $i_1\in\mathcal{B}_{q}$ such that
	\begin{equation*}
		k_{q,i_1}:=\min\{k_{q,i}:\,i\in\mathcal{B}_{q},\, c_i\neq0\}.
	\end{equation*}
	If $\sum_{i\in\mathcal{B}_{\widetilde{q}}}\frac{d_{n_2+j}c_i}{p_1p_2\cdots p_{n_2+j}}\not\in\mathbb{Z}$, then the set $\{i\in\mathcal{B}_{\widetilde{q}}:\,c_i\neq0\}$ is nonempty, so there exists $j_1\in\mathcal{B}_{\widetilde{q}}$ such that
	\begin{equation*}
		k_{\widetilde{q},j_1}:=\min\{k_{\widetilde{q},i}:\,i\in\mathcal{B}_{\widetilde{q}},\, c_i\neq0\}.
	\end{equation*}

	If $\sum_{i\in\mathcal{B}_{q}}\frac{d_{n_2+j}c_i}{p_1p_2\cdots p_{n_2+j}}\in\mathbb{Z}$ and $\sum_{i\in\mathcal{B}_{\widetilde{q}}}\frac{d_{n_2+j}c_i}{p_1p_2\cdots p_{n_2+j}}\in\mathbb{Z}$, we have
	\begin{equation*}
		\Big\Vert\sum_{i\in\mathcal{B}_{q}}\frac{d_{n_2+j}c_i}{p_1p_2\cdots p_{n_2+j}}+\sum_{i\in\mathcal{B}_{\widetilde{q}}}\frac{d_{n_2+j}c_i}{p_1p_2\cdots p_{n_2+j}}\Big\Vert_{\frac{1}{q}}=\frac{1}{q}.
	\end{equation*}

	If $\sum_{i\in\mathcal{B}_{q}}\frac{d_{n_2+j}c_i}{p_1p_2\cdots p_{n_2+j}}\not\in\mathbb{Z}$ and $\sum_{i\in\mathcal{B}_{\widetilde{q}}}\frac{d_{n_2+j}c_i}{p_1p_2\cdots p_{n_2+j}}\in\mathbb{Z}$, we have $k_{q,i_1}<k_{q,n_2+j}$ and
	\begin{equation*}
		\Big\Vert\sum_{i\in\mathcal{B}_{q}}\frac{d_{n_2+j}c_i}{p_1p_2\cdots p_{n_2+j}}+\sum_{i\in\mathcal{B}_{\widetilde{q}}}\frac{d_{n_2+j}c_i}{p_1p_2\cdots p_{n_2+j}}\Big\Vert_{\frac{1}{q}}=\Big\Vert\sum_{i\in\mathcal{B}_{q}}\frac{d_{n_2+j}c_i}{p_1p_2\cdots p_{n_2+j}}\Big\Vert_{\frac{1}{q}}\geqslant q^{k_{q,i_1}-k_{q,n_2+j}-1},
	\end{equation*}
	where the last inequality follows from Lemma \ref{le5.15}.

	If $\sum_{i\in\mathcal{B}_{q}}\frac{d_{n_2+j}c_i}{p_1p_2\cdots p_{n_2+j}}\in\mathbb{Z}$ and $\sum_{i\in\mathcal{B}_{\widetilde{q}}}\frac{d_{n_2+j}c_i}{p_1p_2\cdots p_{n_2+j}}\not\in\mathbb{Z}$, we have $k_{\widetilde{q},j_1}<\nu_{\widetilde{q}}(L_{n_2+j})$ and
	\begin{equation*}
		\Big\Vert\sum_{i\in\mathcal{B}_{q}}\frac{d_{n_2+j}c_i}{p_1p_2\cdots p_{n_2+j}}+\sum_{i\in\mathcal{B}_{\widetilde{q}}}\frac{d_{n_2+j}c_i}{p_1p_2\cdots p_{n_2+j}}\Big\Vert_{\frac{1}{q}}=\Big\Vert\sum_{i\in\mathcal{B}_{\widetilde{q}}}\frac{d_{n_2+j}c_i}{p_1p_2\cdots p_{n_2+j}}\Big\Vert_{\frac{1}{q}}\geqslant q^{-1}\cdot\widetilde{q}^{k_{\widetilde{q},j_1}-\nu_{\widetilde{q}}(L_{n_2+j})},
	\end{equation*}
	where the last inequality follows from Lemma \ref{le5.15}.

	If $\sum_{i\in\mathcal{B}_{q}}\frac{d_{n_2+j}c_i}{p_1p_2\cdots p_{n_2+j}}\not\in\mathbb{Z}$ and $\sum_{i\in\mathcal{B}_{\widetilde{q}}}\frac{d_{n_2+j}c_i}{p_1p_2\cdots p_{n_2+j}}\not\in\mathbb{Z}$, it follows from \eqref{eq6.88-1} that
	\begin{equation*}
		\sum_{i\in\mathcal{B}_{q}}\frac{d_{n_2+j}c_i}{p_1p_2\cdots p_{n_2+j}}\in \frac{\mathbb{Z}\setminus q\mathbb{Z}}{q^{k_{q,n_2+j}+1-k_{q,i_1}}}~~\mbox{and}~~\sum_{i\in\mathcal{B}_{\widetilde{q}}}\frac{d_{n_2+j}c_i}{p_1p_2\cdots p_{n_2+j}}\in\frac{\mathbb{Z}\setminus\widetilde{q}\mathbb{Z}}{\widetilde{q}^{\nu_{\widetilde{q}}(L_{n_2+j})-k_{\widetilde{q},j_1}}}.
	\end{equation*}
	Hence, we obtain that $k_{q,n_2+j}>k_{q,i_1}$, $\nu_{\widetilde{q}}(L_{n_2+j})>k_{\widetilde{q},j_1}$ and
	\begin{equation*}
		\sum_{i\in\mathcal{B}_{q}}\frac{d_{n_2+j}c_i}{p_1p_2\cdots p_{n_2+j}}+\sum_{i\in\mathcal{B}_{\widetilde{q}}}\frac{d_{n_2+j}c_i}{p_1p_2\cdots p_{n_2+j}}\in\frac{\widetilde{q}^{\nu_{\widetilde{q}}(L_{n_2+j})-k_{\widetilde{q},j_1}}(\mathbb{Z}\setminus q\mathbb{Z})+q^{k_{q,n_2+j}+1-k_{q,i_1}}(\mathbb{Z}\setminus\widetilde{q}\mathbb{Z})}{q^{k_{q,n_2+j}+1-k_{q,i_1}}\cdot \widetilde{q}^{\nu_{\widetilde{q}}(L_{n_2+j})-k_{\widetilde{q},j_1}}}.
	\end{equation*}
	Since $k_{q,n_2+j}>k_{q,i_1}$ and $\nu_{\widetilde{q}}(L_{n_2+j})>k_{\widetilde{q},j_1}$, we get
	\begin{equation*}
		\widetilde{q}^{\nu_{\widetilde{q}}(L_{n_2+j})-k_{\widetilde{q},j_1}}(\mathbb{Z}\setminus q\mathbb{Z})+q^{k_{q,n_2+j}+1-k_{q,i_1}}(\mathbb{Z}\setminus\widetilde{q}\mathbb{Z})\in[\mathbb{Z}\setminus(q\mathbb{Z}\cup\widetilde{q}\mathbb{Z})].
	\end{equation*}
	Thus, we have
	\begin{equation*}
		\Big\Vert\sum_{i\in\mathcal{B}_{q}}\frac{d_{n_2+j}c_i}{p_1p_2\cdots p_{n_2+j}}+\sum_{i\in\mathcal{B}_{\widetilde{q}}}\frac{d_{n_2+j}c_i}{p_1p_2\cdots p_{n_2+j}}\Big\Vert_{\frac{1}{q}}\geqslant\frac{1}{q^{k_{q,n_2+j}+1-k_{q,i_1}}\cdot \widetilde{q}^{\nu_{\widetilde{q}}(L_{n_2+j})-k_{\widetilde{q},j_1}}}.
	\end{equation*}
	From the inequalities arising in the four cases above, together with \eqref{eq6.93} and \eqref{eq6.97}, it follows that
	\begin{equation*}
		\Big\Vert\sum_{i\in\mathcal{B}_{q}}\frac{d_{n_2+j}c_i}{p_1p_2\cdots p_{n_2+j}}+\sum_{i\in\mathcal{B}_{\widetilde{q}}}\frac{d_{n_2+j}c_i}{p_1p_2\cdots p_{n_2+j}}\Big\Vert_{\frac{1}{q}}\geqslant\frac{1}{q^{(j+1)c_3}\cdot \widetilde{q}^{(j+1)c_3}}>\frac{1}{3^{2(j+1)c_3}}.
	\end{equation*}
	With this, the proof of the claim is finished.

	In view of the symmetry between $q$ and $\widetilde{q}$, formula \eqref{eq6.104} still holds under the assumption $n_2+j\in X_{\widetilde{q}}$. Namely, for any $1\leqslant j\leqslant\mathcal{N}-n_2$, if $n_2+j\in X_{\widetilde{q}}$, then
	\begin{equation}\label{eq6.100-1}
		\Big\Vert\frac{d_{n_2+j}\lambda}{p_1p_2\cdots p_{n_2+j}}\Big\Vert_{\frac{1}{\widetilde{q}}}>\frac{1}{3^{2(j+1)c_3}},\qquad\forall\,\lambda\in\Lambda.
	\end{equation}
	
	By \eqref{eq6.104} and \eqref{eq6.100-1}, we get $\big|m_{p_1^{-1}p_2^{-1}\cdots p_{n_2+j}^{-1}D_{n_2+j}}(\lambda)\big|\geqslant\varepsilon\big(\frac{1}{3^{2(j+1)c_3}}\big)$ for all $\lambda\in\Lambda$ and $1\leqslant j\leqslant\mathcal{N}-n_2$. Hence, we have
	\begin{equation*}
		\prod_{j=1}^{\mathcal{N}-n_2}\Big|m_{D_{n_2+j}}\Big(\frac{\lambda}{p_1p_2\cdots p_{n_2+j}}\Big)\Big|\geqslant\Big[\varepsilon\Big(\frac{1}{3^{2(\mathcal{N}-n_2+1)c_3}}\Big)\Big]^{\mathcal{N}-n_2},\qquad\forall \,\lambda\in\Lambda
	\end{equation*}
	by noting that $\varepsilon(x)$ is an non-decreasing function of $x$. This proves our lemma.
\end{pf}

For any integer $n\geqslant0$ and real number $\delta>0$, let
\begin{equation}\label{eq6.4}
	\widetilde{\Lambda}[n;~\delta]:=\Big\{z\in\mathbb{Z}:\,\Big|\frac{z}{p_1p_2\cdots p_{n+1}}\Big|\leqslant\delta\Big\}.
\end{equation}

\begin{prop}\label{prop6.3}
	Assume that the pair $(\mathbb{P},\,\mathbb{D})$ satisfies \eqref{eq1.4}, \eqref{eq1.5}, \eqref{eq2.12-3} and {\bfseries (C6)}. Let $n_1$ and $n_2$ be integers satisfying $1\leqslant H(n_1)\leqslant n_2$. Then, there exists $i'\in\mathcal{S}[n_1,\,n_2]$ satisfying $b_{q_{i'},i'}=H(n_2)$, which implies that the set $\mathcal{B}[n_1;~n_2]\cap X_{q_{i'}}$ is nonempty. Furthermore, suppose that $M_1$ and $M_2$ are integers satisfying $n_2\leqslant M_1<M_2\leqslant H(n_2)$. For any nonempty subset $\mathfrak{B}\subset[\mathcal{B}[n_1;~n_2]\cap X_{q_{i'}}]$ and any integer $i\in\mathfrak{B}$, define
	\begin{equation}\label{eq6.5}
		\Lambda_1:=
		\begin{cases}
			\sum_{i\in\mathfrak{B}}\left\{0,\,2^{k_{2,i}}\theta_{2}(p_1p_2\cdots p_{\sigma_i})\right\},&\mbox{if}~q=2,\\
			\sum_{i\in\mathfrak{B}}\left\{0,\,3^{k_{3,i}}\theta_{3}(p_1p_2\cdots p_{\sigma_i}),\,-3^{k_{3,i}}\theta_{3}(p_1p_2\cdots p_{\sigma_i})\right\},&\mbox{if}~q=3,
		\end{cases}
	\end{equation}
	where $\sigma_i\geqslant M_2$ for $i\in\mathfrak{B}$. Moreover, choose $s_1\in\mathfrak{B}$ such that $k_{q_{i'},s_1}=\min\{k_{q_{i'},i}:\,i\in\mathfrak{B}\}$. Then we have
	\begin{equation*}
		\prod_{j=1}^{M_2-M_1}\Big|m_{D_{M_1+j}}\Big(\frac{\lambda}{p_1p_2\cdots p_{M_1+j}}\Big)\Big|\geqslant\varepsilon_1\cdot\Big[\varepsilon\Big(\frac{1}{\widetilde{q}\cdot q^{3G}}\Big)\Big]^{2G+\#\big\{j\in\mathbb{Z}:\,1\leqslant j\leqslant M_2-M_1,\,k_{q,s_1}<\nu_{q_{i'}}(L_{M_1+j})\big\}}
	\end{equation*}
	for all $\lambda\in\Lambda_1+\widetilde{\Lambda}\big[M_1;~\frac{1}{D\cdot  q^{3G}}\big]$, where $k_{q_n,n}$, $l_{q_n,n}$, $\varepsilon_1$, $\mathcal{S}[n_1,\,n_2]$, $H(n_2)$ and $\varepsilon(\cdot)$ are defined in \eqref{eq1.2}, \eqref{eq5.8}, Lemma \ref{le2.9}, \eqref{eq5.39}, \eqref{eq5.40} and \eqref{eq6.3}, respectively.	
\end{prop}

\begin{pf}
	By {\bfseries (C6)}, we get $n_2<H(n_2)$, so $H(n_1)<H(n_2)$. Hence, we can find an integer $i'$ with $n_1<i'\leqslant n_2$ such that $b_{q_{i'},i'}=H(n_2)$. Since $n_2<H(n_2)$, we get $n_2<b_{q_{i'},i'}$, so $i'\in(\mathcal{B}[n_1;~n_2]\cap X_{q_{i'}})$. Therefore $\mathcal{B}[n_1;~n_2]\cap X_{q_{i'}}\neq\emptyset$. Let
	\begin{equation*}
		q:=q_{i'},~~\widetilde{q}:=\widetilde{q}_{i'}~~\mbox{and}~~H:=H(n_2).
	\end{equation*}
	Thus $\nu_{q}(p_{n_2+1}p_{n_2+2}\cdots p_H)\leqslant\nu_{q}(p_{i'+1}p_{i'+2}\cdots p_{b_{q,i'}})$. By Lemma \ref{le5.9} \rm(\romannumeral8), we have
	\begin{equation}\label{eq6.7}
		\nu_{q}(p_{n_2+1}p_{n_2+2}\cdots p_H)\leqslant G-1.
	\end{equation}
	Combining the above with Lemma \ref{le5.9} \rm(\romannumeral7), we have
	\begin{equation}\label{eq6.8}
		\#\{i\in X_{q}:\,n_2<i\leqslant H\}<2G.
	\end{equation}
	Moreover, the definition of $\mathcal{B}[n_1,\, n_2]$ shows $n_2<b_{q,\,i}$ for all $i\in\mathfrak{B}$. Hence, we get $\nu_{q}(p_{i+1}p_{i+2}\cdots p_{n_2})\leqslant\nu_{q}(p_{i+1}p_{i+2}\cdots p_{b_{q,i}})$ for all $i\in\mathfrak{B}$. Applying By Lemma \ref{le5.9} \rm(\romannumeral8) again, we get
	\begin{equation}\label{eq6.9-2}
		\nu_{q}(p_{i+1}p_{i+2}\cdots p_{n_2})\leqslant G-1,\qquad\forall\,i\in\mathfrak{B}.
	\end{equation}
	Combining the above with \eqref{eq6.7}, we get
	\begin{equation}\label{eq6.9-1}
		\nu_{q}(p_{i+1}p_{i+2}\cdots p_{H})\leqslant 2(G-1),\qquad\forall\,i\in\mathfrak{B}.
	\end{equation}

	For every $i\in\mathfrak{B}$ and $1\leqslant j\leqslant H-M_1$, if $M_1+j\in X_q$, then
	\begin{equation}\label{eq6.9}
		\begin{split}
			k_{q,M_1+j}-k_{q,i}=&\nu_{q}(p_{i+1}p_{i+2}\cdots p_{M_1+j})-1-\nu_{q}(d_{M_1+j})+\nu_{q}(q d_{i})\\
			\leqslant&\nu_{q}(p_{i+1}p_{i+2}\cdots p_{H})-1+G \\
			\leqslant&3(G-3),
		\end{split}
	\end{equation}
	where the first inequality holds since $\nu_{q}(q d_{i})\leqslant G$, and the last inequality follows from \eqref{eq6.9-1}. If $M_1+j\in X_{\widetilde{q}}$, similar to \eqref{eq6.9}, we have
	\begin{equation}\label{eq6.10}
		\begin{split}
			\nu_{q}(L_{M_1+j})-k_{q,i}=&\nu_{q}(p_{i+1}p_{i+2}\cdots p_{M_1+j})-\nu_{q}(d_{M_1+j})+\nu_{q}(q d_{s_1})\\
			\leqslant&3G-2.
		\end{split}
	\end{equation}

	Every element $\lambda\in\Lambda_1+\widetilde{\Lambda}\big[M_1;~\frac{1}{D\cdot q^{3G}}\big]$, can be decomposed as $\lambda=\lambda_1+\lambda_2$, with $\lambda_1\in\Lambda_1$ and  $\lambda_2\in\widetilde{\Lambda}\big[M_1;~\frac{1}{D\cdot q^{3G}}\big]$. Furthermore, $\lambda_1$ can be written as $\lambda_1=\sum_{i\in\mathfrak{B}}c_i$, where $c_i\in\Bigg\{\begin{array}{ll}
		\{0,\,2^{k_{2,i}}\theta_{2}(p_1p_2\cdots p_{\sigma_i})\},&\mbox{if}~q=2,\\
		\{0,\,3^{k_{3,i}}\theta_{3}(p_1p_2\cdot p_{\sigma_i}),\,-3^{k_{3,i}}\theta_{3}(p_1p_2\cdots p_{\sigma_i})\},&\mbox{if}~q=3.
	\end{array}$

	By the definition of $\widetilde{\Lambda}\big[M_1;~\frac{1}{D\cdot q^{3G}}\big]$, we have
	\begin{equation}\label{eq6.14}
		\Big|\frac{d_{M_1+j}\lambda_2}{p_1p_2\cdots p_{M_1+j}}\Big|\leqslant\Big|\frac{d_{M_1+1}\lambda_2}{p_1p_2\cdots p_{M_1+1}}\Big|\leqslant\frac{d_{M_1+j}}{D\cdot q^{3G}}\leqslant\frac{1}{q^{3G}},\qquad\forall\,j\geqslant1.
	\end{equation}	
	
	{\bf Claim I}. For any $1\leqslant j\leqslant M_2-M_1$ with $M_1+j\in X_{q}$, we have
	\begin{equation}\label{eq6.13}
		\Big\Vert\frac{d_{M_1+j}\lambda}{p_1p_2\cdots p_{M_1+j}}\Big\Vert_{\frac{1}{q}}\geqslant\frac{1}{q^{3G}}.
	\end{equation}
	
	{\bf Proof of Claim I}. Fix $1\leqslant j\leqslant M_2-M_1$ satisfying $M_1+j\in X_q$.
	
	{\bf Case I-I}. $\frac{d_{M_1+j}\lambda_1}{p_1p_2\cdots p_{M_1+j}}\in\mathbb{Z}$.

	The assumption of Case I-I shows
	\begin{equation*}
		\Big\Vert\frac{d_{M_1+j}\lambda}{p_1p_2\cdots p_{M_1+j}}\Big\Vert_{\frac{1}{q}}=\Big\Vert\frac{d_{M_1+j}\lambda_1}{p_1p_2\cdots p_{M_1+j}}+\frac{d_{M_1+j}\lambda_2}{p_1p_2\cdots p_{M_1+j}}\Big\Vert_{\frac{1}{q}}=\Big\Vert\frac{d_{M_1+j}\lambda_2}{p_1p_2\cdots p_{M_1+j}}\Big\Vert_{\frac{1}{q}}\geqslant\frac{1}{q}-\frac{1}{q^{3G}}.
	\end{equation*}
	Since $G\geqslant1$, we see that Claim I is therefore established for Case I-I.

	{\bf Case I-II}. $\frac{d_{M_1+j}\lambda_1}{p_1p_2\cdots p_{M_1+j}}\not\in\mathbb{Z}$.
	
	Since $\lambda_1=\sum_{i\in\mathfrak{B}}c_i$ and $\frac{d_{M_1+j}\lambda_1}{p_1p_2\cdots p_{M_1+j}}\not\in\mathbb{Z}$, there exists $s_2\in\mathfrak{B}$ such that
	\begin{equation*}
		k_{q,s_2}=\min\{k_{q,i}:\,c_i\ne0,\,i\in\mathfrak{B}\}.
	\end{equation*}
	By Lemma \eqref{le5.15} and the assumption $\sigma_i\geqslant M_2~(i\in\mathfrak{B})$, we have $k_{q,s_2}<k_{q,M_1+j}$ and
	\begin{equation}\label{eq6.15}
		\Big\Vert\frac{d_{M_1+j}\lambda_1}{p_1p_2\cdots p_{M_1+j}}\Big\Vert_{\frac{1}{q}}=\Big\Vert\mathop{\sum}\limits_{\substack{i\in\mathfrak{B}\\ c_i\ne0}}\frac{d_{M_1+j}c_i}{p_1\cdots p_{M_1+j}}\Big\Vert_{\frac{1}{q}}\geqslant q^{k_{q,s_2}-k_{q,M_1+j}-1}.
	\end{equation}	
	By \eqref{eq6.9} and \eqref{eq6.15}, we get
	\begin{equation*}
		\Big\Vert\frac{d_{M_1+j}\lambda_1}{p_1p_2\cdots p_{M_1+j}}\Big\Vert_{\frac{1}{q}}\geqslant\frac{1}{q^{3G-2}}.
	\end{equation*}
	Combining the above with \eqref{eq6.14}, we have $\big|\frac{d_{M_1+j}\lambda_2}{p_1p_2\cdots p_{M_1+j}}\big|<\big\Vert\frac{d_{M_1+j}\lambda_1}{p_1p_2\cdots p_{M_1+j}}\big\Vert_{\frac{1}{q}}$. Hence
	\begin{equation*}
		\Big\Vert\frac{d_{M_1+j}\lambda}{p_1p_2\cdots p_{M_1+j}}\Big\Vert_{\frac{1}{q}}\geqslant\Big\Vert\frac{d_{M_1+j}\lambda_1}{p_1p_2\cdots p_{M_1+j}}\Big\Vert_{\frac{1}{q}}-\Big|\frac{d_{M_1+j}\lambda_2}{p_1p_2\cdots p_{M_1+j}}\Big|\geqslant\frac{q^2-1}{q^{3G}}.
	\end{equation*}
	This proves the Claim I.

	By the definition of $\varepsilon(\cdot)$ and \eqref{eq6.13}, we have $\big|m_{d_{M_1+j}^{-1}D_{M_1+j}}\big(\frac{d_{M_1+j}\lambda}{p_1p_2\cdots p_{M_1+j}}\big)\big|\geqslant\varepsilon\big(\frac{1}{q^{3G}}\big)$ for all $1\leqslant j\leqslant M_2-M_1$ with $M_1+j\in X_q$. Hence, by \eqref{eq6.8}, we have
	\begin{equation}\label{eq6.16}
		\mathop{\prod}\limits_{\substack{1\leqslant j\leqslant M_2-M_1\\ M_1+j\in X_q}}\Big|m_{D_{M_1+j}}\Big(\frac{\lambda}{p_1p_2\cdots p_{M_1+j}}\Big)\Big|>\Big[\varepsilon\Big(\frac{1}{q^{3G}}\Big)\Big]^{2G}.
	\end{equation}

	{\bf Claim II}. For any $1\leqslant j\leqslant M_2-M_1$ with $M_1+j\in X_{\widetilde{q}}$, if $\frac{d_{M_1+j}\lambda_1}{p_1\cdots p_{M_1+j}}\not\in\mathbb{Z}$, then
	\begin{equation}\label{eq6.17}
		\Big\Vert\frac{d_{M_1+j}\lambda}{p_1\cdots p_{M_1+j}}\Big\Vert_{\frac{1}{{\widetilde{q}}}}\geqslant\frac{1}{{\widetilde{q}}\cdot q^{3G}}.
	\end{equation}
	
	{\bf Proof of Claim II}. Fix $1\leqslant j\leqslant M_2-M_1$ such that $M_1+j\in X_{\widetilde{q}}$ and $\frac{d_{M_1+j}\lambda_1}{p_1\cdots p_{M_1+j}}\not\in\mathbb{Z}$. Since $\lambda_1=\sum_{i\in\mathfrak{B}}c_i$ and $\frac{d_{M_1+j}\lambda_1}{p_1\cdots p_{M_1+j}}\not\in\mathbb{Z}$, there exists $s_2\in\mathfrak{B}$ such that
	\begin{equation*}
		k_{q,s_2}=\min\{k_{q,i}:\,c_i\ne0,\,i\in\mathfrak{B}\}.
	\end{equation*}	
	By \eqref{eq6.88-1}, we get
	\begin{equation}\label{eq6.18}
		\mathop{\sum}\limits_{\substack{i\in\mathfrak{B}\\ c_i\ne0}}\frac{d_{M_1+j}c_i}{p_1\cdots p_{M_1+j}}\in q^{k_{q,s_2}-\nu_{q}(L_{M_1+j})}(\mathbb{Z}\setminus q\mathbb{Z}).
	\end{equation}
	Since $\frac{d_{M_1+j}\lambda_1}{p_1\cdots p_{M_1+j}}\not\in\mathbb{Z}$,, we see that $k_{q,s_2}<\nu_{q}(L_{M_1+j})$ and
	\begin{equation*}
		\begin{split}
			\Big\Vert\frac{d_{M_1+j}\lambda_1}{p_1\cdots p_{M_1+j}}\Big\Vert_{\frac{1}{\widetilde{q}}}=\Big\Vert\mathop{\sum}\limits_{\substack{i\in\mathfrak{B}\\ c_i\ne0}}\frac{d_{M_1+j}c_i}{p_1\cdots p_{M_1+j}}\Big\Vert_{\frac{1}{\widetilde{q}}}\geqslant\frac{1}{\widetilde{q}}\cdot q^{k_{q,s_2}-\nu_{q}(L_{M_1+j})}.
		\end{split}
	\end{equation*}
	It follows from \eqref{eq6.10} that
	\begin{equation*}
		\Big\Vert\frac{d_{M_1+j}\lambda_1}{p_1\cdots p_{M_1+j}}\Big\Vert_{\frac{1}{\widetilde{q}}}\geqslant\frac{1}{\widetilde{q}\cdot q^{3G-2}}.
	\end{equation*}
	Note that $\{q,\,\widetilde{q}\}=\{2,\,3\}$, from \eqref{eq6.14} we have $\big|\frac{d_{M_1+j}\lambda_2}{p_1\cdots p_{M_1+j}}\big|<\big\Vert\frac{d_{M_1+j}\lambda_1}{p_1\cdots p_{M_1+j}}\big\Vert_{\frac{1}{\widetilde{q}}}$. Hence,
	\begin{equation*}
		\Big\Vert\frac{d_{M_1+j}\lambda}{p_1\cdots p_{M_1+j}}\Big\Vert_{\frac{1}{\widetilde{q}}}\geqslant\Big\Vert\frac{d_{M_1+j}\lambda_1}{p_1\cdots p_{M_1+j}}\Big\Vert_{\frac{1}{\widetilde{q}}}-\Big|\frac{d_{M_1+j}\lambda_2}{p_1\cdots p_{M_1+j}}\Big|\geqslant\frac{1}{\widetilde{q}\cdot q^{3G}}.
	\end{equation*}
	This proves the Claim II.

	For all $1\leqslant j\leqslant M_2-M_1$ with $M_1+j\in X_{\widetilde{q}}$, \eqref{eq6.18} shows $k_{q,s_2}<\nu_{q}(L_{M_1+j})$ if and only if $\frac{d_{M_1+j}\lambda_1}{p_1\cdots p_{M_1+j}}\not\in\mathbb{Z}$. Hence, the definition of $\varepsilon(\cdot)$ and \eqref{eq6.17} shows
	\begin{equation}\label{eq6.19}
		\big|m_{d_{M_1+j}^{-1}D_{M_1+j}}\big(\frac{d_{M_1+j}\lambda}{p_1p_2\cdots p_{M_1+j}}\big)\big|\geqslant\varepsilon\big(\frac{1}{\widetilde{q}\cdot q^{3G}}\big)
	\end{equation}
	for all integers $j$ satisfying $1\leqslant j\leqslant M_2-M_1$, $M_1+j\in X_{\widetilde{q}}$ and $k_{q,s_2}<\nu_{q}(L_{M_1+j})$. Moreover, the fact $k_{q,s_2}\geqslant k_{q,s_1}$ shows that $\left\{j:\,1\leqslant j\leqslant M_2-M_1,\,M_1+j\in X_{\widetilde{q}},\,k_{q,s_2}<\nu_{q}(L_{M_1+j})\right\}$ contained in $\left\{j:\,1\leqslant j\leqslant M_2-M_1,\,M_1+j\in X_{\widetilde{q}},\,k_{q,s_1}<\nu_{q}(L_{M_1+j})\right\}$. By \eqref{eq6.19}, we have
	\begin{equation}\label{eq6.20-1}
		\begin{split}
			&\mathop{\prod}\limits_{\substack{1\leqslant j\leqslant M_2-M_1\\ M_1+j\in X_{\widetilde{q}}\\ d_{M_1+j}\lambda_1\not\in p_1\cdots p_{M_1+j}\mathbb{Z}}}\Big|m_{D_{M_1+j}}\Big(\frac{\lambda}{p_1p_2\cdots p_{M_1+j}}\Big)\Big|\\
			\geqslant&\Big[\varepsilon\Big(\frac{1}{{\widetilde{q}}\cdot q^{3G}}\Big)\Big]^{\#\left\{j:\,1\leqslant j\leqslant M_2-M_1,\,M_1+j\in X_{\widetilde{q}},\,k_{q,s_2}<\nu_{q}(L_{M_1+j})\right\}}\\
			\geqslant&\Big[\varepsilon\Big(\frac{1}{{\widetilde{q}}\cdot q^{3G}}\Big)\Big]^{\#\left\{j:\,1\leqslant j\leqslant M_2-M_1,\,M_1+j\in X_{\widetilde{q}},\,k_{q,s_1}<\nu_{q}(L_{M_1+j})\right\}},
		\end{split}
	\end{equation}
	where the left-hand side of the above inequality is defined to be $1$  when $\{1\leqslant j\leqslant M_2-M_1:\,M_1+j\in X_{\widetilde{q}},\,\frac{d_{M_1+j}\lambda_1}{p_1\cdots p_{M_1+j}}\not\in\mathbb{Z}\}=\emptyset$.

	It remains to discuss the case $\frac{d_{M_1+j}\lambda_1}{p_1\cdots p_{M_1+j}}\in\mathbb{Z}$ with $1\leqslant j\leqslant M_2-M_1$ and $M_1+j\in X_{\widetilde{q}}$. Note that the function $m_{d_j^{-1}D_j}(x)(j\geqslant1)$ has period $1$. Hence, if $\frac{d_{M_1+j}\lambda_1}{p_1\cdots p_{M_1+j}}\in\mathbb{Z}$, then
	\begin{equation}\label{eq6.20}
		\begin{split}
			m_{d_{M_1+j}^{-1}D_{M_1+j}}\Big(\frac{d_{M_1+j}\lambda}{p_1p_2\cdots p_{M_1+j}}\Big)=&m_{d_{M_1+j}^{-1}D_{M_1+j}}\Big(\frac{d_{M_1+j}\lambda_1}{p_1p_2\cdots p_{M_1+j}}+\frac{d_{M_1+j}\lambda_2}{p_1p_2\cdots p_{M_1+j}}\Big)\\
			=&m_{d_{M_1+j}^{-1}D_{M_1+j}}\Big(\frac{d_{M_1+j}\lambda_2}{p_1p_2\cdots p_{M_1+j}}\Big).
		\end{split}
	\end{equation}
	Since $\lambda_2\in\widetilde{\Lambda}\big[M_1;~\frac{1}{D\cdot q^{3G}}\big]$, we have $\Big|\frac{d_{M_1+j}\lambda_2}{p_1\cdots p_{M_1+1}}\Big|\leqslant\frac{d_{M_1+j}}{D\cdot q^{3G}}\leqslant\frac{1}{q^{3}}\leqslant\frac{1}{8}$ for all $j\geqslant1$. It follows from \eqref{eq2.11} that
	\begin{equation*}
		\Big|\frac{d_{M_1+j}\lambda_2}{p_1p_2\cdots p_{M_1+j}}\Big|=\frac{1}{p_{M_1+2}\cdots p_{M_1+j}}\cdot\Big|\frac{d_{M_1+j}\lambda_2}{p_1\cdots p_{M_1+1}}\Big|\leqslant\frac{1}{2^{j-1}}\times\frac{1}{8}=\frac{1}{2^{j+2}},\qquad\forall \, j\geqslant2.
	\end{equation*}
	Hence, we get
	\begin{equation}\label{eq6.22}
		\sum_{j=1}^{\infty}\Big|\frac{d_{M_1+j}\lambda_2}{p_1p_2\cdots p_{M_1+j}}\Big|\leqslant1.
	\end{equation}
	By Lemma \ref{le2.9}, \eqref{eq6.20} and \eqref{eq6.22}, we have
	\begin{equation}\label{eq6.23}
		\mathop{\prod}\limits_{\substack{1\leqslant j\leqslant M_2-M_1\\ M_1+j\in X_{\widetilde{q}}\\ d_{M_1+j}\lambda_1\in p_1\cdots p_{M_1+j}\mathbb{Z}}}\Big|m_{d_{M_1+j}^{-1}D_{M_1+j}}\Big(\frac{d_{M_1+j}\lambda}{p_1p_2\cdots p_{M_1+j}}\Big)\Big|\geqslant\varepsilon_1.
	\end{equation}
	A combination of inequalities \eqref{eq6.16}, \eqref{eq6.19} and \eqref{eq6.23} shows that the desired result of Proposition \ref{prop6.3} is valid.
\end{pf}

\begin{prop}\label{prop6.6}
	Assume that the pair $(\mathbb{P},\,\mathbb{D})$ satisfies \eqref{eq1.4}, \eqref{eq1.5}, \eqref{eq2.12-3} and {\bfseries (C6)}. Let $n$ and $j$ be two integers satisfying $1\leqslant n<j$. If $j\in X_{\widetilde{q}_n}$ and $\#\{i\in X_{\widetilde{q}_n}:\,j+1\leqslant i\leqslant b_{q_n,n}\}\geqslant2G$, then $k_{q_n,n}\geqslant\nu_{q_{n}}(L_{j})$.
\end{prop}

\begin{pf}
	By Lemma \ref{le5.9} \rm(\romannumeral7), we have
	\begin{equation}\label{eq6.51}
		\nu_{\widetilde{q}_n}(p_{j+1}p_{j+2}\cdots p_{b_{q_n,n}})\geqslant G.
	\end{equation}
	
	Since set $\{i\in X_{\widetilde{q}_n}:\,j\leqslant i\leqslant b_{q_n,n}\}$ is nonempty, we have $j<b_{q_n,n}$. Hence, the assumption $n<j$ shows that $n<b_{q_n,n}$.
	
	If $b_{q_n,n}=l_{q_n,n}$. According to the definition of $l_{q_n,n}$, we have $l_{q_n,n}\in X_{q_n}$ and $k_{q_n,n}>k_{q_n,l_{q_n,n}}$. Since $\nu_{\widetilde{q}_n}(d_{b_{q_n,n}})\leqslant G-1$. By \eqref{eq6.51}, we have
	\begin{equation*}
		\begin{split}
			k_{\widetilde{q}_n,j}<&\nu_{\widetilde{q}_n}(p_1p_2\cdots p_{j})\\
			=&\nu_{\widetilde{q}_n}(p_1p_2\cdots p_{b_{q_n,n}})-\nu_{\widetilde{q}_n}(d_{b_{q_n,n}})+\nu_{\widetilde{q}_n}(d_{b_{q_n,n}})-\nu_{\widetilde{q}_n}(p_{j+1}p_{j+2}\cdots p_{b_{q_n,n}})\\
			\leqslant&\nu_{\widetilde{q}_n}(L_{b_{q_n,n}}).
		\end{split}
	\end{equation*}
	Hence, by \eqref{eq1.5}, we have $\nu_{q_n}(L_{j})\leqslant k_{q_n,b_{q_n,n}}$. Note that $k_{q_n,l_{q_n,n}}<k_{q_n,n}$. Therefore, if $b_{q_n,n}=l_{q_n,n}$, we obtain $\nu_{q_n}(L_{j})\leqslant k_{q_n,n}$.
	
	Else $b_{q_n,n}=\beta_{q_n,n}$. Since $\beta_{q_n,m}\in X_{\widetilde{q}_n}$, we get $b_{q_n,m}\in X_{\widetilde{q}_n}$. It follows from \eqref{eq6.51} that
	\begin{equation}\label{eq6.52}
		\begin{split}
			k_{\widetilde{q}_n,j}<&\nu_{\widetilde{q}_n}(p_1p_2\cdots p_{j})\\
			=&\nu_{\widetilde{q}_n}(p_1p_2\cdots p_{b_{q_n,n}})-\nu_{\widetilde{q}_n}(\widetilde{q}_nd_{b_{q_n,n}})+\nu_{\widetilde{q}_n}(\widetilde{q}_nd_{b_{q_n,n}})-\nu_{\widetilde{q}_n}(p_{j+1}p_{j+2}\cdots p_{b_{q_n,n}})\\
			\leqslant&k_{\widetilde{q}_n,b_{q_n,n}}.
		\end{split}
	\end{equation}
	Moreover, according to the definition of $\beta_{q_{n},n}$, there exists $m\in X_{q_n}$ such that $k_{q_n,m}\leqslant k_{q_n,n}$, $Alpha_{q_n,m}=\beta_{q_n,n}$ and $k_{\widetilde{q}_n,\alpha_{q_n,m}}<\nu_{\widetilde{q}_n}(L_{m})$. Since $\alpha_{q_n,m}=\beta_{q_n,n}=b_{q_n,n}$ and $k_{\widetilde{q}_n,\alpha_{q_n,m}}<\nu_{\widetilde{q}_n}(L_{m})$. By \eqref{eq6.52}, we have $k_{\widetilde{q}_n,j}<\nu_{\widetilde{q}_n}(L_{m})$. Thus, by \eqref{eq1.5}, we have $\nu_{q_n}(L_{j})\leqslant k_{q_n,m}$. Note that $k_{q_n,m}\leqslant k_{q_n,n}$, we have $\nu_{q_n}(L_{j})\leqslant k_{q_n,n}$. This finishes the proof.		
\end{pf}

The result of Proposition \ref{prop6.3} concerns the set $\big\{j\in\mathbb{Z}:\,1\leqslant j\leqslant M_2-M_1,\,k_{q,s_1}<\nu_{q}(L_{M_1+j})\big\}$. For any integers $n_1$, $n_2$ and $M$ with $1\leqslant n_1<n_2\leqslant M\leqslant H(n_2)$, we aim to establish a strictly positive uniform lower bound for $\prod_{j=1}^{H(n_2)-n_2}\big|m_{D_{M+j}}\big(\frac{\lambda}{p_1p_2\cdots p_{M+j}}\big)\big|$ that is independent of $n_1$, $n_2$ and $M$. To this end, we partition the integers from $M$ to $H(n_2)$ into distinct segments according to the elements of $\mathcal{B}[n_1;~n_2]$. This allows us to control the behavior of the function on each segment.

\begin{prop}\label{prop6.4-1}
	Assume that the pair $(\mathbb{P},\,\mathbb{D})$ satisfies \eqref{eq1.4}, \eqref{eq1.5}, \eqref{eq2.12-3} and {\bfseries (C6)}. Suppose $n_1$, $n_2$ and $M$ are integers satisfying $1\leqslant H(n_1)<n_2\leqslant M<H(n_2)$. Let $i'$ be the integer given in Proposition \ref{prop6.3}. Since $\mathcal{B}[n_1;~n_2]\cap X_{q_{i'}}\neq\emptyset$, we can order the elements of
	$\mathcal{B}[n_1;~n_2]\cap X_{q_{i'}}$ as $\mathrm{a}_1$, $\mathrm{a}_2$, $\cdots$, $\mathrm{a}_u$ so that $k_{q_{i'},\mathrm{a}_1}<k_{q_{i'},\mathrm{a}_2}<\cdots<k_{q_{i'},\mathrm{a}_u}$, where $u:=\#(\mathcal{B}[n_1;~n_2]\cap X_{q_{i'}})\leqslant2G$. Then there exists an integer sequence $\{m_t\}_{t=1}^{u}$ with $M<m_1\leqslant m_2\leqslant\cdots\leqslant m_u\leqslant H$ such that the following statements remain valid.
	\begin{enumerate}	
		\item[\rm(\romannumeral1)] $m_t\geqslant b_{q_{i'},\mathrm{a}_t}$ for all $1\leqslant t\leqslant u$.
		\item[\rm(\romannumeral2)] For all $1\leqslant t\leqslant u-1$,  if $m_t<H$, then $m_t-m_{t-1}\geqslant 2c_3$ and $m_{t+1}-m_t\geqslant 2c_3$, where $m_0:=M$.
		\item[\rm(\romannumeral3)] $\#\{j\in X_{\widetilde{q}_{i'}}:\,M<j\leqslant m_t,\,k_{q_{i'},\mathrm{a}_t}<\nu_{q_{i'}}(L_{j})\}\leqslant8Gc_3$.
	\end{enumerate}	
\end{prop}

\begin{pf}
	By Proposition \ref{prop6.3}, we get $\mathcal{B}[n_1;~n_2]\cap X_{q_{i'}}\neq\emptyset$. Furthermore, it is easy to check that \eqref{eq6.8} and \eqref{eq6.9-2} still hold in Proposition \ref{prop6.4-1}. Let
	\begin{equation*}
		q:=q_{i'},~~\widetilde{q}:=\widetilde{q}_{i'}~~\mbox{and}~~H:=H(n_2).
	\end{equation*}
	
	Let $i'':=\min(\mathcal{B}[n_1;~n_2]\cap X_{q_{i'}})$. By \eqref{eq6.9-2}, we have $\nu_{q}(p_{i''+1}p_{i''+2}\cdots p_{n_2})\leqslant G-1$. It follows from Lemma \ref{le5.9} \rm(\romannumeral7) that $\#\{i\in X_q:\, i''+1\leqslant i\leqslant n_2\}\leqslant 2G-1$. Hence $\#\{i\in X_q:\, i''\leqslant i\leqslant n_2\}\leqslant 2G$. Moreover, the definition of $i''$ shows  $(\mathcal{B}[n_1;~n_2]\cap X_{q_{i'}})\subset\{i\in X_q:\, i''\leqslant i\leqslant n_2\}$. Thus, we obtain
	\begin{equation}\label{eq6.55}
		u\leqslant2G.
	\end{equation}

	According to the definition of the set $\mathcal{B}[n_1;~n_2]$, we get $n_1<\mathrm{a}_1,\,\mathrm{a}_2,\,\cdots,\,\mathrm{a}_u\leqslant n_2$. Moreover, by Lemma \ref{le5.9} \rm(\romannumeral5) and the definition of the integer sequence $\{\mathrm{a}_t\}_{t=1}^{u}$, we have $n_2<b_{q,\mathrm{a}_1}\leqslant b_{q,\mathrm{a}_2}\leqslant \cdots\leqslant b_{q,\mathrm{a}_u}=H$.

	Let $m_0:=M$, and define the integer sequence $\{m_t\}_{t=1}^{u}$ recursively by
	\begin{equation}\label{eq6.25-1}
		m_{t}:=\left\{\begin{array}{ll}
			H,&\mbox{if}~~H-2c_3<m_{t-1}+2c_3,\\
			m_{t-1}+2c_3,&\mbox{if}~~b_{q,\mathrm{a}_t}\leqslant m_{t-1}+2c_3\leqslant H-2c_3,\\
			b_{q,\mathrm{a}_t},&\mbox{if}~~m_{t-1}+2c_3<b_{q,\mathrm{a}_t}\leqslant H-2c_3,\\
			H,&\mbox{if}~~m_{t-1}+2c_3\leqslant H-2c_3<b_{q,\mathrm{a}_t}.
		\end{array}\right.
	\end{equation}
	Since $H>M$ and $H>b_{q,\mathrm{a}_t}$ for all $1\leqslant t\leqslant u$, substituting $m_{t-1}$ into the above recursive formula yields $m_{t}\geqslant m_{t-1}$ and $m_i\geqslant b_{q_{i'},\mathrm{a}_t}$ for all $1\leqslant t\leqslant u$. Therefore, the condition $M<m_1\leqslant m_2\leqslant\cdots\leqslant m_u\leqslant H$ and statement \rm(\romannumeral1) hold.
	
	If $m_t<H$, then the second and third relations in \eqref{eq6.25-1} imply that $m_t-m_{t-1}\geqslant 2c_3$. On the other hand, \eqref{eq6.25-1} shows
	\begin{equation}\label{eq6.25-2}
		m_t\leqslant H-2c_3
	\end{equation}
	if $m_t<H$. If $H-2c_3<m_{t}+2c_3$, then $m_{t+1}=H$. Combining this with \eqref{eq6.25-2}, we get $m_{t+1}-m_t=H-m_t\geqslant 2c_3$ when $H-2c_3<m_{t}+2c_3$. Else $m_{t}+2c_3\leqslant H-2c_3$, then the last three formulas of the recurrence relation \eqref{eq6.25-1} imply $m_{t+1}-m_t\geqslant 2c_3$. Hence, \rm(\romannumeral2) holds. Next, it suffices to prove \rm(\romannumeral3).
	
	{\bf Claim I}. $\#\{j\in X_{\widetilde{q}}:\,n_2<j\leqslant b_{q,\mathrm{a}_t},\,k_{q,\mathrm{a}_t}<\nu_{q}(L_{j})\}<c_3$ for all $1\leqslant t\leqslant u$.
	
	{\bf Proof of Claim I}. Fix an arbitrary $1\leqslant t\leqslant u$. Assume that $b_{q,\mathrm{a}_t}-n_2\geqslant c_3$, otherwise the conclusion holds.
	
	For any $n\in X_{\widetilde{q}}$ with $n_2<n\leqslant b_{q,\mathrm{a}_t}-4G$, we have $\#\{i\in X_{\widetilde{q}}:\,n<i\leqslant b_{q,\mathrm{a}_t}\}>2G$ by \eqref{eq6.8}. From Proposition \ref{prop6.6}, we get that for any $n\in X_{\widetilde{q}}$ with $n_2<n\leqslant b_{q,\mathrm{a}_t}-4G$,
	\begin{equation*}
		k_{q,\mathrm{a}_t}\geqslant\nu_{q}(L_n).
	\end{equation*}
	Thus $\{n\in X_{\widetilde{q}}:\,n_2<n\leqslant b_{q,\mathrm{a}_t},\,k_{q,\mathrm{a}_t}<\nu_{q}(L_{n})\}\subset\{n\in X_{\widetilde{q}}:\,b_{q,\mathrm{a}_t}-4G<n\leqslant b_{q,\mathrm{a}_t}\}$. Hence $\{j\in X_{\widetilde{q}}:\,n_2<n\leqslant b_{q,\mathrm{a}_t},\,k_{q,\mathrm{a}_t}<\nu_{q}(L_{j})\}\leqslant4G<c_3$. Thus, Claim I is established.

	In order to prove \rm(\romannumeral3), we establish the following claim by mathematical induction.
	
	{\bf Claim II}. For all integers $t$ with $1\leqslant t\leqslant u$, we have
	\begin{equation}\label{eq6.64-1}
		\#\{j\in X_{\widetilde{q}}:\,M<j\leqslant m_t,\,k_{q,\mathrm{a}_t}<\nu_{q}(L_{j})\}\leqslant4tc_3.
	\end{equation}

	{\bf Proof of Claim II}. We first establish the base case, $t=1$.
	
	If the relations between $b_{q,\mathrm{a}_1}$, $m_0$ and $H$ satisfy the first two conditions in \eqref{eq6.25-1}, then $m_1-m_0<4c_3$. Since $m_0=M$, we have
	\begin{equation*}
		\#\{j\in X_{\widetilde{q}}:\,M<j\leqslant m_1,\,k_{q,\mathrm{a}_1}<\nu_{q}(L_{j})\}\leqslant m_1-M<4c_3.
	\end{equation*}
	If the relations between $b_{q,\mathrm{a}_1}$, $m_0$ and $H$ satisfy the last two conditions in \eqref{eq6.25-1}, it follows that $0\leqslant m_1-b_{q,\mathrm{a}_1}<2c_3$. Moreover, by the assumption $n_2\leqslant M$ and Claim I, we get $\#\{j\in X_{\widetilde{q}}:\,M<j\leqslant b_{q,\mathrm{a}_1},\,k_{q,\mathrm{a}_1}<\nu_{q}(L_{j})\}\leqslant c_3$. It follows that
	\begin{equation*}
		\begin{split}
			&\#\{j\in X_{\widetilde{q}}:\,M<j\leqslant m_1,\,k_{q,\mathrm{a}_1}<\nu_{q}(L_{j})\}\\
			=&\#\{j\in X_{\widetilde{q}}:\,M<j\leqslant b_{q,\mathrm{a}_1},\,k_{q,\mathrm{a}_1}<\nu_{q}(L_{j})\}+\#\{j\in X_{\widetilde{q}}:\,b_{q,\mathrm{a}_1}<j\leqslant m_1,\,k_{q,\mathrm{a}_1}<\nu_{q}(L_{j})\}\\
			<&c_3+2c_3.
		\end{split}
	\end{equation*}
	Thus, \eqref{eq6.64-1} holds for the case $t=1$.

	Assuming that \eqref{eq6.64-1} holds for all $1\leqslant t\leqslant k$, we now prove that it holds for $t=k+1$.

	If the relations between $b_{q,\mathrm{a}_{k+1}}$, $m_k$ and $H$ satisfy the first two conditions in \eqref{eq6.25-1}, then $m_{k+1}-m_k<4c_3$. Moreover, the assumption $k_{q,\mathrm{a}_k}<k_{q,\mathrm{a}_{k+1}}$ shows
	\begin{equation*}
		\{j\in X_{\widetilde{q}}:\,M<j\leqslant m_k,\,k_{q,\mathrm{a}_{k+1}}<\nu_{q}(L_{j})\}\subset\{j\in X_{\widetilde{q}}:\,M<j\leqslant m_k,\,k_{q,\mathrm{a}_k}<\nu_{q}(L_{j})\}.
	\end{equation*}
	Given the inductive assumption that \eqref{eq6.64-1} holds for $n=k$, and based on the above inclusion relation, we obtain $\#\{j\in X_{\widetilde{q}}:\,M<j\leqslant m_k,\,k_{q,\mathrm{a}_{k+1}}<\nu_{q}(L_{j})\}<4kc_3$. Hence
	\begin{equation*}
		\begin{split}
			&\#\{j\in X_{\widetilde{q}}:\,M<j\leqslant m_{k+1},\,k_{q,\mathrm{a}_{k+1}}<\nu_{q}(L_{j})\}\\
			=&\#\{j\in X_{\widetilde{q}}:\,M<j\leqslant m_k,\,k_{q,\mathrm{a}_{k+1}}<\nu_{q}(L_{j})\}+\#\{j\in X_{\widetilde{q}}:\,m_k<j\leqslant m_{k+1},\,k_{q,\mathrm{a}_{k+1}}<\nu_{q}(L_{j})\}\\
			<&4kc_3+4c_3.
		\end{split}
	\end{equation*}
	If the relations between $b_{q,\mathrm{a}_{k+1}}$, $m_k$ and $H$ satisfy the last two conditions in \eqref{eq6.25-1}, then $0\leqslant m_{k+1}-b_{q,\mathrm{a}_{k+1}}<4c_3$. Moreover, by the assumption $n_2\leqslant M$ and Claim I, we obtain $\#\{j\in X_{\widetilde{q}}:\,M<j\leqslant b_{q,\mathrm{a}_{k+1}},\,k_{q,\mathrm{a}_{k+1}}<\nu_{q}(L_{j})\}\leqslant c_3$. It follows that
	\begin{equation*}
		\begin{split}
			&\#\{j\in X_{\widetilde{q}}:\,M<j\leqslant m_{k+1},\,k_{q,\mathrm{a}_{k+1}}<\nu_{q}(L_{j})\}\\
			=&\#\{j\in X_{\widetilde{q}}:\,M<j\leqslant b_{q,\mathrm{a}_{k+1}},\,k_{q,\mathrm{a}_{k+1}}<\nu_{q}(L_{j})\}+\#\{j\in X_{\widetilde{q}}:\,b_{q,\mathrm{a}_{k+1}}<j\leqslant m_{k+1},\,k_{q,\mathrm{a}_{k+1}}<\nu_{q}(L_{j})\}\\
			<&c_3+2c_3.
		\end{split}
	\end{equation*}
	Therefore, we see that \eqref{eq6.64-1} holds for $t=k+1$. By mathematical induction, we obtain that \eqref{eq6.64-1} holds for all integers $t$ with $1\leqslant t\leqslant u$. This completes the proof of Claim II.

	It follows from \eqref{eq6.55} that \rm(\romannumeral3) holds. This completes the proof.
\end{pf}

\begin{prop}\label{prop6.7}
	Assume that the pair $(\mathbb{P},\,\mathbb{D})$ satisfies \eqref{eq1.4}, \eqref{eq1.5}, \eqref{eq2.12-3} and {\bfseries (C6)}. Suppose $n_1$, $n_2$ and $M$ are integers satisfying $1\leqslant H(n_1)<n_2\leqslant M<H(n_2)$. Let $i'$ be the integer given in Proposition \ref{prop6.3}. Then there exists a spectrum $\Lambda:=\Lambda(n_1,\,n_2,\,M)$ of $\bigast\limits_{i\in(\mathcal{B}[n_1,\,n_2]\cap X_{q_{i'}})}\delta_{p_1^{-1}p_2^{-1}\cdots p_i^{-1}D_i}$ which is of the form described in Theorem \ref{th5.11} (\romannumeral4), such that
	\begin{equation}\label{eq6.53}
		\prod_{j=1}^{H(n_2)-M}\Big|m_{D_{M+j}}\Big(\frac{\lambda}{p_1p_2\cdots p_{M+j}}\Big)\Big|\geqslant\Big\{\varepsilon_1\cdot\Big[\varepsilon\Big(\frac{1}{2^{2c_3}}\Big)\Big]^{10Gc_3}\Big\}^{2G}
	\end{equation}
	for all  $\lambda\in \Lambda+\widetilde{\Lambda}\big[M;~\frac{1}{D\cdot q^{3G}}\big]$, where $k_{q_n,n}$, $l_{q_n,n}$, $\varepsilon_1$, $U_n$, $\mathcal{S}[n_1,\,n_2]$, $H(n_2)$, $c_3$ and $\varepsilon(\cdot)$ are defined in \eqref{eq1.2}, \eqref{eq5.8}, Lemma \ref{le2.9}, \eqref{eq5.28}, \eqref{eq5.39}, \eqref{eq5.40}, \eqref{eq6.24} and \eqref{eq6.3}, respectively.	
\end{prop}

\begin{pf}
	Proposition \ref{prop6.3} shows that the set $\mathcal{B}[n_1,\,n_2]\cap X_{q_{i'}}$ is nonempty. Let
	\begin{equation*}
		q:=q_{i'},~~\widetilde{q}:=\widetilde{q}_{i'}~~\mbox{and}~~H:=H(n_2).
	\end{equation*}
	Moreover, it is easy to check that \eqref{eq6.9-1} still hold in Proposition \ref{prop6.7}. Similar to \eqref{eq6.9} and \eqref{eq6.10}, we have
	\begin{equation}\label{eq6.30}
		k_{q,M+j}-k_{q,i}\leqslant3(G-1)
	\end{equation}
	for all $i\in(\mathcal{B}[n_1,\,n_2]\cap X_{q})$ and $1\leqslant j\leqslant H-M_1$ with $M_1+j\in X_q$, and
	\begin{equation}\label{eq6.31}
		\nu_{q}(L_{M+j})-k_{q,i}\leqslant3G-2
	\end{equation}
	for all $i\in(\mathcal{B}[n_1,\,n_2]\cap X_{q})$ and $1\leqslant j\leqslant H-M_1$ with $M_1+j\in X_{\widetilde{q}}$.

	Let $\{\mathrm{a}_t\}_{t=1}^{u}$ and $\{m_t\}_{t=1}^{u}$ be the integer sequence given in Proposition \ref{prop6.4-1}. It follows that $\mathcal{B}[n_1,\,n_2]\cap X_{q}=\{\mathrm{a}_1,\,\mathrm{a}_2,\,\cdots,\,\mathrm{a}_u\}$. We now begin the construction of a prespectrum $\Lambda$ of $\bigast\limits_{t=1}^u\delta_{p_1^{-1}p_2^{-1}\cdots p_{\mathrm{a}_t}^{-1}D_{\mathrm{a}_t}}$. For any integer $t$ with $1\leqslant t\leqslant u$, we take
	\begin{equation}\label{eq6.72}
		C_{\mathrm{a}_t}:=\begin{cases}
			\{0,\,q^{k_{q,\mathrm{a}_{t}}}\theta_{q}(p_1p_2\cdots p_{m_t})\},&\mbox{if}~~q=2,\\
			\{0,\,q^{k_{q,\mathrm{a}_{t}}}\theta_{q}(p_1p_2\cdots p_{m_t}),\,-q^{k_{q,\mathrm{a}_{t}}}\theta_{q}(p_1p_2\cdots p_{m_t})\},&\mbox{if}~~q=3.
		\end{cases}
	\end{equation}
	It follows from Proposition \ref{prop6.4-1} \rm(\romannumeral1) that $C_{\mathrm{a}_t}\subset(\{0\}\cup U_{\mathrm{a}_t})$,  By Theorem \ref{th5.11} \rm(\romannumeral4), we see that $\sum_{t=1}^{u}C_{\mathrm{a}_t}$ is a spectrum of $\bigast\limits_{t=1}^u\delta_{p_1^{-1}p_2^{-1}\cdots p_{\mathrm{a}_t}^{-1}D_{\mathrm{a}_t}}$. Let $\Lambda:=\sum_{t=1}^{u}C_{\mathrm{a}_t}$.

	Choose $\lambda\in\Lambda+\widetilde{\Lambda}\big[M;~\frac{1}{D\cdot q^{3G}}\big]$, it can be written as $\lambda=\sum_{t=1}^{u}c_{\mathrm{a}_t}+\lambda_2$, where $c_{\mathrm{a}_t}\in C_{\mathrm{a}_t}(1\leqslant t\leqslant u)$ and $\lambda_2\in\widetilde{\Lambda}\big[M;~\frac{1}{D\cdot q^{3G}}\big]$.

	By Proposition \ref{prop6.4-1}, we have $m_t\geqslant m_1$ for all $1\leqslant t\leqslant u$. Therefore, the set $\sum_{t=1}^{u}C_{\mathrm{a}_t}$ defined on $\{\mathrm{a}_t\}_{t=1}^{u}$ by \eqref{eq6.72} satisfies the requirements of the set $\Lambda_1$ defined on $M_1:=M$, $M_2:=m_1$ and $\mathfrak{B}:=\{\mathrm{a}_t\}_{t=1}^{u}$ by \eqref{eq6.5}. Hence, by Proposition \ref{prop6.3} and Proposition \ref{prop6.4-1} \rm(\romannumeral3), we have
	\begin{equation}\label{eq6.74}
		\prod_{j=1}^{m_1-M}\Big|m_{D_{M+j}}\Big(\frac{\lambda}{p_1p_2\cdots p_{M+j}}\Big)\Big|\geqslant\varepsilon_1\cdot\Big[\varepsilon\Big(\frac{1}{\widetilde{q}\cdot q^{3G}}\Big)\Big]^{9Gc_3}.
	\end{equation}

	If $m_1=H$, then by the arbitrariness of $\lambda$, the desired formula \eqref{eq6.53} holds. Else $m_1<H$, then $u\geqslant2$. Moreover, Proposition \ref{prop6.4-1} \rm(\romannumeral2) shows $m_1-M\geqslant2c_3$ if $m_1<H$. It follow from \eqref{eq2.12-3} that if $m_1<H$, then
	\begin{equation}\label{eq6.75-1}
		p_{M+1}p_{M+2}\cdots p_{m_1+j}>2^{c_3},\qquad\forall\, j\geqslant1.
	\end{equation}

	{\bf Claim I}. If $m_1<H$, then
	\begin{equation}\label{eq6.75}
		\prod_{j=m_1-M+1}^{m_2-M}\Big|m_{D_{M+j}}\Big(\frac{\lambda}{p_1p_2\cdots p_{M+j}}\Big)\Big|\geqslant\varepsilon_1\cdot\Big[\varepsilon\Big(\frac{1}{2^{2c_3}}\Big)\Big]^{10Gc_3}.
	\end{equation}
	
	{\bf Proof of Claim I}. By the definition of $m_1$, we get $M+1\leqslant m_1$. Hence, we have $\big|\frac{\lambda_2}{p_1p_2\cdots p_{m_1}}\big|\leqslant\big|\frac{\lambda_2}{p_1p_2\cdots p_{M+1}}\big|\leqslant\frac{1}{D\cdot q^{3G}}$ by noting $\lambda_2\in\widetilde{\Lambda}\big[M;~\frac{1}{D\cdot q^{3G}}\big]$. Moreover, by \eqref{eq6.72} and the fact $k_{q,\mathrm{a}_1}<\nu_{q}(p_1p_2\cdots p_{m_1})$, we get $\big|\frac{c_{\mathrm{a}_1}}{p_1p_2\cdots p_{m_1}}\big|\leqslant q^{k_{q,\mathrm{a}_1}-\nu_{q}(p_1p_2\cdots p_{m_1})}\leqslant\frac{1}{q}$.
	It follows that
	\begin{equation}\label{eq6.76}
		\Big|\frac{c_{\mathrm{a}_1}+\lambda_2}{p_1p_2\cdots p_{m_1}}\Big|\leqslant\frac{1}{q}+\frac{1}{D\cdot q^{3G}}<1.
	\end{equation}

	%Since $m_1<H$, it follows from Proposition \ref{prop6.4-1} \rm(\romannumeral2) that $m_1\leqslant H-2c_3$. Moreover,

	{\bf Case I-I}. $p_{m_1+1}\geqslant2^{c_3}$.
	
	By the assumption $p_{m_1+1}\geqslant2^{c_3}$ and \eqref{eq6.76}, we have $\big|\frac{c_{\mathrm{a}_1}+\lambda_2}{p_1p_2\cdots p_{m_1+1}}\big|\leqslant\frac{1}{p_{m_1+1}}<\frac{1}{D\cdot3^{5G}}$. This means that
	\begin{equation}\label{eq6.77}
		c_{\mathrm{a}_1}+\lambda_2\in\widetilde{\Lambda}\Big[m_1;~\frac{1}{D\cdot q^{3G}}\Big]
	\end{equation}
	when $m_1<H$ and $p_{m_1+1}\geqslant2^{c_3}$.

	Since $k_{q,\mathrm{a}_t}\geqslant k_{q,\mathrm{a}_2}$ and $m_t\geqslant m_2$ for each integer $t$ with $2\leqslant t\leqslant u$. Thus,  when we take $M_1=m_1$, $M_2=m_2$ and $\mathfrak{B}:=\{\mathrm{a}_t\}_{t=2}^{u}$, the set $\sum_{t=2}^{u}C_{\mathrm{a}_t}$ defined by \eqref{eq6.72} satisfies the condition of $\Lambda_1$ given in \eqref{eq6.5}. Hence, by Proposition \ref{prop6.3} and \eqref{eq6.77}, we have
	\begin{equation}\label{eq6.78-2}
		\prod_{j=m_1-M+1}^{m_2-M}\Big|m_{D_{M+j}}\Big(\frac{\sum_{t=2}^{u}c_{\mathrm{a}_t}+(c_{\mathrm{a}_1}+\lambda_2)}{p_1p_2\cdots p_{M+j}}\Big)\Big|\geqslant\varepsilon_1\cdot\Big[\varepsilon\Big(\frac{1}{\widetilde{q}\cdot q^{3G}}\Big)\Big]^{9Gc_3}.
	\end{equation}
	Since $\{q,\,\widetilde{q}\}=\{2,\,3\}$ and $2^{c_3}>3^{5G}D$, we get $\widetilde{q}\cdot q^{3G}<2^{2c_3}$. Note that the function $\varepsilon(\cdot)$ is non-decreasing, we get $\prod\limits_{j=m_1-M+1}^{m_2-M}\big|m_{D_{M+j}}\big(\frac{\lambda}{p_1p_2\cdots p_{M+j}}\big)\big|\geqslant\varepsilon_1\cdot\big[\varepsilon\big(\frac{1}{2^{2c_3}}\big)\big]^{10Gc_3}$.
	
	{\bf Case I-II}. $p_{m_1+1}<2^{c_3}$.
	
	Our assumption $p_{m_1+1}<2^{c_3}$ implies that the set $\{j\geqslant1:\,p_{m_1+1}p_{m_1+2}\cdots p_{m_1+j}<2^{c_3}\}$ is nonempty. Let
	\begin{equation*}
		M_3:=\max\{m_1+j:\,p_{m_1+1}p_{m_1+2}\cdots p_{m_1+j}<2^{c_3}\}.
	\end{equation*}
	By \eqref{eq2.11}, we have $M_3<m_1+c_3$. Moreover, Proposition \ref{prop6.4-1} \rm(\romannumeral2) shows $m_1+2c_3\leqslant m_2$. Hence,
	\begin{equation}\label{eq6.78-1}
		M_3<m_2.
	\end{equation}
	On the other hand, the definition of $M_3$ shows
	\begin{equation}\label{eq6.28}
		\theta_{q}(p_{m_1+1}p_{m_1+2}\cdots p_{m_1+j})<2^{c_3},\qquad 1\leqslant j\leqslant M_3-m_1.
	\end{equation}

	If $\sum_{t=2}^u\frac{d_{m_1+j}c_{\mathrm{a}_t}}{p_1\cdots p_{m_1+j}}\neq0$, there exists $i_2\in\{\mathrm{a}_t:\,t\in\mathbb{Z},\,2\leqslant t\leqslant u,\,c_{\mathrm{a}_t}\ne0\}$ such that
	\begin{equation*}
		k_{q,i_2}=\min\{k_{q,\mathrm{a}_t}:\,t\in\mathbb{Z},\,2\leqslant t\leqslant u,\,c_{\mathrm{a}_t}\ne0\}.
	\end{equation*}
	It follows from the definition of the sequence $\{\mathrm{a}_t\}_{t=1}^{u}$ that $k_{q,\mathrm{a}_1}<k_{q,\mathrm{a}_t}$ for all $2\leqslant t\leqslant u$. Thus
	\begin{equation}\label{eq6.37}
		k_{q,\mathrm{a}_1}<k_{q,i_2}.
	\end{equation}
	By \eqref{eq6.88-1}, \eqref{eq6.72} and \eqref{eq6.78-1}, we get
	\begin{equation}\label{eq6.38}
		\sum_{t=2}^u\frac{d_{m_1+j}c_{\mathrm{a}_t}}{p_1\cdots p_{m_1+j}}=\sum_{2\leqslant t\leqslant u,\,c_{\mathrm{a}_t}\ne0}\frac{d_{m_1+j}c_{\mathrm{a}_t}}{p_1\cdots p_{m_1+j}}\in q^{k_{q,i_2}-k_{q,m_1+j}-1}\theta_{q}(d_{m_1+j})(\mathbb{Z}\setminus q\mathbb{Z})
	\end{equation}
	for any $1\leqslant j\leqslant M_3-m_1$ with $m_1+j\in X_q$, and
	\begin{equation}\label{eq6.39}
		\sum_{t=2}^u\frac{d_{m_1+j}c_{\mathrm{a}_t}}{p_1\cdots p_{m_1+j}}=\sum_{2\leqslant t\leqslant u,\,c_{\mathrm{a}_t}\ne0}\frac{d_{m_1+j}c_{\mathrm{a}_t}}{p_1\cdots p_{m_1+j}}\in q^{k_{q,i_2}-\nu_{q}(L_{m_1+j})}\theta_{q}(d_{m_1+j})(\mathbb{Z}\setminus q\mathbb{Z})
	\end{equation}
	for any $1\leqslant j\leqslant M_3-m_1$ with $m_1+j\in X_{\widetilde{q}}$.

	{\bf Claim I-II-I}. For any $1\leqslant j\leqslant M_3-m_1$ with $m_1+j\in X_q$, we have
	\begin{equation}\label{eq6.40}
		\Big\Vert\frac{d_{m_1+j}c_{\mathrm{a}_1}}{p_1\cdots p_{m_1+j}}+\sum_{t=2}^u\frac{d_{m_1+j}c_{\mathrm{a}_t}}{p_1\cdots p_{m_1+j}}\Big\Vert_{\frac{1}{q}}\geqslant\frac{1}{q^{3G-2}}\cdot\frac{1}{2^{c_3}}.
	\end{equation}
	
	{\bf Proof of Claim I-II-I}. Fix $1\leqslant j\leqslant M_3-m_1$ satisfying $m_1+j\in X_q$.
	
	If $c_{\mathrm{a}_1}\neq0$, then \eqref{eq6.72} shows
	\begin{equation}\label{eq6.35}
		\Big|\frac{d_{m_1+j}c_{\mathrm{a}_1}}{p_1\cdots p_{m_1+j}}\Big|=\frac{q^{k_{q,i_1}-k_{q,m_1+j}-1}\theta_{q}(d_{m_1+j})}{\theta_{q}(p_{m_1+1}\cdots p_{m_1+j})}.
	\end{equation}
	By \eqref{eq6.37}, \eqref{eq6.38} and \eqref{eq6.35}, we see that regardless of whether $\sum_{t=2}^u\frac{d_{m_1+j}c_{\mathrm{a}_t}}{p_1\cdots p_{m_1+j}}=0$ or $\sum_{t=2}^u\frac{d_{m_1+j}c_{\mathrm{a}_t}}{p_1\cdots p_{m_1+j}}\neq0$, we have
	\begin{equation*}
		\frac{d_{m_1+j}c_{\mathrm{a}_1}}{p_1\cdots p_{m_1+j}}+\sum_{t=2}^u\frac{d_{m_1+j}c_{\mathrm{a}_t}}{p_1\cdots p_{m_1+j}}\in\frac{q^{k_{q,i_1}-k_{q,m_1+j}-1}(\mathbb{Z}\setminus q\mathbb{Z})}{\theta_{q}(p_{m_1+1}\cdots p_{m_1+j})}.
	\end{equation*}
	This means that
	\begin{equation*}
		\Big\Vert\frac{d_{m_1+j}c_{\mathrm{a}_1}}{p_1\cdots p_{m_1+j}}+\sum_{t=2}^u\frac{d_{m_1+j}c_{\mathrm{a}_t}}{p_1\cdots p_{m_1+j}}\Big\Vert_{\frac{1}{q}}\geqslant
		\begin{cases}
			\frac{1}{q}\cdot\frac{1}{\theta_{q}(p_{m_1+1}\cdots p_{m_1+j})},&\mbox{if } k_{q,i_1}>k_{q,m_1+j},\\
			\frac{1}{q^{k_{q,m_1+j}+1-k_{q,i_1}}\cdot\theta_{q}(p_{m_1+1}\cdots p_{m_1+j})},&\mbox{if } k_{q,i_1}<k_{q,m_1+j}.
		\end{cases}
	\end{equation*}
	By \eqref{eq6.28}, \eqref{eq6.30} and the above inequality, we have
	\begin{equation}\label{eq6.41}
		\Big\Vert\frac{d_{m_1+j}c_{\mathrm{a}_1}}{p_1\cdots p_{m_1+j}}+\sum_{t=2}^u\frac{d_{m_1+j}c_{\mathrm{a}_t}}{p_1\cdots p_{m_1+j}}\Big\Vert_{\frac{1}{q}}\geqslant\frac{1}{q^{3G-2}}\cdot\frac{1}{2^{c_3}}
	\end{equation}
	when $c_{\mathrm{a}_1}\neq0$.
	
	If $c_{\mathrm{a}_1}=0$, then
	\begin{equation}\label{eq6.41-1}
		\big\Vert\frac{d_{m_1+j}c_{\mathrm{a}_1}}{p_1\cdots p_{m_1+j}}+\sum_{t=2}^u\frac{d_{m_1+j}c_{\mathrm{a}_t}}{p_1\cdots p_{m_1+j}}\big\Vert_{\frac{1}{q}}=\big\Vert\sum_{t=2}^u\frac{d_{m_1+j}c_{\mathrm{a}_t}}{p_1\cdots p_{m_1+j}}\big\Vert_{\frac{1}{q}}.
	\end{equation}
	If $\big\Vert\sum_{t=2}^u\frac{d_{m_1+j}c_{\mathrm{a}_t}}{p_1\cdots p_{m_1+j}}\big\Vert_{\frac{1}{q}}\not\in\mathbb{Z}$, then \eqref{eq6.30} and \eqref{eq6.38} shows
	\begin{equation*}
		\big\Vert\frac{d_{m_1+j}c_{\mathrm{a}_1}}{p_1\cdots p_{m_1+j}}+\sum_{t=2}^u\frac{d_{m_1+j}c_{\mathrm{a}_t}}{p_1\cdots p_{m_1+j}}\big\Vert_{\frac{1}{q}}\geqslant\frac{1}{q^{k_{q,M_1+j}+1-k_{q,i_2}}}\geqslant\frac{1}{q^{3G-2}}.
	\end{equation*}
	Else $\big\Vert\sum_{t=2}^u\frac{d_{m_1+j}c_{\mathrm{a}_t}}{p_1\cdots p_{m_1+j}}\big\Vert_{\frac{1}{q}}\in\mathbb{Z}$, then \eqref{eq6.41-1} shows $\big\Vert\frac{d_{m_1+j}c_{\mathrm{a}_1}}{p_1\cdots p_{m_1+j}}+\sum_{t=2}^u\frac{d_{m_1+j}c_{\mathrm{a}_t}}{p_1\cdots p_{m_1+j}}\big\Vert_{\frac{1}{q}}\geqslant\frac{1}{q}$. Thus, we get $\big\Vert\frac{d_{m_1+j}c_{\mathrm{a}_1}}{p_1\cdots p_{m_1+j}}+\sum_{t=2}^u\frac{d_{m_1+j}c_{\mathrm{a}_t}}{p_1\cdots p_{m_1+j}}\big\Vert_{\frac{1}{q}}\geqslant\frac{1}{q^{3G-2}}$ when $c_{\mathrm{a}_1}=0$. Combining this with \eqref{eq6.41}, the desired inequality \eqref{eq6.40} holds. This proves the Claim I.

	By \eqref{eq6.40} and the assumption $\lambda_2\in\widetilde{\Lambda}\big[M;~\frac{1}{D\cdot q^{3G}}\big]$, we have $\big|\frac{d_{m_1+j}\lambda_2}{p_1p_2\cdots p_{M+1}}\big|\leqslant\frac{d_{m_1+j}}{D\cdot q^{3G}}\leqslant\frac{1}{q^{3G}}$ for all $j\geqslant1$. It follows from \eqref{eq6.75-1} that
	\begin{equation}\label{eq6.43}
		\Big|\frac{d_{m_1+j}\lambda_2}{p_1p_2\cdots p_{m_1+j}}\Big|\leqslant\frac{1}{p_{M+1}p_{M+2}\cdots p_{m_1+j}}\cdot\Big|\frac{d_{m_1+j}\lambda_2}{p_1p_2\cdots p_{M+1}}\Big|<\frac{1}{2^{c_3}}\cdot\frac{1}{q^{3G}},\qquad\forall\,j\geqslant1.
	\end{equation}
	By \eqref{eq6.40} and \eqref{eq6.43}, we get $\big|\frac{d_{m_1+j}\lambda_2}{p_1\cdots p_{m_1+j}}\big|<\big\Vert\frac{d_{m_1+j}c_{\mathrm{a}_1}}{p_1\cdots p_{m_1+j}}+\sum_{t=2}^u\frac{d_{m_1+j}c_{\mathrm{a}_t}}{p_1\cdots p_{m_1+j}}\big\Vert_{\frac{1}{q}}$ for any $1\leqslant j\leqslant M_3-m_1$ with $m_1+j\in X_q$. Since $q\in\{2,\,3\}$, it follows from \eqref{eq6.40} and \eqref{eq6.43} that
	\begin{equation*}
		\Big\Vert\frac{d_{m_1+j}\lambda}{p_1\cdots p_{m_1+j}}\Big\Vert_{\frac{1}{q}}=\Big\Vert\frac{d_{m_1+j}c_{\mathrm{a}_1}}{p_1\cdots p_{m_1+j}}+\sum_{t=2}^u\frac{d_{m_1+j}c_{\mathrm{a}_t}}{p_1\cdots p_{m_1+j}}\Big\Vert_{\frac{1}{q}}-\Big|\frac{d_{m_1+j}\lambda_3}{p_1\cdots p_{m_1+j}}\Big|\geqslant\frac{1}{2^{c_3}\cdot q^{3G}}\geqslant\frac{1}{\widetilde{q}\cdot2^{c_3}\cdot q^{3G}}.
	\end{equation*}
	for any $1\leqslant j\leqslant M_3-m_1$ with $m_1+j\in X_q$. Hence, we have
	\begin{equation}\label{eq6.44}
		m_{d_{m_1+j}^{-1}D_{m_1+j}}\Big(\frac{d_{m_1+j}\lambda}{p_1\cdots p_{m_1+j}}\Big)\geqslant\varepsilon\Big(\frac{1}{\widetilde{q}\cdot2^{c_3}\cdot q^{3G}}\Big).
	\end{equation}
	for any $1\leqslant j\leqslant M_3-m_1$ with $m_1+j\in X_q$.

	{\bf Claim I-II-II}. For any $1\leqslant j\leqslant M_3-m_1$ with $m_1+j\in X_{\widetilde{q}}$, we have
	\begin{equation}\label{eq6.45}
		\Big\Vert\frac{d_{m_1+j}c_{\mathrm{a}_1}}{p_1\cdots p_{m_1+j}}+\sum_{t=2}^u\frac{d_{m_1+j}c_{\mathrm{a}_t}}{p_1\cdots p_{m_1+j}}\Big\Vert_{\frac{1}{\widetilde{q}}}\geqslant\frac{1}{\widetilde{q}}\cdot\frac{1}{q^{3G-2}}\cdot\frac{1}{2^{c_3}}.
	\end{equation}
	
	{\bf Proof of Claim I-II-II}. Fix $1\leqslant j\leqslant M_3-m_1$ satisfying $m_1+j\in X_{\widetilde{q}}$.
	
	If $c_{\mathrm{a}_1}\neq0$. By \eqref{eq2.11}, \eqref{eq2.10}, \eqref{eq6.72} and the fact $k_{q,i_1}<\nu_{q}(p_1p_2\cdots p_{m_1})$, we get
	\begin{equation*}
		\Big|\frac{d_{m_1+j}c_{\mathrm{a}_1}}{p_1\cdots p_{m_1+j}}\Big|=\left\{\begin{array}{ll}
			\big|\frac{d_{m_1+j}\cdot2^{k_{2,\mathrm{a}_1}}\cdot\theta_{2}(p_1p_2\cdots p_{m_1})}{p_1p_2\cdots p_{m_1+j}}\big|\leqslant\frac{d_{m_1+j}}{p_{m_1+j}}\cdot2^{k_{2,\mathrm{a}_1}-\nu_{2}(p_1p_2\cdots p_{m_1})}<\frac{1}{4},&\mbox{if}~~q=2,\\
			\big|\frac{d_{m_1+j}3^{k_{3,\mathrm{a}_1}}\theta_{3}(p_1p_2\cdots p_{m_1})}{p_1p_2\cdots p_{m_1+j}}\big|<3^{k_{3,\mathrm{a}_1}-\nu_{3}(p_1p_2\cdots p_{m_1})}\leqslant\frac{1}{3},&\mbox{if}~~q=3.
		\end{array}\right.
	\end{equation*}
	It follows that if $c_{\mathrm{a}_1}\neq0$ and $\frac{d_{m_1+j}\lambda_2}{p_1\cdots p_{m_1+j}}\in\mathbb{Z}$, then
	\begin{equation*}
		\Big\Vert\frac{d_{m_1+j}c_{\mathrm{a}_1}}{p_1\cdots p_{m_1+j}}+\sum_{t=2}^u\frac{d_{m_1+j}c_{\mathrm{a}_t}}{p_1\cdots p_{m_1+j}}\Big\Vert_{\frac{1}{\widetilde{q}}}=\Big\Vert\frac{d_{m_1+j}c_{\mathrm{a}_1}}{p_1\cdots p_{m_1+j}}\Big\Vert_{\frac{1}{\widetilde{q}}}>\begin{cases}
			\frac{1}{3}-\frac{1}{4},&\mbox{if}~~q=2,\\
			\frac{1}{2}-\frac{1}{3},&\mbox{if}~~q=3.
		\end{cases}
	\end{equation*}
	Since $c_3>5$, we see that the desired inequality \eqref{eq6.45} holds whenever $c_{\mathrm{a}_1}\neq0$ and $\frac{d_{m_1+j}\lambda_2}{p_1\cdots p_{m_1+j}}\in\mathbb{Z}$.

	If $c_{\mathrm{a}_1}\neq0$, then \eqref{eq6.72} shows
	\begin{equation}\label{eq6.36}
		\Big|\frac{d_{m_1+j}c_{\mathrm{a}_1}}{p_1\cdots p_{m_1+j}}\Big|=\frac{q^{k_{q,i_1}-\nu_{q}(L_{m_1+j})}\theta_{q}(d_{m_1+j})}{\theta_{q}(p_{m_1+1}\cdots p_{m_1+j})}.
	\end{equation}
	It follows from \eqref{eq6.36} and \eqref{eq6.39} that if $c_{\mathrm{a}_1}\neq0$ and $\sum_{t=2}^u\frac{d_{m_1+j}c_{\mathrm{a}_t}}{p_1p_2\cdots p_{m_1+j}}\not\in\mathbb{Z}$, then
	\begin{equation}\label{eq6.46}
		\frac{d_{m_1+j}c_{\mathrm{a}_1}}{p_1p_2\cdots p_{m_1+j}}+\sum_{t=2}^u\frac{d_{m_1+j}c_{\mathrm{a}_t}}{p_1p_2\cdots p_{m_1+j}}\in\frac{q^{k_{q,\mathrm{a}_1}-\nu_{q}(L_{m_1+j})}\theta_{q}(d_{m_1+j})(\mathbb{Z}\setminus q\mathbb{Z})}{\theta_{q}(p_{m_1+1}p_{m_1+2}\cdots p_{m_1+j})}.
	\end{equation}
	Applying \eqref{eq6.39} again, we get $k_{q,i_2}<\nu_{q}(L_{m_1+j})$ when $\sum_{t=2}^u\frac{d_{m_1+j}c_{\mathrm{a}_t}}{p_1p_2\cdots p_{m_1+j}}\not\in\mathbb{Z}$. Furthermore, by \eqref{eq6.37}, we have
	\begin{equation}\label{eq6.47}
		k_{q,\mathrm{a}_1}<\nu_{q}(L_{m_1+j})
	\end{equation}
	if $\sum_{t=2}^u\frac{d_{m_1+j}c_{\mathrm{a}_t}}{p_1p_2\cdots p_{m_1+j}}\not\in\mathbb{Z}$. This means that the denominator of the fraction on the right-hand side of \eqref{eq6.46}, when reduced to its lowest terms, must contain a factor of $q$. Hence
	\begin{equation*}
		\frac{d_{m_1+j}c_{\mathrm{a}_1}}{p_1p_2\cdots p_{m_1+j}}+\sum_{t=2}^u\frac{d_{m_1+j}c_{\mathrm{a}_t}}{p_1p_2\cdots p_{m_1+j}}\not\in\frac{\mathbb{Z}\setminus\widetilde{q}\mathbb{Z}}{\widetilde{q}}.
	\end{equation*}
	Thus $\big\Vert\frac{d_{m_1+j}c_{\mathrm{a}_1}}{p_1p_2\cdots p_{m_1+j}}+\sum_{t=2}^u\frac{d_{m_1+j}c_{\mathrm{a}_t}}{p_1p_2\cdots p_{m_1+j}}\big\Vert_{\frac{1}{\widetilde{q}}}>0$ if $c_{\mathrm{a}_1}\neq0$ and $\sum_{t=2}^u\frac{d_{m_1+j}c_{\mathrm{a}_t}}{p_1p_2\cdots p_{m_1+j}}\not\in\mathbb{Z}$. By \eqref{eq6.31}, \eqref{eq6.28}, \eqref{eq6.46} and \eqref{eq6.47}, we obtain that if $c_{\mathrm{a}_1}\neq0$ and $\sum_{t=2}^u\frac{d_{m_1+j}c_{\mathrm{a}_t}}{p_1p_2\cdots p_{m_1+j}}\not\in\mathbb{Z}$, then
	\begin{equation*}
		\Big\Vert\frac{d_{m_1+j}c_{\mathrm{a}_1}}{p_1p_2\cdots p_{m_1+j}}+\sum_{t=2}^u\frac{d_{m_1+j}c_{\mathrm{a}_t}}{p_1p_2\cdots p_{m_1+j}}\Big\Vert_{\frac{1}{\widetilde{q}}}\geqslant\frac{1}{\widetilde{q}}\cdot\frac{1}{q^{\nu_{q}(L_{m_1+j})-k_{q,i_1}}\theta_{q}(p_{m_1+1}\cdots p_{m_1+j})}\geqslant\frac{1}{\widetilde{q}}\cdot\frac{1}{q^{3G-2}}\cdot\frac{1}{2^{c_3}}
	\end{equation*}
	Thus, whenever $c_{\mathrm{a}_1}\neq0$ and $\sum_{t=2}^u\frac{d_{m_1+j}c_{\mathrm{a}_t}}{p_1p_2\cdots p_{m_1+j}}\not\in\mathbb{Z}$ hold, \eqref{eq6.45} holds.

	If $c_{\mathrm{a}_1}=0$ and $\sum_{t=2}^u\frac{d_{m_1+j}c_{\mathrm{a}_t}}{p_1p_2\cdots p_{m_1+j}}\in\mathbb{Z}$, then $\Big\Vert\frac{d_{m_1+j}c_{\mathrm{a}_1}}{p_1p_2\cdots p_{m_1+j}}+\sum_{t=2}^u\frac{d_{m_1+j}c_{\mathrm{a}_t}}{p_1p_2\cdots p_{m_1+j}}\Big\Vert_{\frac{1}{\widetilde{q}}}\geqslant\frac{1}{\widetilde{q}}$. In this case, \eqref{eq6.45} holds.

	If $c_{\mathrm{a}_1}=0$ and $\sum_{t=2}^u\frac{d_{m_1+j}c_{\mathrm{a}_t}}{p_1p_2\cdots p_{m_1+j}}\not\in\mathbb{Z}$, then \eqref{eq6.39} shows
	\begin{equation*}
		\Big\Vert\frac{d_{m_1+j}c_{\mathrm{a}_1}}{p_1p_2\cdots p_{m_1+j}}+\sum_{t=2}^u\frac{d_{m_1+j}c_{\mathrm{a}_t}}{p_1p_2\cdots p_{m_1+j}}\Big\Vert_{\frac{1}{\widetilde{q}}}=\Big\Vert\sum_{t=2}^u\frac{d_{m_1+j}c_{\mathrm{a}_t}}{p_1p_2\cdots p_{m_1+j}}\Big\Vert_{\frac{1}{\widetilde{q}}}\geqslant\frac{1}{\widetilde{q}}\cdot\frac{1}{q^{\nu_{q}(L_{m_1+j})-k_{q,i_2}}}.
	\end{equation*}
	Combining the above with \eqref{eq6.31}, we get $\big\Vert\frac{d_{m_1+j}\lambda_0}{p_1\cdots p_{m_1+j}}+\frac{d_{m_1+j}\lambda_2}{p_1\cdots p_{m_1+j}}\big\Vert_{\frac{1}{\widetilde{q}}}\geqslant\frac{1}{\widetilde{q}}\cdot\frac{1}{q^{3G-2}}$. This proves the Claim I-II-II.

	By \eqref{eq6.43} and \eqref{eq6.45}, we have $\big|\frac{d_{m_1+j}\lambda_2}{p_1\cdots p_{m_1+j}}\big|<\big\Vert\frac{d_{m_1+j}c_{\mathrm{a}_1}}{p_1p_2\cdots p_{m_1+j}}+\sum_{t=2}^u\frac{d_{m_1+j}c_{\mathrm{a}_t}}{p_1p_2\cdots p_{m_1+j}}\big\Vert_{\frac{1}{\widetilde{q}}}$ for any $1\leqslant j\leqslant M_3-m_1$ with $m_1+j\in X_{\widetilde{q}}$. Note that $\{q,\,\widetilde{q}\}=\{2,\,3\}$, \eqref{eq6.43} and \eqref{eq6.45} yields
	\begin{equation*}
		\Big\Vert\frac{d_{m_1+j}\lambda}{p_1\cdots p_{m_1+j}}\Big\Vert_{\frac{1}{\widetilde{q}}}=\Big\Vert\frac{d_{m_1+j}c_{\mathrm{a}_1}}{p_1p_2\cdots p_{m_1+j}}+\sum_{t=2}^u\frac{d_{m_1+j}c_{\mathrm{a}_t}}{p_1p_2\cdots p_{m_1+j}}\Big\Vert_{\frac{1}{\widetilde{q}}}-\Big|\frac{d_{m_1+j}\lambda_2}{p_1\cdots p_{m_1+j}}\Big|\geqslant\frac{1}{\widetilde{q}\cdot2^{c_3}\cdot q^{3G}}.
	\end{equation*}
	for any $1\leqslant j\leqslant M_3-m_1$ with $m_1+j\in X_{\widetilde{q}}$. Hence, we have
	\begin{equation}\label{eq6.49}
		m_{d_{m_1+j}^{-1}D_{m_1+j}}\Big(\frac{d_{m_1+j}\lambda}{p_1\cdots p_{m_1+j}}\Big)\geqslant\varepsilon\Big(\frac{1}{\widetilde{q}\cdot2^{c_3}\cdot q^{3G}}\Big).
	\end{equation}
	for any $1\leqslant j\leqslant M_3-m_1$ with $m_1+j\in X_{\widetilde{q}}$. By \eqref{eq6.78-1}, \eqref{eq6.44} and \eqref{eq6.49}, we have
	\begin{equation}\label{eq6.79}
		\prod_{j=m_1-M+1}^{M_3-M}\Big|m_{D_{m_1+j}}\Big(\frac{\lambda}{p_1p_2\cdots p_{M+j}}\Big)\Big|\geqslant\Big[\varepsilon\Big(\frac{1}{\widetilde{q}\cdot2^{c_3}\cdot q^{3G}}\Big)\Big]^{c_3}.
	\end{equation}

	On the other hand, the definition of $M_3$ shows $p_{m_1+1}p_{m_1+2}\cdots p_{m_1+j}\geqslant2^{c_3}$ for every $j>M_3-m_1$. Hence, we have $\big|\frac{c_{\mathrm{a}_1}+\lambda_2}{p_1p_2\cdots p_{M_3+1}}\big|=\frac{1}{p_{m_1+1}p_{m_1+2}\cdots p_{M_3+1}}\cdot	\Big|\frac{c_{\mathrm{a}_1}+\lambda_2}{p_1p_2\cdots p_{m_1+1}}\Big|<\frac{1}{2^{c_3}}$ by
	\eqref{eq6.76}. Since $2^{c_3}>3^{5G}D$, we get
	\begin{equation}\label{eq6.80}
		c_{\mathrm{a}_1}+\lambda_2\in\widetilde{\Lambda}\Big[M_3;~\frac{1}{D\cdot q^{3G}}\Big].
	\end{equation}
	Hence, by Proposition \ref{prop6.3}, Proposition \ref{prop6.4-1} \rm(\romannumeral3), \eqref{eq6.72} and \eqref{eq6.80}, we have
	\begin{equation}\label{eq6.81-1}
		\prod_{j=M_3-M+1}^{m_2-M}\Big|m_{D_{M+j}}\Big(\frac{\sum_{t=2}^{u}c_{\mathrm{a}_t}+(c_{\mathrm{a}_1}+\lambda_2)}{p_1p_2\cdots p_{M+j}}\Big)\Big|\geqslant\varepsilon_1\cdot\Big[\varepsilon\Big(\frac{1}{\widetilde{q}\cdot q^{3G}}\Big)\Big]^{9Gc_3}.
	\end{equation}
	Since $\{q,\,\widetilde{q}\}=\{2,\,3\}$ and $2^{c_3}>3^{5G}D$, we get $\widetilde{q}\cdot q^{3G}<\widetilde{q}\cdot2^{c_3}\cdot q^{3G}<2^{2c_3}$. Note that the function $\varepsilon(\cdot)$ is non-decreasing, from \eqref{eq6.79} and \eqref{eq6.81-1}, we have
	\begin{equation*}
		\prod_{j=m_1-M+1}^{m_2-M}\Big|m_{D_{M+j}}\Big(\frac{\lambda}{p_1p_2\cdots p_{M+j}}\Big)\Big|\geqslant\varepsilon_1\cdot\Big[\varepsilon\Big(\frac{1}{2^{2c_3}}\Big)\Big]^{10Gc_3}.
	\end{equation*}
	This finishes the proof of the Claim I.

	If $m_2=H$, then the arbitrariness of $\lambda$ shows that the desired formula \eqref{eq6.53} holds. Else $m_2<H$, then $u\geqslant3$. Moreover, Proposition \ref{prop6.4-1} \rm(\romannumeral2) shows $m_2-m_1\geqslant2c_3$ if $m_2<H$. It follow from \eqref{eq2.12-3} that if $m_2<H$, then
	\begin{equation*}
		p_{m_1+1}p_{m_1+2}\cdots p_{m_2+j}>2^{c_3},\qquad\forall\, j\geqslant1.
	\end{equation*}
	
	{\bf Claim II}. If $m_2<H$, then
	\begin{equation}\label{eq6.81-7}
		\prod_{j=m_2-M+1}^{m_3-M}\Big|m_{D_{M+j}}\Big(\frac{\lambda}{p_1p_2\cdots p_{M+j}}\Big)\Big|\geqslant\varepsilon_1\cdot\Big[\varepsilon\Big(\frac{1}{2^{2c_3}}\Big)\Big]^{10Gc_3}.
	\end{equation}
	
	{\bf Proof of Claim II}.
	By Proposition \ref{prop6.4-1}, we obtain that
	\begin{equation}\label{eq6.81-3}
		k_{q,\mathrm{a}_1}<k_{q,\mathrm{a}_2}<\cdots<k_{q,\mathrm{a}_u}.
	\end{equation}
	Moreover, applying Proposition \ref{prop6.4-1} again, we get
	\begin{equation}\label{eq6.81-4}
		b_{q,\mathrm{a}_2}\leqslant m_2\leqslant m_3-2c_3
	\end{equation}
	by noting $m_2<H$.
	
	If $p_{m_2+1}\geqslant2^{c_3}$, then similarly to \eqref{eq6.78-2}, it follows from Proposition \ref{prop6.3} that
	\begin{equation}\label{eq6.81-2}
		\prod_{j=m_2-M+1}^{m_3-M}\Big|m_{D_{M+j}}\Big(\frac{\sum_{t=3}^{u}c_{\mathrm{a}_t}+(\sum_{t=1}^{2}c_{\mathrm{a}_t}+\lambda_2)}{p_1p_2\cdots p_{M+j}}\Big)\Big|\geqslant\varepsilon_1\cdot\Big[\varepsilon\Big(\frac{1}{\widetilde{q}\cdot q^{3G}}\Big)\Big]^{9Gc_3}.
	\end{equation}
	Thus, \eqref{eq6.81-7} holds for case $p_{m_2+1}\geqslant2^{c_3}$.

	If $p_{m_2+1}<2^{c_3}$. Let
	\begin{equation*}
		M_4:=\max\{m_2+j:\,p_{m_2+1}p_{m_2+2}\cdots p_{m_2+j}<2^{c_3}\}.
	\end{equation*}
	By analogy with \eqref{eq6.40} and \eqref{eq6.45}, from \eqref{eq6.81-3} and \eqref{eq6.81-4}, we get
	\begin{equation}\label{eq6.81-5}
		\Big\Vert\sum_{t=1}^2\frac{d_{m_2+j}c_{\mathrm{a}_t}}{p_1\cdots p_{m_2+j}}+\sum_{t=3}^u\frac{d_{m_2+j}c_{\mathrm{a}_t}}{p_1\cdots p_{m_2+j}}\Big\Vert_{\frac{1}{q}}\geqslant\frac{1}{q^{3G-2}}\cdot\frac{1}{2^{c_3}}.
	\end{equation}
	for all $1\leqslant j\leqslant M_4-m_2$ with $m_2+j\in X_q$, and
	\begin{equation}\label{eq6.81-6}
		\Big\Vert\sum_{t=1}^2\frac{d_{m_2+j}c_{\mathrm{a}_t}}{p_1\cdots p_{m_2+j}}+\sum_{t=3}^u\frac{d_{m_2+j}c_{\mathrm{a}_t}}{p_1\cdots p_{m_2+j}}\Big\Vert_{\frac{1}{\widetilde{q}}}\geqslant\frac{1}{\widetilde{q}}\cdot\frac{1}{q^{3G-2}}\cdot\frac{1}{2^{c_3}}.
	\end{equation}
	for any $1\leqslant j\leqslant M_4-m_2$ with $m_2+j\in X_{\widetilde{q}}$. Hence, similar to \eqref{eq6.79}, from \eqref{eq6.81-5} and \eqref{eq6.81-6}, we get
	\begin{equation}\label{eq6.81-8}
		\prod_{j=m_2-M+1}^{M_4-M}\Big|m_{D_{m_2+j}}\Big(\frac{\lambda}{p_1p_2\cdots p_{M+j}}\Big)\Big|\geqslant\Big[\varepsilon\Big(\frac{1}{\widetilde{q}\cdot2^{c_3}\cdot q^{3G}}\Big)\Big]^{c_3}.
	\end{equation}
	Moreover, similar to \eqref{eq6.81-1}, according to the definition of $M_4$ and Proposition \ref{prop6.3}, we get
	\begin{equation*}
		\prod_{j=M_4-M+1}^{m_3-M}\Big|m_{D_{M+j}}\Big(\frac{\sum_{t=3}^{u}c_{\mathrm{a}_t}+(\sum_{t=1}^{2}c_{\mathrm{a}_t}+\lambda_2)}{p_1p_2\cdots p_{M+j}}\Big)\Big|\geqslant\varepsilon_1\cdot\Big[\varepsilon\Big(\frac{1}{\widetilde{q}\cdot q^{3G}}\Big)\Big]^{9Gc_3}.
	\end{equation*}
	Combining the above with \eqref{eq6.81-8}, we get
	\begin{equation}\label{eq6.81-9}
		\prod_{j=m_2-M+1}^{m_3-M}\Big|m_{D_{m_2+j}}\Big(\frac{\lambda}{p_1p_2\cdots p_{M+j}}\Big)\Big|\geqslant\Big[\varepsilon\Big(\frac{1}{\widetilde{q}\cdot2^{c_3}\cdot q^{3G}}\Big)\Big]^{10Gc_3}.
	\end{equation}
	Since $\{q,\,\widetilde{q}\}=\{2,\,3\}$ and $2^{c_3}>3^{5G}D$, we get $\widetilde{q}\cdot q^{3G}<\widetilde{q}\cdot2^{c_3}\cdot q^{3G}<2^{2c_3}$. Note that the function $\varepsilon(\cdot)$ is non-decreasing, from \eqref{eq6.81-2} and \eqref{eq6.81-9}, we see that the desired inequality \eqref{eq6.81-7} holds.
	
	If $m_3=H$, then the arbitrariness of $\lambda$ shows \eqref{eq6.53} holds. Else $m_3<H$, then $u\geqslant4$. Iterating the argument used in \eqref{eq6.75} and \eqref{eq6.81-7}, we deduce that
	\begin{equation*}
		\prod_{j=m_t-M+1}^{m_{t+1}-M}\Big|m_{D_{M+j}}\Big(\frac{\lambda}{p_1p_2\cdots p_{M+j}}\Big)\Big|\geqslant\varepsilon_1\cdot\Big[\varepsilon\Big(\frac{1}{2^{2c_3}}\Big)\Big]^{10Gc_3},\qquad\forall\,3\leqslant t\leqslant u'-1,
	\end{equation*}
	where $u'$ is the smallest index such that $m_t=H$, i.e., $u':=\min\{t\geqslant3:\,m_t=H\}$. Hence, it follows from \eqref{eq6.75} and \eqref{eq6.81-7} that
	\begin{equation*}
		\prod_{j=1}^{H}\Big|m_{D_{M+j}}\Big(\frac{\lambda}{p_1p_2\cdots p_{M+j}}\Big)\Big|=	\prod_{t=0}^{u'-1}\prod_{j=m_t-M+1}^{m_{t+1}-M}\Big|m_{D_{M+j}}\Big(\frac{\lambda}{p_1p_2\cdots p_{M+j}}\Big)\Big|\geqslant\Big\{\varepsilon_1\cdot\Big[\varepsilon\Big(\frac{1}{2^{2c_3}}\Big)\Big]^{10Gc_3}\Big\}^{u'},
	\end{equation*}
	where $m_0:=M$. It is clear that $u'\leqslant u$. Thus, we get $u'\leqslant 2G$ by Proposition \ref{prop6.4-1}. Therefore, we have
	\begin{equation*}
		\prod_{j=1}^{H}\Big|m_{D_{M+j}}\Big(\frac{\lambda}{p_1p_2\cdots p_{M+j}}\Big)\Big|\geqslant\Big\{\varepsilon_1\cdot\Big[\varepsilon\Big(\frac{1}{2^{2c_3}}\Big)\Big]^{10Gc_3}\Big\}^{2G}.
	\end{equation*}
	By the arbitrariness of $\lambda$, the desired formula \eqref{eq6.53} holds.
\end{pf}

\begin{prop}\label{prop6.9}
	Assume that the pair $(\mathbb{P},\,\mathbb{D})$ satisfies \eqref{eq1.4}, \eqref{eq1.5}, \eqref{eq2.12-3} and {\bfseries (C6)}. There exist constants $\varepsilon_3>0$ and $\omega_3>0$ such that for any integers $n_0\geqslant1$ and $n_1\geqslant N_4$, one can find an integer $n_2>H(n_0+n_1)$ and a spectrum $\Lambda(n_1,\,n_2)$ of $\bigast\limits_{i\in\mathcal{B}[n_1,\,n_2]}\delta_{p_1^{-1}p_2^{-1}\cdots p_i^{-1}D_i}$ which is of the form described in Theorem \ref{th5.11} (\romannumeral4), such that
	\begin{equation*}
		\prod_{j=1}^{H(n_2)-n_2}\Big|m_{D_{n_2+j}}\Big(\frac{y}{p_{n+1}p_{n+2}\cdots p_{n+j}}+\frac{\lambda}{p_1p_2\cdots p_{n_2+j}}\Big)\Big|\geqslant\varepsilon_3,\qquad\forall  \, y\in[-\omega_3,\,\omega_3],\quad\lambda\in\Lambda(n_1,\,n_2),
	\end{equation*}
	where $k_{q_n,n}$, $l_{q_n,n}$, $U_n$, $H(n_2)$ and $c_3$ are defined in \eqref{eq1.2}, \eqref{eq5.8}, \eqref{eq5.28}, \eqref{eq5.40} and \eqref{eq6.24}, respectively.
\end{prop}

\begin{pf}
	By {\bfseries (C6)}, we get $H>n_2$, so $\mathcal{B}[n_1,\,n_2]\neq\emptyset$.

	{\bf Claim}. One can find an integer $n_2>H(n_0+n_1)$ and a spectrum $\Lambda:=\Lambda(n_1,\,n_2)$ of $\bigast\limits_{i\in\mathcal{B}[n_1,\,n_2]}\delta_{p_1^{-1}p_2^{-1}\cdots p_i^{-1}D_i}$ which is of the form described in Theorem \ref{th5.11} (\romannumeral4), such that
	\begin{equation}\label{eq6.91-1}
		\prod_{j=1}^{H(n_2)-n_2}\Big|m_{D_{n_2+j}}\Big(\frac{\lambda}{p_1p_2\cdots p_{n_2+j}}\Big)\Big|\geqslant\varepsilon_1^{2G}\cdot\big[\varepsilon\big(\frac{1}{3^{2(2c_3+1)c_3}}\big)\big]^{21G^2c_3},\qquad\forall \,\lambda\in\Lambda.
	\end{equation}
	
	{\bf Proof of Claim}. We prove the assertion in two cases.
	
	{\bf Case I}. $\{n\geqslant H(n_0+n_1):\,\mbox{either}~~\nu_{2}(p_{n+1})\geqslant c_3~~\mbox{or}~~\nu_{3}(p_{n+1})\geqslant c_3\}\neq\emptyset$.
	
	Choose $n_2>H(n_0+n_1)$ and $\widetilde{q}\in\{2,\,3\}$ such that either $\nu_{\widetilde{q}}(p_{n_2+1})\geqslant c_3$. Let $\mathcal{S}:=\mathcal{S}[n_1,\,n_2]$, $H:=H(n_2)$ and
	\begin{equation*}
		q\in(\{2,\,3\}\setminus\{\widetilde{q}\}).
	\end{equation*}
	By Lemma \ref{le5.9} \rm(\romannumeral8) and the assumption $\nu_{\widetilde{q}}(p_{n_2+1})\geqslant c_3$, we have
	\begin{equation}\label{eq6.83}
		b_{\widetilde{q},i}\leqslant n_2,\qquad\forall \, i\in(\mathcal{S}\cap X_{\widetilde{q}}).
	\end{equation}
	By {\bfseries (C6)}, we get $H>n_2$, so $H>H(n_1)$. It follows from Proposition \ref{prop6.3} and \eqref{eq6.83} that there exists $i'\in [\mathcal{S}\cap X_{q}]$ such that
	\begin{equation}\label{eq6.84}
		b_{q,i'}=H.
	\end{equation}
	Moreover, we get
	\begin{equation*}
		\mathcal{B}[n_1,\,n_2]=\mathcal{B}[n_1,\,n_2]\cap X_{q}
	\end{equation*}
	by \eqref{eq6.83}.
	
	Note that $q_{i'}=q$, applying Proposition \ref{prop6.7} with $M:=n_2$, there exists a spectrum $\Lambda(n_1,\,n_2,\,n_2)$ of $\bigast\limits_{i\in\mathcal{B}[n_1,\,n_2]}\delta_{p_1^{-1}p_2^{-1}\cdots p_i^{-1}D_i}$ which is of the form described in Theorem \ref{th5.11} (\romannumeral4), such that
	\begin{equation*}
		\prod_{j=1}^{H-M}\Big|m_{D_{M+j}}\Big(\frac{\lambda}{p_1p_2\cdots p_{M+j}}\Big)\Big|\geqslant\Big\{\varepsilon_1\cdot\Big[\varepsilon\Big(\frac{1}{2^{2c_3}}\Big)\Big]^{10Gc_3}\Big\}^{2G},\qquad\forall\,\lambda\in \Lambda(n_1,\,n_2,\,n_2).
	\end{equation*}
	Taking $\Lambda(n_1,\,n_2):=\Lambda(n_1,\,n_2,\,n_2)$ completes the proof of the claim in this case.

	{\bf Case II}. $\{n\geqslant n_0+n_1:\,\mbox{either}~~\nu_{2}(p_{n+1})\geqslant c_3~~\mbox{or}~~\nu_{3}(p_{n+1})\geqslant c_3\}=\emptyset$.
	
	Let $n_2:=H(n_0+n_1)$, $\mathcal{S}:=\mathcal{S}[n_1,\,n_2]$ and $H:=H(n_2)$. Since $H(n_2)>n_2>H(n_1)$, Proposition \ref{prop6.3} shows that there exists $i'\in\mathcal{S}$ such that
	\begin{equation}\label{eq6.86-1}
		b_{q,i'}=H.
	\end{equation}
	Let
	\begin{equation*}
		q:=q_{i'}~~\mbox{and}~~\widetilde{q}:=\widetilde{q}_{i'}.
	\end{equation*}
	
	If $H-n_2\leqslant 2c_3$. Let
	\begin{equation}\label{eq6.81-10}
		x_i=\begin{cases}
			q^{k_{q,\mathrm{a}_t}}\theta_{q}(p_1p_2\cdots p_{n_2+2c_3}), & \mbox{if }i\in(\mathcal{B}[n_1,\,n_2]\cap X_{q}),\\
			\widetilde{q}^{k_{\widetilde{q},i}}\theta_{\widetilde{q}}(p_1p_2\cdots p_{n_2+2c_3}), & \mbox{if }i\in(\mathcal{B}[n_1,\,n_2]\cap X_{\widetilde{q}}),
		\end{cases}	
	\end{equation}
	and
	\begin{equation*}
		C_i:=\begin{cases}
			\{0,\,x_i\},  & \mbox{if } i\in X_2,\\
			\{0,\,x_i,\,-x_i\},  & \mbox{if } i\in X_3.\\
		\end{cases}
	\end{equation*}
	Since $b_{q_i,i}\leqslant H\leqslant n_2+2c_3$, we get $C_i\subset U_i$ for all $i\in\mathcal{B}[n_1,\,n_2]$. By Theorem \ref{th5.11} \rm(\romannumeral4), we see that $\sum_{i\in\mathcal{B}[n_1,\,n_2]}C_i$. is a spectrum of $\bigast\limits_{i\in\mathcal{B}[n_1,\,n_2]}\delta_{p_1^{-1}p_2^{-1}\cdots p_{i}^{-1}D_i}$. Furthermore, it is easy to verify that $x_i$, as defined by \eqref{eq6.81-10}, corresponds to setting $\mathcal{N}:=H$ and $\eta_i:=n_2+2c_3$ in \eqref{eq6.98}. It follows from Proposition \eqref{prop6.10} that
	\begin{equation*}
		\prod_{j=1}^{H-n_2}\Big|m_{D_{n_2+j}}\Big(\frac{\lambda}{p_1p_2\cdots p_{n_2+j}}\Big)\Big|\geqslant\Big[\varepsilon\Big(\frac{1}{3^{2(H-n_2+1)c_3}}\Big)\Big]^{H-n_2},\qquad\forall \,\lambda\in\Lambda.
	\end{equation*}
	Hence, when $H-n_2\leqslant 2c_3$, the proposition holds by setting $\Lambda(n_1,\,n_2):=\sum_{i\in\mathcal{B}[n_1,\,n_2]}C_i$.

	If $H-n_2>2c_3$. Define $\{\mathrm{a}_t\}_{t=1}^u$ be the integer sequence obtained from Proposition \ref{prop6.4-1} with $M:=n_2+2c_3$. Hence $\{\mathrm{a}_t\}_{t=1}^u=\mathcal{B}[n_1,\,n_2]\cap X_{q}$.
	
	We now begin the construction of the prespectrum $\Lambda$ of $\bigast\limits_{i\in\mathcal{B}[n_1,\,n_2]}\delta_{p_1^{-1}p_2^{-1}\cdots p_i^{-1}D_i}$. For any $i\in\mathcal{B}[n_1,\,n_2]$, we take
	\begin{equation}\label{eq6.85}
		x_i=\begin{cases}
			q^{k_{q,\mathrm{a}_t}}\theta_{q}(p_1p_2\cdots p_{m_t}), & \mbox{if }i=\mathrm{a}_t,\,1\leqslant t\leqslant u,\\
			\widetilde{q}^{k_{\widetilde{q},i}}\theta_{\widetilde{q}}(p_1p_2\cdots p_{n_2+2c_3}), & \mbox{if }i\in(\mathcal{B}[n_1,\,n_2]\cap X_{\widetilde{q}}),
		\end{cases}	
	\end{equation}
	and
	\begin{equation*}
		C_i:=\begin{cases}
			\{0,\,x_i\},  & \mbox{if } i\in X_2,\\
			\{0,\,x_i,\,-x_i\},  & \mbox{if } i\in X_3.\\
		\end{cases}
	\end{equation*}
	Clearly, \eqref{eq6.8} still holds under this proposition. Therefore, it follows from \eqref{eq6.8} that
	$\#\{i\in X_{\widetilde{q}}:\,n_2<i\leqslant n_2+2c_3\}>2c_3-2G$. Combining this with Lemma \ref{le5.9} \rm(\romannumeral7), we have
	\begin{equation}\label{eq6.81-11}
		\nu_{\widetilde{q}}(p_{n_2+1}p_{n_2+2}\cdots p_{n_2+2c_3})>c_3
	\end{equation}
	by noting $c_3>5G$. It follows from Lemma \ref{le5.9} \rm(\romannumeral8) that
	\begin{equation*}
		b_{\widetilde{q},i}\leqslant n_2+2c_3,\qquad\forall\,i\in(\mathcal{B}[n_1,\,n_2]\cap X_{\widetilde{q}}).
	\end{equation*}
	Moreover, by Proposition \ref{prop6.4-1} \rm(\romannumeral1), we get $m_t\geqslant b_{q,\mathrm{a}_t}$ for all $1\leqslant t\leqslant u$. Hence, $C_i\subset U_i$ for all $i\in\mathcal{B}[n_1,\,n_2]$. By Theorem \ref{th5.11} \rm(\romannumeral4), we see that $\sum_{i\in\mathcal{B}[n_1,\,n_2]}C_i$ is a spectrum of $\bigast\limits_{i\in\mathcal{B}[n_1,\,n_2]}\delta_{p_1^{-1}p_2^{-1}\cdots p_{i}^{-1}D_i}$.
	
	We can readily check that $x_i$ defined by \eqref{eq6.81-10} is obtained by letting $\mathcal{N}:=n_2+2c_3$, $\eta_{\mathrm{a}_t}:=m_t$ if $1\leqslant t\leqslant u$ and $\eta_i:=n_2+2c_3$ if $i\in(\mathcal{B}[n_1,\,n_2]\cap X_{\widetilde{q}})$ in \eqref{eq6.98}. It follows from Proposition \eqref{prop6.10} that
	\begin{equation}\label{eq6.81-12}
		\prod_{j=1}^{2c_3}\Big|m_{D_{n_2+j}}\Big(\frac{\lambda}{p_1p_2\cdots p_{n_2+j}}\Big)\Big|\geqslant\Big[\varepsilon\Big(\frac{1}{3^{2(2c_3+1)c_3}}\Big)\Big]^{2c_3},\qquad\forall \,\lambda\in\sum_{i\in\mathcal{B}[n_1,\,n_2]}C_i.
	\end{equation}

	On the other hand, we have
	\begin{equation*}
		\begin{split}
			\nu_{\widetilde{q}}(p_{1}p_{2}\cdots p_{n_2+2c_3})-k_{\widetilde{q},i}=\nu_{\widetilde{q}}(p_{i+1}p_{i+2}\cdots p_{n_2+2c_3})+\nu_{\widetilde{q}}(\widetilde{q}d_i)
			\geqslant\nu_{\widetilde{q}}(p_{n_2+1}p_{n_2+2}\cdots p_{n_2+2c_3})+1
		\end{split}
	\end{equation*}
	for all $i\leqslant n_2$ with $i\in X_{\widetilde{q}}$. By \eqref{eq6.81-11}, we have
	\begin{equation*}
		\nu_{\widetilde{q}}(p_{1}p_{2}\cdots p_{n_2+2c_3})-k_{\widetilde{q},i}>c_3+1,\qquad\forall\,i\in(\mathcal{B}[n_1,\,n_2]\cap X_{\widetilde{q}}).
	\end{equation*}
	From Lemma \ref{le5.14} and \eqref{eq6.85}, we get
	\begin{equation*}
		\Big|\sum_{i\in(\mathcal{B}[n_1,\,n_2]\cap X_{\widetilde{q}})}\frac{x_i}{p_1p_2\cdots p_{n_2+2c_3}}\Big|\leqslant
		\frac{\widetilde{q}}{\widetilde{q}-1}\cdot \widetilde{q}^{k_{\widetilde{q},i_1}-\nu_{\widetilde{q}}(p_1p_2\cdots p_{n_2+2c_3})},
	\end{equation*}
	where $i_1\in(\mathcal{B}[n_1,\,n_2]\cap X_{\widetilde{q}})$ satisfies $k_{\widetilde{q},i_1}=\max\{k_{\widetilde{q},i}:\,i\in(\mathcal{B}[n_1,\,n_2]\cap X_{\widetilde{q}})\}$. Combining the above two inequalities, we obtain $\big|\sum_{i\in(\mathcal{B}[n_1,\,n_2]\cap X_{\widetilde{q}})}\frac{x_i}{p_1p_2\cdots p_{n_2+2c_3}}\big|\leqslant\widetilde{q}^{-c_3}\leqslant\frac{1}{2^{c_3}}<\frac{1}{D\cdot3^{5G}}$. This means that
	\begin{equation*}
		\sum_{i\in(\mathcal{B}[n_1,\,n_2]\cap X_{\widetilde{q}})}C_i\subset\widetilde{\Lambda}\Big[n_2+2c_3;~\frac{1}{D\cdot q^{3G}}\Big].
	\end{equation*}
	Applying Proposition \ref{prop6.7} with $M:=n_2+2c_3$, we have
	\begin{equation*}
		\prod_{j=1}^{H-(n_2+2c_3)}\Big|m_{D_{n_2+2c_3+j}}\Big(\frac{\lambda}{p_1p_2\cdots p_{n_2+2c_3+j}}\Big)\Big|\geqslant\Big\{\varepsilon_1\cdot\Big[\varepsilon\Big(\frac{1}{2^{2c_3}}\Big)\Big]^{10Gc_3}\Big\}^{2G}
	\end{equation*}
	for all $\lambda\in\sum_{i\in(\mathcal{B}[n_1,\,n_2]\cap X_{q})}C_i+\sum_{i\in(\mathcal{B}[n_1,\,n_2]\cap X_{\widetilde{q}})}C_i$. Combining the above with \eqref{eq6.81-12}, we have
	\begin{equation*}
		\prod_{j=1}^{H-n_2}\Big|m_{D_{n_2+j}}\Big(\frac{\lambda}{p_1p_2\cdots p_{n_2+j}}\Big)\Big|\geqslant\Big[\varepsilon\Big(\frac{1}{3^{2(2c_3+1)c_3}}\Big)\Big]^{2c_3}\cdot\Big\{\varepsilon_1\cdot\Big[\varepsilon\Big(\frac{1}{2^{2c_3}}\Big)\Big]^{10Gc_3}\Big\}^{2G},\qquad\forall \,\lambda\in\sum_{i\in\mathcal{B}[n_1,\,n_2]}C_i.
	\end{equation*}
	Since $\varepsilon(\cdot)$ is non-decreasing, we have $\big\{\varepsilon_1\cdot\big[\varepsilon\big(\frac{1}{2^{2c_3}}\big)\big]^{10Gc_3}\big\}^{2G}\geqslant\varepsilon_1^{2G}\cdot\big[\varepsilon\big(\frac{1}{3^{2(2c_3+1)c_3}}\big)\big]^{20G^2c_3}$. It follows that
	\begin{equation*}
		\prod_{j=1}^{H-n_2}\Big|m_{D_{n_2+j}}\Big(\frac{\lambda}{p_1p_2\cdots p_{n_2+j}}\Big)\Big|\geqslant\varepsilon_1^{2G}\cdot\big[\varepsilon\big(\frac{1}{3^{2(2c_3+1)c_3}}\big)\big]^{21G^2c_3},\qquad\forall \,\lambda\in\sum_{i\in\mathcal{B}[n_1,\,n_2]}C_i.
	\end{equation*}
	Hence, the claim holds by setting $\Lambda(n_1,\,n_2):=\sum_{i\in\mathcal{B}[n_1,\,n_2]}C_i$.

	Note that
	\begin{equation*}
		\{\delta_{p_{n+1}^{-1}D_{n+1}}\bigast\delta_{p_{n+1}^{-1}p_{n+2}^{-1}D_{n+2}}\bigast\cdots\bigast\delta_{p_{n+1}^{-1}p_{n+2}^{-1}\cdots p_{n+m}^{-1}D_{n+m}}:\,n>0,\,m>0\}
	\end{equation*}
	is a family of probability measures supported in subsets of $[0,\,1]$. Hence, their Fourier transformations are equi-continuous. It follows from \eqref{eq6.91-1} that for  $\varepsilon_1^{2G}\cdot\big[\varepsilon\big(\frac{1}{3^{2(2c_3+1)c_3}}\big)\big]^{21G^2c_3}$, there exists $\omega_3>0$ such that for any $\lambda\in\Lambda(n_1,\,n_2)$ and $y\in[-\omega_3,\,\omega_3]$,
	\begin{equation*}
		\prod_{j=1}^{H-n_2}\Big|m_{p_{n+1}^{-1}p_{n+2}^{-1}\cdots p_{n+j}^{-1}D_{n_2+j}}\Big(y+\frac{\lambda}{p_1p_2\cdots p_{n_2}}\Big)\Big|\geqslant\frac{1}{2}\varepsilon_1^{2G}\cdot\big[\varepsilon\big(\frac{1}{3^{2(2c_3+1)c_3}}\big)\big]^{21G^2c_3}.
	\end{equation*}
	Finally, we complete the proof of the proposition by taking $\varepsilon_3:=\frac{1}{2}\varepsilon_1^{2G}\cdot\big[\varepsilon\big(\frac{1}{3^{2(2c_3+1)c_3}}\big)\big]^{21G^2c_3}$.
\end{pf}

\begin{prop}\label{prop7.13}
Assume that the pair $(\mathbb{P},\,\mathbb{D})$ satisfies \eqref{eq1.4}, \eqref{eq1.5} and \eqref{eq2.12-3}. There are small positive constants $\varepsilon_4$ and $\omega_4$ such that for any integers $n_0\geqslant1$ and $n_1\geqslant1$, there exist an integer $n_2>H(n_0+n_1)$ and a spectrum $\Lambda(n_1,\,n_2)$ of $\bigast\limits_{i=n_1+1}^{n_2}\delta_{p_1^{-1}p_2^{-1}\cdots p_i^{-1}D_i}$ which is of the form described in Theorem \ref{th5.11} (\romannumeral4), such that
\begin{equation*}
	\Big|\widehat{\mu}_{>n_2}\Big(y+\frac{\lambda}{p_1p_2\cdots p_{n_2}}\Big)\Big|\geqslant\varepsilon_4,\qquad\forall \, y\in[-\omega_4,\,\omega_4],\quad\lambda\in\Lambda.
\end{equation*}	
\end{prop}

\begin{pf}
If the set $\{n\geqslant1:\,b_{q_j,j}\leqslant b_{q_n,n}~\mbox{for~all}~j\leqslant b_{q_n,n}\}$ is infinite, there exists $n_2\in\{n\geqslant1:\,b_{q_j,j}\leqslant b_{q_n,n}~\mbox{for~all}~j\leqslant b_{q_n,n}\}$ such that $n_2>H(n_0+n_1)$. It is clear $H(n_2)=n_2$. Let $\Lambda$ be a spectrum of $\bigast\limits_{i=n_1+1}^{n_2}\delta_{p_{1}^{-1}p_{2}^{-1}\cdots p_{i}^{-1}D_{i}}$ as described in Theorem \ref{th5.11} \rm(\romannumeral3). By Proposition \ref{prop5.16}, we obtain that for all $\lambda\in\Lambda$, one can find $z_{\lambda}\in\mathbb{Z}$ with $z_{0}=0$ so that
\begin{equation*}
	\Big|\hat\mu_{>n_2}\Big(y+\frac{\lambda}{p_1p_2\cdots p_{n_2}}+z_{\lambda}\Big)\Big|>\frac{1}{2}\varepsilon_2,\qquad\forall \, y\in[-\omega_2,\,\omega_2],\quad\lambda\in\Lambda.
\end{equation*}
Moreover, Proposition \ref{prop5.16} shows that the set $\{\lambda+p_1p_2\cdots p_{n_2}z_{\lambda}:\,\lambda\in\Lambda\}\subset\sum_{i=n_1+1}^{n_2}U_i$ is also a spectrum of $\bigast\limits_{i=n_1+1}^{n_2}\delta_{p_1^{-1}p_2^{-1}\cdots p_i^{-1}D_i}$. In this case, we complete the proof of the proposition by taking $\varepsilon_4:=\frac{1}{2}\varepsilon_2$, $\omega_4:=\omega_2$ and $\Lambda(n_1,\,n_2):=\{\lambda+p_1p_2\cdots p_{n_2}z_{\lambda}:\,\lambda\in\Lambda\}$.

If the set $\{n\geqslant1:\,b_{q_j,j}\leqslant b_{q_n,n}~\mbox{for~all}~j\leqslant b_{q_n,n}\}$ is finite, then {\bfseries (C6)} holds. By Proposition \ref{prop6.9}, one can find an integer $n_2>H(n_0+n_1)$ and a sequence of integer subsets $\{C_i\}_{i\in\mathcal{B}[n_1,\,n_2]}$ with $C_i\subset(\{0\}\cup U_i)$ such that $\sum_{i\in\mathcal{B}[n_1,\,n_2]}C_i$ is a spectrum of $\bigast\limits_{i\in\mathcal{B}[n_1,\,n_2]}\delta_{p_1^{-1}p_2^{-1}\cdots p_i^{-1}D_i}$ and
\begin{equation}\label{eq5.91}
	\prod_{j=1}^{H-n_2}\Big|m_{D_{n_2+j}}\Big(y+\frac{\lambda}{p_1p_2\cdots p_{n_2+j}}\Big)\Big|\geqslant\varepsilon_3,\qquad\forall\,y\in[-\omega_3,\,\omega_3] ,\quad\,\lambda\in\sum_{i\in\mathcal{B}[n_1,\,n_2]}C_i,
\end{equation}
where $H:=H(n_2)$.

We now begin the construction of a prespectrum $\Lambda$ of $\bigast\limits_{i=n_1+1}^{n_2}\delta_{p_1^{-1}p_2^{-1}\cdots p_{\mathrm{a}_t}^{-1}D_{\mathrm{a}_t}}$. For any $i\in(\mathcal{S}[n_1;\,n_2]\setminus\mathcal{B}[n_1,\,n_2])$, we take
\begin{equation}
	x_i=\begin{cases}
		q^{k_{q,i}}\theta_{q}(p_1p_2\cdots p_{b_{q,i}}), & i\in[(\mathcal{S}[n_1;\,n_2]\setminus\mathcal{B}[n_1,\,n_2])\cap X_{q}],\\
		\widetilde{q}^{k_{\widetilde{q},i}}\theta_{\widetilde{q}}(p_1p_2\cdots p_{b_{\widetilde{q},i}}), & i\in[(\mathcal{S}[n_1;\,n_2]\setminus\mathcal{B}[n_1,\,n_2])\cap X_{\widetilde{q}}],
	\end{cases}	
\end{equation}
and
\begin{equation*}
	C_i:=\begin{cases}
		\{0,\,x_i\},  & \mbox{if } i\in X_2,\\
		\{0,\,x_i,\,-x_i\},  & \mbox{if } i\in X_3.\\
	\end{cases}
\end{equation*}
It follows from Theorem \ref{th5.11} \rm(\romannumeral3) that $\sum_{i\in\mathcal{S}[n_1;\,n_2]}C_i$ is a spectrum of $\bigast\limits_{i=n_1+1}^{n_2}\delta_{p_1^{-1}p_2^{-1}\cdots p_i^{-1}D_i}$.

Choose $\lambda\in\sum_{i\in\mathcal{S}[n_1;\,n_2]}C_i$, there exists an integer sequence $\{c_i\}_{i\in\mathcal{S}[n_1;\,n_2]}$ such that $\lambda=\sum_{i\in\mathcal{S}[n_1;\,n_2]}c_i$. If there exists $i\in(\mathcal{S}[n_1;\,n_2]\setminus\mathcal{B}[n_1,\,n_2])$ such that $c_i\neq0$, it follows from Proposition \ref{prop5.16} that there exists $h_{\frac{\lambda}{p_1p_2\cdots p_{n_2}},\,n_2}\in\mathbb{Z}$ with $h_{0,\,n_2}:=0$ so that
\begin{equation}\label{eq5.93}
	\Big|\hat\mu_{>n_2}\Big(y+\frac{\lambda}{p_1p_2\cdots p_{n_2}}+h_{\frac{\lambda}{p_1p_2\cdots p_{n_2}},\,n_2}\Big)\Big|\geqslant\frac{1}{2}\varepsilon_2,\qquad\forall \, y\in[-\omega_2,\,\omega_2].
\end{equation}
Else $c_i=0$ for all $i\in(\mathcal{S}[n_1;\,n_2]\setminus\mathcal{B}[n_1,\,n_2])$, applying Proposition \ref{prop5.16} again, we obtain that  there exists $h_{\frac{\lambda}{p_1p_2\cdots p_H},\,H}\in\mathbb{Z}$ with $h_{0,\,H}:=0$ such that
\begin{equation}\label{eq5.96}
	\Big|\hat\mu_{>H}\Big(y+\frac{\lambda}{p_1p_2\cdots p_H}+h_{\frac{\lambda}{p_1p_2\cdots p_H},\,H}\Big)\Big|\geqslant\frac{1}{2}\varepsilon_2,\qquad\forall \, y\in[-\omega_2,\,\omega_2].
\end{equation}
Let $\omega_4:=\min\{\omega_2,\,\omega_3\}$. Note that $c_i=0$ for all $i\in(\mathcal{S}[n_1;\,n_2]\setminus\mathcal{B}[n_1,\,n_2])$, we get $\lambda\in\sum_{i\in\mathcal{B}[n_1,\,n_2]}C_i$. Hence, by \eqref{eq5.91} and \eqref{eq5.96}, we get that for any $y\in[-\omega_4,\,\omega_4]$,
\begin{equation}\label{eq5.92}
	\begin{split}
		&\Big|\hat\mu_{>n_2}\Big(y+\frac{\lambda+p_1p_2\cdots p_Hh_{\frac{\lambda}{p_1p_2\cdots p_H},\,H}}{p_1p_2\cdots p_{n_2}}\Big)\Big|\\
		=&\prod_{j=1}^{H-n_2}\Big|m_{D_{n_2+j}}\Big(\frac{y}{p_{n_2+1}\cdots p_{n_2+j}}+\frac{\lambda}{p_1p_2\cdots p_{n_2+j}}\Big)\Big|\cdot\Big|\hat\mu_{>H}\Big(\frac{y}{p_{n_2+1}\cdots p_H}+\frac{\lambda}{p_1p_2\cdots p_H}+h_{\frac{\lambda}{p_1p_2\cdots p_H},\,H}\Big)\Big|\\
		\geqslant&\varepsilon_3\cdot\frac{1}{2}\varepsilon_2,\raisetag{-2em}
	\end{split}
\end{equation}
where the first equality holds because $m_{d_n^{-1}D_n}$ is an integral periodic function for all $n\geqslant1$.

For any $\lambda\in\sum_{i\in\mathcal{S}[n_1;\,n_2]}C_i$, let
\begin{equation*}
	z_{\lambda}:=\left\{\begin{array}{ll}
		p_1p_2\cdots p_{n_2}h_{\frac{\lambda}{p_1p_2\cdots p_{n_2}},\,n_2},&\mbox{if}~~\{i\in(\mathcal{S}[n_1;\,n_2]\setminus\mathcal{B}[n_1,\,n_2]):\,c_i\neq0\}\neq\emptyset,\\
		p_1p_2\cdots p_{H}h_{\frac{\lambda}{p_1p_2\cdots p_{H}},\,H},&\mbox{if}~~\{i\in(\mathcal{S}[n_1;\,n_2]\setminus\mathcal{B}[n_1,\,n_2]):\,c_i\neq0\}=\emptyset,
	\end{array}\right.
\end{equation*}
where $\{c_i\}_{i\in\mathcal{S}[n_1;\,n_2]}$ is a sequence of integers such that $\lambda=\sum_{i\in\mathcal{S}[n_1;\,n_2]}c_i$. By the arbitrariness of $\lambda$, it follows from \eqref{eq5.93} and \eqref{eq5.93} that
\begin{equation*}
	\Big|\hat\mu_{>n_2}\Big(y+\frac{\lambda+z_{\lambda}}{p_1p_2\cdots p_{n_2}}\Big)\Big|\geqslant\frac{1}{2}\varepsilon_2\varepsilon_3,\qquad\forall \, y\in[-\omega_4,\,\omega_4],\,\,\lambda\in\sum_{i\in\mathcal{S}[n_1;\,n_2]}C_i.
\end{equation*}
Moreover, define
$j_{\lambda}:=\min\{i\in\mathcal{S}[n_1;\,n_2]:\,c_i\neq0\}$ and $n_{\lambda}:=\max(\{b_{q_{j_{\lambda}},j_{\lambda}}\}\cup\mathcal{S}[n_1;\,n_2])$. For any $\lambda\in\sum_{i\in\mathcal{S}[n_1;\,n_2]}C_i$, it is clear that $n_{\lambda}=n_2$ if $\{i\in(\mathcal{S}[n_1;\,n_2]\setminus\mathcal{B}[n_1,\,n_2]):\,c_i\neq0\}\neq\emptyset$ and $n_{\lambda}\leqslant H$ if $\{i\in(\mathcal{S}[n_1;\,n_2]\setminus\mathcal{B}[n_1,\,n_2]):\,c_i\neq0\}=\emptyset$. Hence, we get
\begin{equation*}
	z_{\lambda}\in p_1p_2\cdots p_{n_{\lambda}}\mathbb{Z},\qquad\forall \,\lambda\in\sum_{i\in\mathcal{S}[n_1;\,n_2]}C_i.
\end{equation*}
By Theorem \ref{th5.11} \rm(\romannumeral4), we get $\{\lambda+z_{\lambda}:\,\lambda\in\sum_{i\in\mathcal{S}[n_1;\,n_2]}C_i\}\subset\sum_{i\in\mathcal{S}[n_1;\,n_2]}(\{0\}\cup U_i$ is also a spectrum of $\bigast\limits_{i=n_1+1}^{n_2}\delta_{p_1^{-1}p_2^{-1}\cdots p_i^{-1}D_i}$.

Finally, we complete the proof of the proposition by taking $\varepsilon_4:=\frac{1}{2}\varepsilon_2\varepsilon_3$, $\omega_4:=\min\{\omega_2,\,\omega_3\}$ and $\Lambda(n_1;\,n_2):=\{\lambda+z_{\lambda}:\,\lambda\in\sum_{i\in\mathcal{S}[n_1;\,n_2]}C_i\}$.	
\end{pf}

\hspace{-0.05em}{\bfseries\large Proof of the sufficiency of Theorem \ref{th1.3}}.

Take $n_1>1$, let $\Lambda_1$ be a spectrum of $\bigast\limits_{n=1}^{n_1}\chi_n$ as described in Theorem \ref{th5.11} \rm(\romannumeral3). It is clear that there exists an integer $m_1>1$ such that
\begin{equation*}
	(p_1p_2\cdots p_{n_1+m_1})^{-1}\Lambda_1\subset\big[-\frac{\omega_4}{2},\,\frac{\omega_4}{2}\big].
\end{equation*}
It follows from Proposition \ref{prop7.13} that for integers $n_1$ and $m_1$, there exists an integer $n_2>H(n_1+m_1)$ and a spectrum $\Lambda_{1,2}\subset\sum_{n=n_1+1}^{n_2}(\{0\}\cup U_i)$ of $\bigast\limits_{n=n_1+1}^{n_2}\chi_n$ such that
\begin{equation}\label{eq7.172}
	\Big|\widehat{\mu}_{>n_2}\Big(\frac{\lambda'}{p_1p_2\cdots p_{n_2}}+\frac{\lambda}{p_1p_2\cdots p_{n_2}}\Big)\Big|\geqslant\varepsilon_4,\qquad\forall \, \lambda'\in\Lambda_1,\quad\lambda\in\Lambda_{1,2}.
\end{equation}
By Theorem \ref{th5.11} \rm(\romannumeral3), we see that $\Lambda_2:=\Lambda_1+\Lambda_{1,2}$ is a spectrum of the probability measure $\bigast\limits_{n=1}^{n_2}\chi_n$. Since $0\in\Lambda_1$ and $0\in\Lambda_{1,2}$, the definitions of $\Lambda_1$ and $\Lambda_2$ shows $0
\in\Lambda_1\subset\Lambda_2\subset\sum_{i=1}^{n_2}(\{0\}\cup U_i\})$. By \eqref{eq7.172}, we have
\begin{equation}\label{eq5.423}
	\Big|\widehat{\mu}_{>n_2}\Big(\frac{\lambda}{p_1p_2\cdots p_{n_2}}\Big)\Big|\geqslant\varepsilon_4,\qquad\forall \, \lambda\in\Lambda_2.
\end{equation}

Continuing in this way, we can find a strictly increasing sequence $\{n_k\}_{k=1}^{\infty}\subset\mathbb{N}$ and $\Lambda_k$ such that the follows properties.
\begin{equation}\label{eq7.174}
	0\in\Lambda_k\subset\Lambda_{k+1}\subset\sum_{i=1}^{n_{k+1}}(\{0\}\cup U_i\}),\qquad k=1,\,2,\,\cdots,
\end{equation}
\begin{equation}\label{eq7.175}
	(p_1p_2\cdots p_{n_{k+1}})^{-1}\Lambda_k\subset\Big[-\frac{\omega_4}{2^{k+1}},\,\frac{\omega_4}{2^{k+1}}\Big],\qquad k=1,\,2,\,\cdots,
\end{equation}
\begin{equation}\label{eq7.176}
	\Big|\widehat{\mu}_{>n_k}\Big(\frac{\lambda}{p_1p_2\cdots p_{n_k}}\Big)\Big|\geqslant\varepsilon_4,\qquad\forall \, \lambda\in\Lambda_k,\qquad k=2,\,3,\,\cdots,
\end{equation}
and $\Lambda_k$ is a spectrum of $\bigast\limits_{n=1}^{n_k}\chi_n$ for all $k=1,\,2,\,\cdots$.

Let $\Gamma:=\cup_{k=1}^{\infty}\Lambda_k$. We shall prove $\Gamma$ is a spectrum of $\mu_{\mathbb{P},\,\mathbb{D}}$.

For any $a\ne b\in\Gamma$, from \eqref{eq7.174} it follows that $a\ne b\in\Lambda_k$ for some $k>0$.  Hence, $a-b$ is a zero point of the Fourier transform of $\delta_{p_1^{-1}D_1}\bigast\delta_{p_1^{-1}p_2^{-1}D_2}\bigast\cdots\bigast\delta_{p_1^{-1}p_2^{-1}\cdots p_{n_k}^{-1}D_{n_k}}$. Hence, $\widehat{\mu_{\mathbb{P},\,\mathbb{D}}}(a-b)=0$, which yields the exponential function set $E_{\Gamma}:=\{e^{2\pi i\gamma x}:\,\gamma\in\Gamma\}$ is an orthogonal family of $L^2(\mu_{\mathbb{P},\,\mathbb{D}})$.

Suppose on the contrary that $\Gamma$ is not a spectrum of $\mu_{\mathbb{P},\,\mathbb{D}}$. Then, Lemma \ref{le2.2} shows that $Q_{\Gamma}(x_0)<1$ for some $x_0\in\mathbb{R}$.

Recall that $\lim_{k\to\infty}(p_1p_2\cdots p_{n_{k+1}})^{-1}x_0=0$ and $\widehat{\Phi}:=\{\widehat{\mu}_{>n}:\,n\geqslant1\}$ is equi-continuous. By \eqref{eq7.175}, we get
\begin{equation}\label{eq7.177}
	\sigma_k:=\inf_{\lambda\in\Lambda_k}|\widehat{\mu}_{>n_{k+1}}((p_1p_2\cdots p_{n_{k+1}})^{-1}(\lambda+x_0))|\to1~\mbox{as}~k\to\infty.
\end{equation}
Moreover, \eqref{eq7.176} shows that there exists a positive integer $k_0>0$ such that for any $k>k_0$ and $\lambda\in\Lambda_k$, we have
\begin{equation}\label{eq7.178}
	|\widehat{\mu}_{>n_k}((p_1p_2\cdots p_{n_{k}})^{-1}(\lambda+x_0))|\geqslant\frac{1}{2}\varepsilon_4.
\end{equation}

Let
\begin{equation*}
	Q_k(x_0):=\sum_{\lambda\in\Lambda_k}|\widehat{\mu_{\mathbb{P},\,\mathbb{D}}}(\lambda+x_0)|^2,\qquad k=1,\,2,\,\cdots.
\end{equation*}
For $k>k_0$, it follows from \eqref{eq2.3} obtain 
\begin{equation*}
	\begin{split}
		&Q_{k+1}(x_0)-Q_k(x_0)\\
		=&\sum_{\lambda\in(\Lambda_{k+1}\setminus\Lambda_k)}\prod_{n=1}^{\infty}|m_{D_n}((p_1p_2\cdots p_n)^{-1}(\lambda+x_0))|^2\\
		=&\sum_{\lambda\in(\Lambda_{k+1}\setminus\Lambda_k)}\prod_{n=1}^{n_{k+1}}|m_{D_n}((p_1p_2\cdots p_n)^{-1}(\lambda+x_0))|^2|\widehat{\mu}_{>n_{k+1}}((p_1p_2\cdots p_{n_{k+1}})^{-1}(\lambda+x_0))|^2.
	\end{split}
\end{equation*}
Hence \eqref{eq7.178} shows that for any $k>k_0$, we have
\begin{equation}\label{eq7.179}
	Q_{k+1}(x_0)-Q_k(x_0)\geqslant\frac{1}{4}\varepsilon_4^2\sum_{\lambda\in(\Lambda_{k+1}\setminus\Lambda_k)}\prod_{n=1}^{n_{k+1}}|m_{D_n}((p_1p_2\cdots p_n)^{-1}(\lambda+x_0))|^2.
\end{equation}
Since $\Lambda_{k+1}$ is a spectrum of $\bigast\limits_{n=1}^{n_{k+1}}\chi_n$, it follows from Lemma \ref{le2.2} that
\begin{equation*}
	\sum_{\lambda\in\Lambda_{k+1}}\prod_{n=1}^{n_{k+1}}|m_{D_n}((p_1p_2\cdots p_n)^{-1}(\lambda+x_0))|^2=1,\qquad k=1,\,2,\,\cdots.
\end{equation*}
Combining the above with \eqref{eq7.179}, we have
\begin{equation*}
	Q_{k+1}(x_0)-Q_k(x_0)\geqslant\frac{1}{4}\varepsilon_4^2\Big(1-\sum_{\lambda\in\Lambda_k}\prod_{n=1}^{n_{k+1}}|m_{D_n}((p_1p_2\cdots p_n)^{-1}(\lambda+x_0))|^2\Big),\qquad\forall\,k>k_0.
\end{equation*}
On the other hand, the definition of $\sigma_k$ in \eqref{eq7.177} shows that for $k\geqslant1$, we have
\begin{equation*}
	Q_k(x_0)=\sum_{\lambda\in\Lambda_k}\prod_{n=1}^{\infty}|m_{D_n}((p_1p_2\cdots p_n)^{-1}(\lambda+x_0))|^2
	\geqslant\sigma_k^2\sum_{\lambda\in\Lambda_k}\prod_{n=1}^{n_{k+1}}|m_{D_n}((p_1p_2\cdots p_n)^{-1}(\lambda+x_0))|^2.
\end{equation*}
From the two inequalities above, it follows that 
\begin{equation*}
	Q_{k+1}(x_0)-Q_k(x_0)\geqslant\frac{1}{4}\varepsilon_4^2\Big(1-\sigma_k^{-2}Q_k(x_0)\Big),\qquad\forall \, k>k_0.
\end{equation*}
Thus, the limit behavior of $\sigma_k$ in \eqref{eq7.177}, together with the assumption $Q_{\Gamma}(x_0)<1$, gives 
\begin{equation*}
	\liminf_{k\to\infty}(Q_{k+1}(x_0)-Q_k(x_0))\geqslant\frac{1}{4}\varepsilon_4^2\big(1-\lim_{k\to\infty}\sigma_k^{-2}Q_k(x_0)\big)=\frac{1}{4}\varepsilon_4^2\big(1-Q_{\Gamma}(x_0)\big)>0.
\end{equation*}
Combining the above with \eqref{eq7.174}, we get
\begin{equation*}
	1>Q_{\Gamma}(x_0)=\lim_{k\to\infty}Q_k(x_0)\geqslant\sum_{k=1}^{\infty}(Q_{k+1}(x_0)-Q_k(x_0))=\infty,
\end{equation*}
a contradiction. Hence, $\Gamma$ is a spectrum of $\mu_{\mathbb{P},\,\mathbb{D}}$. We have thus proved the sufficiency of Theorem  \ref{th1.3}.

\small

\bigskip

\indent  School of Mathematics and Statistics \& Key Laboratory of Analytical Mathematics and Applications (Ministry of Education) \& Fujian Provincial Key Laboratory of Statistics and Artificial Intelligence \& Fujian Key Laboratory of Analytical Mathematics and Applications (FJKLAMA) \& Center for Applied Mathematics of Fujian Province (FJNU), Fujian Normal University, 350117 Fuzhou, P.R.China.

\medskip

E-mail: caoys8869@126.com (Y.-S. Cao), dengfractal@126.com (Q.-R. Deng), limtwd@fjnu.edu.cn (M.-T. Li).

\end{document}